\documentclass[12pt,twoside]{article}

\usepackage{graphicx}
\usepackage{xcolor}
\usepackage{styleset}
\usepackage{macros}
\let\phi\varphi

\let\subsubsection\subparagraph

\title  {Relations in singular instanton homology}

\author {P. B. Kronheimer and T. S. Mrowka%
	\thanks{%
		The work of the first author was supported by the National
		Science Foundation through NSF grants DMS-1707924 and
		DMS-2005310. The work of the second author was supported by
		NSF grants DMS-1808794 and DMS-2105512. Both authors were
                supported by a Simons Foundation Award \#994330 (Simons
                Collaboration on New Structures in Low-Dimensional
                Topology). This paper was completed while 
                the second author was in residence at the Simons
                Laufer Mathematical Sciences Institute as a  Clay
                Senior Scholar and supported by NSF grant DMS-1928930.}}

\address {Harvard University, Cambridge MA 02138 \\
	Massachusetts Institute of Technology, Cambridge MA 02139}

\begin{document}
	
	\maketitle
	
	\begin{abstract}
		We calculate the singular instanton homology with local
        coefficients for the simplest $n$-strand braids in
        $S^{1}\times S^{2}$ for all odd $n$,
        describing these homology groups and their module structures in terms of
        the coordinate rings of explicit algebraic curves. The
        calculation is expected to be equivalent to computing the quantum cohomology
        ring of a certain Fano variety, namely a moduli space of stable
        parabolic bundles on a sphere with $n$ marked points.
	\end{abstract}
	 
	\tableofcontents
	
	\section{Introduction}
        \subsection{Background}

        A pair $(Y,K)$, consisting of a closed, oriented 3-manifold
        and an embedded link, gives rise to a 3-dimensional orbifold
        $Z=Z(Y,K)$ whose underlying topology is that of $Y$ and whose
        singular consists of the locus $K$ where the orbifold
        structure has local stabilizers of order $2$. The pair
        $(Y,K)$, or the orbifold $Z$, is \emph{admissible} if $[K]$
        has odd pairing with some integer homology class. To an
        admissible orbifold $Z$, there is associated its
        \emph{singular instanton homology} \cite{KM-yaft},
        constructed from the Morse theory of the Chern-Simons
        functional on the space of $\SO(3)$ orbifold connections
        modulo a determinant-1 gauge group. With rational
        coefficients, we denote the singular instanton homology by
        $I(Z;\Q)$.

        A deformation of this instanton homology is
        described in \cite{KM-deformation}. It
        can be viewed as an instanton homology group with values in a local
        coefficient system on the space of connections modulo gauge,
        and it appears in this paper as $I(Z;
        \Gamma)$, where $\Gamma$ denotes a local system of free rank-1
        modules over the ring of Laurent polynomials \[ \cR =
        \Q[\tau^{\pm 1}]. \] (See section~\ref{sec:local-system}.)

        A choice of a 2-dimensional homology class in $Z$ gives rise
        to an operator $\alpha$, on both $I(Z;\Q)$ and $I(Z;\Gamma)$.
        For each choice of basepoint $p\in K$, there is also an
        operator $\delta_{p}$, depending on the connected
        component of $K$ on which $p$ lies and a choice of local
        orientation at $p$. These operators commute,
        and make $I(Z;\Q)$ and $I(Z;\Gamma)$ into modules over the
        rings $\Q[\alpha,\delta_{1},\dots,\delta_{n}]$ and
        $\cR[\alpha,\delta_{1},\dots,\delta_{n}]$ respectively.

        In \cite{Street}, Street completely described the instanton
        homology $I(Z;\Q)$ and its module structure in the case that
        $Z$ is the product
        \[
                    Z_{n}  = S^{1} \times S^{2}_{n}.
        \]
        Here $S^{2}_{n}$ denotes the 2-sphere with $n$ orbifold
        points. An extension of Street's result to the case of
        $S^{1}\times \Sigma_{g,n}$ was obtained by Xie and Zhang
        \cite{Xie-Zhang}, and an earlier model for both of these
        calculations is the work of Mu\~noz \cite{Munoz1,Munoz2} on
        the case of $S^{1}\times\Sigma_{g}$ (where the orbifold locus
        is empty).

        The purpose of this paper is to extend Street's
        calculation to the case of instanton homology with local
        coefficients $\Gamma$. Alongside $Z_{n}$, a closely related
        calculation is for the instanton homology of an orbifold we
        call
        $Z_{n,1}$. If the $n$ orbifold points in $S^{2}_{n}$ are
        arranged symmetrically around a circle, then a rotation $h$
        through $2\pi/n$ is an automorphism of $S^{2}_{n}$ which
        permutes the orbifold points, and we write $Z_{n,1}$ for its
        mapping torus:
        \[
        \begin{gathered}
        Z_{n,1} = M_{h}\\
        h : S^{2}_{n} \to S^{2}_{n}.
        \end{gathered}
        \]
        Since the orbifold locus in $Z_{n,1}$ is connected, there is
        only one operator $\delta=\delta_{p}$ in this case, and
        $I(Z_{n,1};\Gamma)$ is  a module for an algebra
        $\cR[\alpha,\delta]$ where $\cR$ is again a ring of Laurent
        polynomials. We can summarize the main theme of this paper as
        the solution to the following.

        \begin{problem}[$\star$]
            Describe $I(Z_{n};\Gamma)$ and\/ $I(Z_{n,1};\Gamma)$
            explicitly as modules for the algebras
            $\cR[\alpha,\delta_{1},\dots,\delta_{n}]$ and
            $\cR[\alpha,\delta]$ respectively.
        \end{problem}

        The motivation for studying this question came from a desire
        to calculate a variant of the singular instanton homology of torus knots,
        $\Inat(T_{n,q}; \Gamma)$, as studied in \cite{KM-InstConc},
        and the related knot concordance invariants of these. In
        \cite{KM-InstConc}, the base ring always had characteristic
        $2$, as necessitated by the construction there. An alternative
        formulation allows characteristic $0$, and the results of this
        paper are a main step. We return to this discussion briefly in
        section~\ref{sec:further}.
        
        \subsection{Statement of the result}
        
        We shall give a complete answer to $(\star)$, and 
        to give a flavor of the result here, we describe
        $I(Z_{n,1};\Gamma)$. First, there is an involution on the
        configuration space of connections on both of these orbifolds,
        defined by multiplying the holonomy on the $S^{1}$ factor in
        $S^{1}\times S^{2}$ by
        $-1 \in \SU(2)$. This gives rise to an operator $\epsilon$ on
        instanton homology, and there is therefore a decomposition
        \[
                I(Z_{n,1};\Gamma) = I(Z_{n,1};\Gamma)^{+} \oplus
                 I(Z_{n,1};\Gamma)^{-}
        \]
        into the eigenspaces of $\epsilon$. As modules, these two are
        related by changing the variable $\tau\in \cR$ to $-\tau$.
        Each of the two summands is a cyclic module for
        $\cR[\alpha,\delta]$ and they are therefore characterized by their
        ideals of relations, $J_{n,1}^{\pm}$ in the algebra:
        \[
        \begin{aligned}
        I(Z_{n,1};\Gamma)^{+} &\cong \cR[\alpha,\delta] /
               J_{n,1}^{+} \\
        I(Z_{n,1};\Gamma)^{-} &\cong \cR[\alpha,\delta] /
               J_{n,1}^{-} .
               \end{aligned}
        \]
        Over the
        field $\C$, we can regard $J_{n,1}^{+}$ and $J_{n,1}^{-}$
        as the defining ideals
        of possibly non-reduced curves \[
        \begin{aligned}
        D_{n}^{+} , D_{n}^{-}\subset
        \C^{*}\times \C\times \C \\
        \end{aligned}
        \] with coordinates $(\tau,
        \alpha,\delta)$. Our final description of these curves is as 
        determinantal varieties: they are the loci of points where 
        particular $m\times(m+1)$ matrices $S^{+}$ and $S^{-}$ with entries in
        $\cR[\alpha,\delta]$ fail to have full rank. Here
        $m=(n-1)/2$. Equivalently, $J_{n,1}^{\pm }$ is the ideal
        generated by the $m\times m$ minors of $S^{\pm}$. Explicitly when
        $n=11$ and $m=5$, the matrix $S^{\pm}$ is given by $S_{0} \pm S_{1}$
\[
S_{0}  = \begin{pmatrix}
 -\alpha -\delta/2  & \alpha -19 \delta/2  & 0 & 0 & 0 & 0 \\
 0 & -\alpha -5 \delta/2  & \alpha -15 \delta/2  & 0 & 0 & 0 \\
 0 & 0 & -\alpha -9 \delta/2  & \alpha -11 \delta/2  &  0 & 0 \\
 0 & 0 & 0 & -\alpha -13 \delta/2  & \alpha -7 \delta/2  & 0 \\
 0 & 0 & 0 & 0 & -\alpha -17 \delta/2  & \alpha -3 \delta/2  \\
\end{pmatrix}
 \]
 and
\begin{multline*}
    S_{1} = 
    \begin{pmatrix}
         \tau ^7 & 0 & 0 & 0 & 0 \\
 0 & \tau ^3 & 0 & 0 & 0 \\
 0 & 0 & \frac{1}{\tau } & 0 & 0 \\
 0 & 0 & 0 & \frac{1}{\tau ^5} & 0 \\
 0 & 0 & 0 & 0 & \frac{1}{\tau ^9} 
    \end{pmatrix} % \\
   % \null
    \cdot
    \begin{pmatrix}
         0 & 0 & 0 & 0 & -9 & 5 \tau ^4+4 \\
 0 & 0 & 0 & -7 & 3 \tau ^4+2 & 2 \tau ^4 \\
 0 & 0 & -5 & \tau ^4 & 4 \tau ^4 & 0 \\
 0 & -3 & -\tau ^4-2 & 6 \tau ^4 & 0 & 0 \\
 -1 & -3 \tau ^4-4 & 8 \tau ^4 & 0 & 0 & 0
    \end{pmatrix}
\end{multline*}
 Although the matrices may look elaborate at first glance, they follow
 a fairly simple pattern that is readily described for general $n$.
 (See section~\ref{subsec:application-hilb}.)
 Note in particular that $S_{0}$ is a 2-band matrix with entries that
 are linear forms in $(\alpha, \delta)$, while the entries of $S_{1}$
 depend only on $\tau$.
         On setting $\tau=1$ in $S_{0}$ above, one
        recovers generators for the ideal that is identified by
        Street in \cite{Street}. For a general fixed value of $\tau$,
        the corresponding locus is a subscheme of the
        $(\alpha,\delta)$-plane of length $m(m+1)$. 
        A picture of the real locus of $D^{\pm}_{n}$ for
        $n=7$ is given in Figure~\ref{fig:curve7}, together the set of
        points on $D^{\pm}_{n}$ where $\tau=0.6$. 
\begin{figure}
    \begin{center}
       \includegraphics[width=10.5cm]{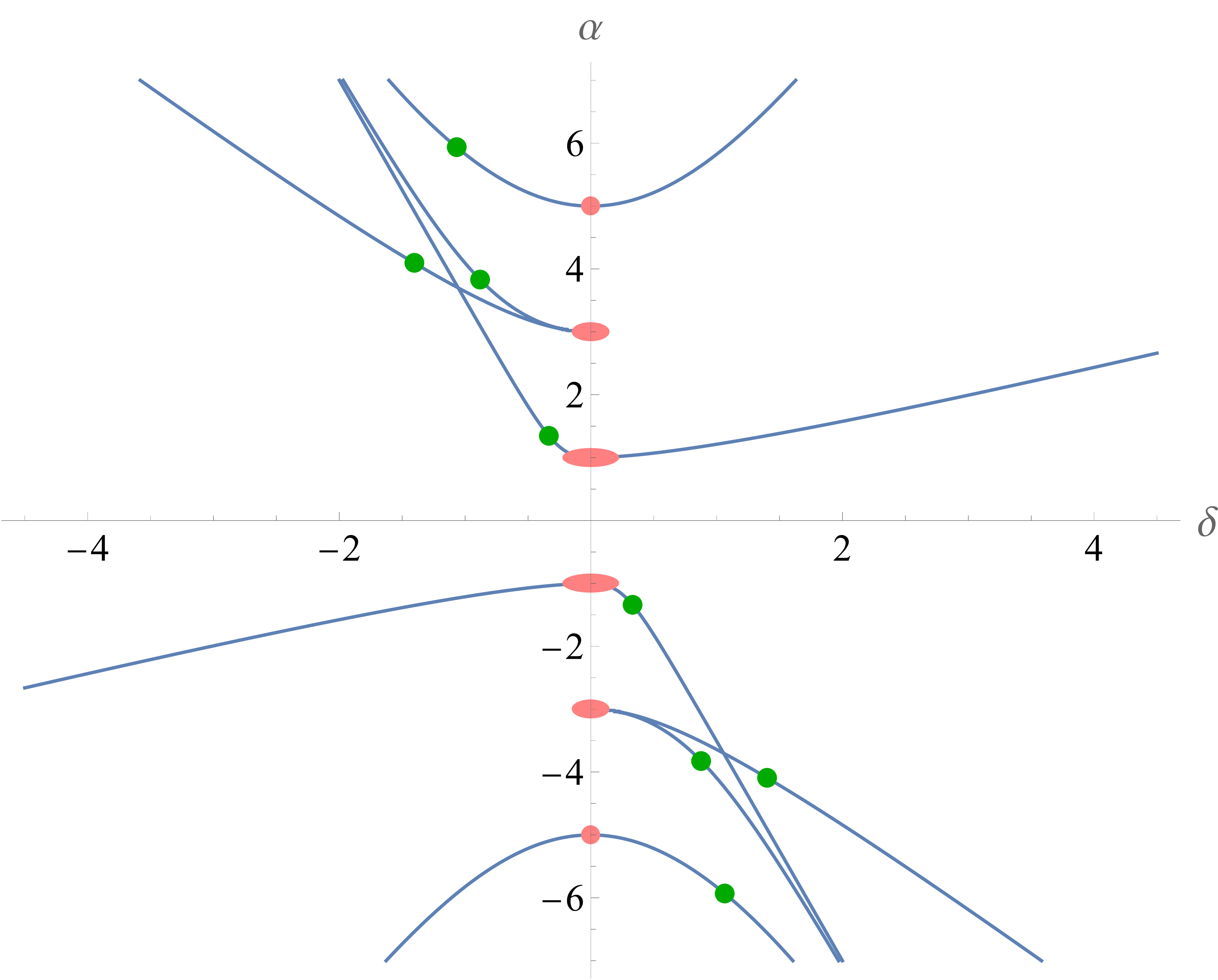}
    \end{center}
    \caption{\label{fig:curve7} The blue curve is the projection of
    the real locus of
    $D^{\pm}_{n}$ to the $(\delta,\alpha)$ plane for $n=7$. The green points are
    the points where $\tau=0.6$, showing the simultaneous eigenvalues
    of the operators $\delta$ and $\alpha$ for this value of $\tau$.
    There are 12 of these, only 8 of which are real. The pink points
    indicate the subscheme of total length 12 defined by the minors of
    $S_{0}\pm S_{1}$ when $\tau=1$, which is the case described by
    Street \cite{Street}. Although the real curve looks rather smooth at
    $\alpha=\pm 1$, it has a uni-branch triple points there: in local
    analytic coordinates, the equation of the curve has the form
    $y^{3}=x^{7}$.}
\end{figure}

\begin{remark}
This description of $D^{\pm}_{n}$ as a determinantal variety means that
the corresponding ideal $J_{n,1}^{\pm}$ is generated by $m+1$
elements, for this is the number of $m\times m$ minors. We shall see in
fact that each of these ideals can be generated by just two of the
minors.
\end{remark}

As in \cite{Munoz1, Munoz2, Street, Xie-Zhang}, the starting point for
the calculation is an explicit generating set for the ideal of
relations in the ordinary cohomology of a representation variety: in
our case, as in \cite{Xie-Zhang}, these are the ``Mumford relations''
in the cohomology of the representation variety associated to
$S^{2}_{n}$. (See \cite{Earl-Kirwan} for example.) We obtain simple
explicit formulae for these relations as products of linear forms in
the variables $\alpha$ and $\delta_{i}$. The matrix $S_{0}$ above
arises as a matrix of syzygies for the Mumford relations. To compute
the deforming term $S_{1}$, it is only necessary to understand the
contributions of moduli spaces of instantons on $\R\times Z_{n}$ of
smallest non-zero action (action $1/4$ in the normalization where the
standard instanton on $\R^{4}$ has action $1$). The contributions of
these moduli spaces can be understood quite explicitly by a
wall-crossing argument. A closely related phenomenon is present in
\cite{Munoz2}.
        
 \subsection{Outline}

 In section~\ref{sec:Instanton} we recall the definition of singular
 instanton homology with local coefficients and the construction of
 the operators that act on it in general. (Note that from
 section~\ref{sec:Instanton} onwards, we simply write $I(Z)$ for the
 homology group referred to as $I(Z;\Gamma)$ above, without explicit
 mention of the local coefficients.) In section~\ref{sec:braids}, we
 introduce $Z_{n}$ and $Z_{n,\pm 1}$ and study the ordinary cohomology
 of the relevant representation
 varieties and instanton homologies, enough to show that these can be
 described as cyclic modules for the algebra of operators which act on
 them. This material is quite standard.

 In section~\ref{sec:relations-ordinary}, we describe the Mumford
 relations in the ordinary cohomology of the representation variety
 of $Z_{n}$. We derive a very explicit formula for generators of the
 ideal of relations in these cohomology groups. The relations in the
 ordinary cohomology ring of the representation variety of $Z_{n}$
 admit a deformation which yields relations in the instanton homology
 $I(Z_{n})$. The existence of this deformation is established in
 section~\ref{sec:relations-instanton} together with a calculation of
 the subleading term using a wall-crossing calculation rather as in
 \cite{Munoz2}.

 Knowledge of the subleading term turns out to be sufficient to obtain
 a complete answer, and the description of $I(Z_{n,1})$ (or
 equivalently $I(Z_{n,-1})$) that is outlined earlier in this
 introduction is derived in section~\ref{sec:calculation}. Some further
 remarks are contained in section~\ref{sec:further} at the end of the
 paper.

	\section{A version of singular instanton homology}
        \label{sec:Instanton}

        In this section we review the construction of instanton
        homology with local coefficients, for admissible bifolds.
        General references include \cite{KM-yaft} and
        \cite{KM-ibn1}.
        
	\subsection{Bifolds and their Floer homology}
        \label{subsec:bifold-connections}
        
        For economy of notation, we will typically write simply $Z$
        for a pair consisting of a connected, oriented 3-manifold
        $Y$ and an embedded (unoriented) link $K=K(Z) \subset Y$.
        Following \cite{KM-yaft} and \cite{KM-unknot}, we will regard
        $Z$ as determining an orbifold (a \emph{bifold} in the
        notation of \cite{KM-Tait}) whose underlying topological
        space is $Y$ and whose singular set is $K(Z)$. The local
        stabilizer of the orbifold geometry at points of $K(Z)$ is of
        order $2$. When talking of (for example) Riemannian metrics on
        $Z$, we will always mean orbifold Riemannian metrics. A bifold $Z$ is
        \emph{admissible} if there is an element of $H^1(Y;\Z)$ which has
        non-zero mod-2 pairing with the class $[K(Z)] \in H_1(Y;\Z/2)$.
	
        Associated to a 3-dimensional bifold $Z$, we have a space of bifold 
        connections $\bonf(Z)$. In this paper, $\bonf(Z)$ will always
        consist of the bifold $\SO(3)$ connections with $w_{2}=0$
        modulo the determinant-1 gauge group. In the language of
        \cite[section~2]{KM-ibn1}, this is the space of marked
        bifold connections in which the marking region is the
        complement of the singular set $K(Z)$ and the bundle has
        $w_{2}=0$ on the marking region.

        \begin{remark}
        The space $\bonf(Z)$ can be identified with the
        space of gauge equivalence classes of $\SU(2)$ connections on the complement of the
        singular set $K(Z)$ such that the associated $\SO(3)$ bundle
        extends to an orbifold $\SO(3)$ bundle on $Z$ with non-trivial
        monodromy (of order 2) at the singular points. When
        interpreted as $\SU(2)$ connections in this way, the limiting
        holonomy of the $\SU(2)$ connections on small loops linking
        the singular locus has order $4$. This is the viewpoint adopted,
        for example, in \cite{KM-gtes1, KM-gtes2}.
        \end{remark}

        \begin{definition}\label{def:rep}
        We write  
        $\Rep(Z)\subset \bonf(Z)$ for the space of flat bifold
        connections modulo the determinant-1 gauge group. If $Z$ is admissible, then
        $\Rep(Z)$ consists only of irreducible connections.
        \end{definition}
	
	\subsection{A local coefficient system}
        \label{sec:local-system}

         For each
        component $K^i\subset K(Z)$, after choosing a framing, we
        obtain a map to $S^1$,
	\[
	             h_i : \bonf(Z)\to S^1,
	\]
        as in \cite{KM-yaft} and \cite[section 2.2]{KM-ibn1}.
        Specifically, following \cite{KM-yaft},
        given $[A]\in \bonf(Z)$, we may restrict the 
        connection $[A]$ to the boundary of the framed
        $\epsilon$-tubular neighborhood of $K^i$ and obtain, in the
        limit as $\epsilon\to 0$, a flat $\SO(3)$ connection on the
        torus whose structure group reduces to $\SO(2)$. The holonomy
        of the $\SO(2)$ connection along the longitude defines
        $h_i([A])$. 

        An orientation of $K^i$ is not needed here, because the
        orientation of the $\SO(2)$ bundle also depends on an
        orientation of $K^i$. (That is, the orientation of $K^i$ is
        used twice in this construction.) The framing is also
        inessential, as a change of framing will change $h_i$ by a
        half-period.

        Taking the product over the set of all components of $K$, we
        define a single map $h: \bonf(Z)\to S^{1}$ by
        \[
             h = \times_{i} h_i.
        \]
        
        Over the circle $S^1$, there is a standard local system with
        fiber the ring of finite Laurent series
        \begin{equation}\label{eq:cR-def}
                 \cR = \Q[\tau^{\pm 1}].
        \end{equation}
        such
        that the monodromy of the local system around the positive
        generator of $S^1$ is multiplication by $\tau$. Then by pulling
        back this local system by the  map $h$,
        we obtain a local system $\Gamma$ on
        $\bonf(Z)$. We summarize this construction with a definition.

        \begin{definition}\label{def:cR}
        Unless otherwise stated, the notation $\cR$ will denote the
        ring $\Q[\tau^{\pm 1}]$, and $\Gamma$ will denote the
        corresponding local system of free rank-1 $\cR$-modules over
        $\bonf(Z)$, for any 3-dimensional bifold $Z$.
        \end{definition}
        
        If $Z$ is
        admissible, then by the standard construction (see
        \cite{KM-yaft, KM-deformation}), we obtain an instanton homology
        group for admissible bifolds:

        \begin{definition}\label{def:I}
        Let $Z$ be an admissible bifold of dimension $3$.
        After choosing a Riemannian
        metric and perturbation to achieve a Morse-Smale condition for
        the gradient flow of the Chern-Simons functional on
        $\bonf(Z)$, we obtain an instanton Floer complex
        $\CI(Z;\Gamma)$ of free $\cR$-modules whose homology
        $I(Z;\Gamma)$ is the instanton homology
        of $Z$. We will generally write $I(Z)$ and omit $\Gamma$ from
        the notation, unless the context demands otherwise. This is a
        $\Z/4$ graded module.
        \end{definition}

        \subsection{Functoriality and operators}
        \label{subsec:operators}
        
        We consider 4-dimensional bifolds $W$ as cobordisms between
        3-dimensional bifolds. In the context of this paper, the
        singular locus $\Sigma=\Sigma(W)$ of the orbifold $W$ will
        always be an embedded surface (not necessarily orientable). In
        particular, we
        do not consider foams -- singular surfaces -- as in \cite{KM-Tait}. The Floer
        homology groups $I(Z)$ are  functorial
        in the sense that a bifold cobordism $W$ from $Z^{0}$ to
        $Z^{1}$ gives rise to a map
        \[
               I(W) : I(Z^{0})\to I(Z^{1})
        \]
         compatible with compositions.

         The map $I(W)$ is obtained from
         suitable weighted counts of solutions to the perturbed
         anti-self-duality equations on the bifold $W$, after
         attaching cylindrical ends. This construction initially
         gives rise only to a projective functor, in that the overall
         sign of $I(W)$ is ambiguous. When $\Sigma(W)$ is oriented, the sign
         ambiguity can be resolved by choosing a homology orientation
         for $W$ in the sense of \cite{KM-yaft}. In the case that
         $\Sigma(W)$ is not necessarily orientable, an appropriate
         substitute is the notion of an $\imath$-orientation introduced
         in \cite{KM-unknot}. (The sign ambiguity in the
         non-orientable case will not particularly concern us in this
         paper.)

         Recall that in the present context $I(Z)$ denotes the
         instanton homology with coefficients in the local system
         $\Gamma$. That being so, the solutions $A$ to the perturbed
         anti-self-duality equations on $W$ are counted not just with
         signs $\pm 1$, but with weights that are units in the ring
         $\cR$. More precisely, if $\rho_{0}$ and $\rho_{1}$ are
         critical points of the perturbed Chern-Simons functional in
         $\bonf(Z^{0})$ and $\bonf(Z^{1})$, and if $[A]$ is a solution
         of the perturbed equations on $W$ with cylindrical ends,
         asymptotic to $\rho_{0}$ and $\rho_{1}$, then $[A]$
         contributes to the matrix entry of the map $I(W)$ at the
         chain level with a contribution $\pm \Gamma(A)$, where
         $\Gamma(A): \Gamma(\rho_{0}) \to \Gamma(\rho_{1})$ is given
         by
         \begin{equation}\label{eq:Gamma-A}
                    \Gamma(A) = \tau^{\nu(A) +
                    (1/2)(\Sigma\cdot\Sigma)}.
         \end{equation}
	Here $\mu$ is obtained from a curvature integral on the
         2-dimensional singular set $\Sigma=\Sigma(W)$, and the
         self-intersection number $\Sigma\cdot\Sigma$ is computed
         relative to chosen framings of the singular sets $K(Z^{0})$
         and $K(Z^{1})$. The expression on the right-hand side of
         \eqref{eq:Gamma-A} is not an element of $\cR$ itself, because
         the exponent is not generally an integer. It is, however, a
         homomorphism between the rank-1 $\cR$-modules
         $\Gamma(\rho_{0})\to \Gamma(\rho_{1})$ in a natural way.
         For details of this
         construction see, for example, \cite[section 3.9]{KM-yaft} and
         \cite{KM-ibn1}. As explained there, the choice of framings
         is essentially immaterial. Consistent with our notation
         $I(Z)$ in which the local coefficient system $\Gamma$ is
         implied, we will continue to write simply $I(W)$ for the
         $\cR$-module homomorphism between these instanton
         homology groups.

	As well as the map $I(W)$ above, we have the generalizations
        obtained by cutting down the moduli spaces on $W$ by
        cohomology classes in the configuration space of bifold
        connections $\bonf(W)$. Here $\bonf(W)$ is a space of $\SO(3)$
        bifold connections modulo the determinant-1 gauge group, and
        in the language of \cite{KM-ibn1}, this is the space of marked
        bifold connections in which the marking region is the
        complement of the singular set $\Sigma(W)$ and the bundle has
        $w_{2}=0$ on the marking region.

        To describe the relevant cohomology classes more specifically,
        and to fix conventions,
        there is a universal orbifold $\SO(3)$ bundle,
        \[
                \mathbb{E}    \to \bonf^{*}(W) \times W
        \]
        which has an orbifold Pontryagin class, \[ p_{1}^{\orb}(\mathbb{E}) \in
        H^{4}(\bonf^{*}(W) \times W; \Q). \] We adopt
        the convention that our preferred 4-dimensional characteristic
        class is $-(1/4)p_{1}^{\orb}(\mathbb{E})$, which coincides with
        $c_{2}^{\orb}(\tilde {\mathbb{E}})$ in the case that there is a lift to an
        $\SU(2)$ bundle $\tilde {\mathbb{E}}$. Given a class $\gamma$ in
        $H^{2}(W;\Q)$ or $H^{0}(W;\Q)$, we obtain classes
        \begin{equation}\label{eq:slant-mu}
            -(1/4)p_{1}^{\orb}(\mathbb{E}) / [\gamma]
        \end{equation}
        in $H^{2}(\bonf^{*}(W);\Q)$ or $H^{4}(\bonf^{*}(W);\Q)$
        respectively. 

        In addition to the classes \eqref{eq:slant-mu}, if $p$ is a
        point of the orbifold locus $\Sigma(W)$, then the restriction
        of $\mathbb{E}$ to $\bonf^{*}(W) \times \{p\}$ has a decomposition
        \[
               \mathbb{E}_{p} = \R \oplus \mathbb{V}_{p}
        \]
        where $\mathbb{V}_{p}$ is a 2-plane bundle.
        An orientation of $\mathbb{V}_{p}$
        depends on a choice of normal orientation to the orbifold
        locus at $p$. Having chosen such an orientation,
        a class $\delta_{p} \in H^{2}(\bonf^{*}(W); \Q)$ is then defined as
        \begin{equation}\label{eq:delta-p-def}
                    \delta_{p} = \frac{1}{2} e(\mathbb{V}_{p}).
        \end{equation}
        We can regard $\delta$ here as depending on a choice of an
        element in $H_{0}(\Sigma(W);O)$, where $O$ is the orientation
        bundle of $\Sigma(W)$ with rational coefficients.
        
        Combining the classes \eqref{eq:slant-mu} for $\gamma\in
        H^{i}(W; \Q)$ and the classes $\delta_{p}$, we
        obtain homomorphisms of $\cR$-modules
        \begin{equation}\label{eq:IWa}
                I(W,a) : I(Z^{0})\to I(Z^{1})
        \end{equation}
        depending linearly on
        \begin{equation}\label{eq:sympolys}
                a \in \mathrm{Sym}_{*}\biggl( H_{2}(W; \Q) \oplus
                 H_{0}(W; \Q) \oplus
                          H_{0}(\Sigma(W); O)\biggr).
        \end{equation}
        Since $I(Z^{0})$ and $I(Z^{1})$ are $\cR$-modules, we may
        extend linearly over $\cR$ to allow also
      \begin{equation}\label{eq:sympolys-R}
                a \in \mathrm{Sym}_{*}\biggl( H_{2}(W; \Q) \oplus
                 H_{0}(W; \Q) \oplus
                          H_{0}(\Sigma(W); O)\biggr) \otimes \cR.
        \end{equation}

        The construction of the operators $I(W,a)$ is suitably
        functorial. In particular, this means for us that, in the case
        that $W$ is a cylinder $[0,1]\times Z$, we have
        \[
                I(W, a_{1} a_{2}) = I(W, a_{1}) I(W, a_{2}).
        \]
        
        We will always be dealing with the case that $W$ is connected,
        so there is only one class $[w]$ in $H_{0}(W;\Q)$.
        From \cite{Obstruction,KM-yaft}, we note the following relation
        among the homomorphisms $I(W,a)$.
        
        \begin{proposition}\label{prop:2-d-relation}
        Let $p$ a point in $\Sigma(W)$  with a chosen orientation of
        $T_{p}\Sigma(W)$, representing a class in
        $H_{0}(\Sigma(W); O)$ in the algebra \eqref{eq:sympolys}.
        Let $w$ be a point in $W$, representing a class in
        $H_{0}(W; \Q)$. Then we have a relation
        \[
                     I\bigl(W, (p^{2} + w - \tau^{2} -
                     \tau^{-2})b\bigr) = 0,
        \]
        for any $b$ in the algebra \eqref{eq:sympolys}.
        \end{proposition}

        \begin{corollary}
            The map $I(W, p^{2}b)$ is
            independent of the choice of oriented point $p \in \Sigma(W)$.            
        \end{corollary}
 
        \begin{remark}
        The relation in Proposition~\ref{prop:2-d-relation} reflects
        (in part) a relation in the cohomology ring of
        $\bonf^{*}(W)$, where we have a 2-dimensional class
        $\delta_{p}$ and a 4-dimensional class
        $-(1/4)p_{1}^{\orb}(\mathbb{E})/ [w]$. From their construction as
        characteristic classes, these satisfy
        \begin{equation}\label{eq:2-d-relation-coh}
               \delta^{2}_{p} -(1/4)p_{1}^{\orb}(\mathbb{E})/ [w] = 0
        \end{equation}
        in $H^{4}(\bonf^{*}(W); \Q)$. The extra terms
        $\tau^{2}+\tau^{-2}$ in the proposition arise from instanton
        bubbling
        contributions \cite{Obstruction}. 
        \end{remark}

        Proposition~\ref{prop:2-d-relation} also tells that the generator
        corresponding to $[w]\in H_{0}(W;\Q)$ is redundant. We obtain
        the most general homomorphism $I(W,a)$ if we only take $a$ in
        the smaller algebra
        \begin{equation}\label{eq:sympolys-reduced}
        \mathrm{Sym}_{*}\biggl( H_{2}(W; \Q) \oplus
                          H_{0}(\Sigma(W); O)\biggr).
        \end{equation}

        There is an additional construction we can make if
        we are given a distinguished class $e\in
        H_{2}(W; \Z)$. We consider the space $\bonf(W)^{e}$ of marked bifold
        $\SO(3)$ connections on $W$ where the marking region is again
        the complement of $\Sigma(W)$ and where the marking data has
        \[
        w_{2}=\PD(e)|_{W\setminus \Sigma(W)} \bmod  2.
        \]
        After attaching cylindrical ends, the instantons in
        $\bonf(W)^{e}$  provide us with maps
        \begin{equation}\label{eq:IWa-e}
                 I(W,a)^{e} : I(Z^{0})\to I(Z^{1}) .
        \end{equation}
        The integer lift $e$ in homology is used to orient the moduli
        spaces and determines the overall sign of the map
        $I(W,a)^{e}$. If $e - e' = 2 v$, so that $e$ and $e'$ define
        the same mod 2 class, then (as in \cite{Donaldson-orientations})
        we have
        \begin{equation}\label{eq:e-sign}
                        I(W,a)^{e'} = (-1)^{v\cdot v} I(W,a)^{e}.
        \end{equation}
        
        \begin{remark}
        Note that, as discussed for example in \cite{KM-unknot}, one
        can more generally consider the case that $e$ is a relative
        class so that  $\partial e \in H_{1}(\Sigma(W))$, but the more
        restrictive version here is required because we wish to use
        the local coefficient system $\Gamma$, which is otherwise not
        defined. See also \cite[section 2.2]{KM-ibn1}.
        \end{remark}

	\section{Torus braids in $S^1\times S^2$}
        \label{sec:braids}

        \subsection{The torus braids}
        
        The following examples play an important role for us.
	
	\begin{definition}\label{def:Zn}
        Let $\pp=\{p_1, \dots, p_n\}$ be $n$ points arranged
        symmetrically around the equator of $S^2$. We write $Z_{n}$
        for the bifold whose underlying 3-manifold $Y$ is the product
        $S^1\times S^2$ and whose singular locus $K$ is the
        $n$-component link \[ K_{n}=S^1\times \pp \subset S^{1}\times
        S^{2}.\] 
	\end{definition}

	\begin{definition}\label{def:Znq}
    For any $q\in \Z$, we define a bifold $Z_{n,q}$ as follows. The
    3-manifold $Y$ is again $S^1\times S^2$. If $\phi\in \R/(2\pi\Z)$
    denotes an angular coordinate on the equator of $S^2$, and
    $\theta$ a coordinate on the $S^1$ factor, then 
    $K=K_{n,q}$ will be the link determined by $n\phi =
    q\theta\pmod{2\pi}$.
	\end{definition}

        The bifold $Z_{n,q}$ is admissible when $n$ is odd. The link
        $K_{n,q}\subset S^1\times S^2$ is connected (a knot) when $n$
        and $q$ are coprime. When $q=0$, the orbifold $Z_{n,0}$
        coincides with $Z_{n}$ above.

        It is evident from the definitions that the orbifold $Z_{n,q}$
        is isomorphic to $Z_{n,-q}$ by an orientation-reversing map.
        With a little more thought, one can see that there is also an
        orientation-\emph{preserving} isomorphism:

\begin{lemma}\label{lem:orientation-Knq}
       The link $K_{n,q}$ is isotopic in $S^{1}\times S^{2}$ to the
       link $K_{n,-q}$. As a consequence, there is an
       orientation-preserving isomorphism of bifolds from $Z_{n,q}$ to
       $Z_{n,-q}$. 
\end{lemma}
	
	\begin{proof}
            Let $L$ be an oriented axis in $\R^{3}$ passing through
        two points of the equatorial circle in the above description
        of $K_{n,q}$. Let $\rho_{t}$ be the
        rotation of $S^{2}$ about this axis through angle $2\pi t$,
        and let $1\times \rho_{t}$ be the resulting map $S^{1}\times
        S^{2}\to S^{1}\times S^{2}$. Then the link
        \[
               K_{t} = (1\times \rho_{t})(K_{n,-q}) \subset
               S^{1}\times S^{2}
        \]
        coincides with $K_{n,-q}$ when $t=0$ and with $K_{n,q}$ when
        $t=1/2$.
        \end{proof}

        We aim to give a description of $I(Z_{n})$ (the
        instanton homology with local coefficients) as an
        $\cR$-module, together with a description of the operators
        \[
            I([0,1] \times Z_{n}, a) : I(Z_{n}) \to I(Z_{n})
        \]
        and
        \[
               I([0,1] \times Z_{n}, a)^{e} : I(Z_{n}) \to I(Z_{n})
        \]
        arising from classes $a$ by the general construction
        \eqref{eq:IWa} and \eqref{eq:IWa-e}, where $e$ is the
        2-dimensional class in $H_{2}(Z_{n};\Q)$. 
        
	\subsection{The representation variety of $S^{2}_{n}$.}
        \label{subsec:Rep2d}
	
         Let us assume henceforth that $n$ is odd, so that the orbifold $Z_{n}$ described
         above is
         admissible. We may describe $Z_{n}$ as a
        product $S^1\times S^2_n$, where $S^2_n$ is a 2-dimensional
        bifold of genus $0$, and we begin with some observations about
        the 
        the representation variety $\Rep(S^2_{n})$, drawn from
        \cite{Boden, Weitsman, Street}. Note that we can identify
        $\Rep(S^{2}_{n})$ with the space of flat $\SU(2)$ connections
        on the complement of the $n$ singular points such that the
        monodromy at each puncture has order $4$. (See the remark in
        section~\ref{subsec:bifold-connections}.)
        
        First, as $n$ is odd, the variety $\Rep(S^2_{n})$ consists entirely of
        irreducible connections. It is a smooth, compact, connected
        manifold of dimension $2n-6$ for $n \ge 3$, and is empty for
        $n=1$. We have no need for a detailed description of their
        topology, but we record the fact that $\Rep(S^{2}_{3})$ is a
        single point and $\Rep(S^{2}_{5})$ is diffeomorphic to the
        blow up of $\CP^2$ at $5$ points.  It will be convenient to
        make use of the following result, which the authors believe
        has the status of folklore. The statement and proof are very
        minor adaptations of the main result of \cite{Kirk-Klassen}.
        See also \cite{Thaddeus}.

        \begin{lemma}\label{lem:Morse-even}
            For any odd $n$, the manifold $\Rep(S^{2}_{n})$ admits a
            Morse function with critical points only in even index.
        \end{lemma}

        \begin{proof}
        Following \cite{Kirk-Klassen}, we present a proof by induction
        on $n$. So assume the result is true for a particular $n$, and
        consider $\Rep(S^{2}_{n+2})$. Let $C\subset \SU(2)$ be the
        subset of elements of order $4$, i.e.~the unit sphere of
        imaginary quaternions. Let $\tilde R \subset C^{n+2}$ be the
        locus
        \[
                \{ \, (i_{1},\dots, i_{n+2}) \in C^{n+2} \mid
                  i_{1}i_{2}\cdots i_{n+2}=1 \,\},
        \]
        so that the representation variety $\Rep(S^{2}_{n+1})$ is the
        quotient of $\tilde R$ by conjugation. For $\mathbf{i} \in
        \tilde R$, there is a unique $\theta\in [0,\pi]$ such that
        \[
            i_{n+1}i_{n+2} \sim 
            \begin{pmatrix}
                e^{i\theta} & 0 \\ 0 & e^{-i\theta}                
            \end{pmatrix}
        \]
        and we have a smooth function
        \[
                h = \cos(\theta) = \frac{1}{2}\tr ( i_{n+1}i_{n+2})
        \]
        which descends to a smooth function
        \[
                 h : \Rep(S^{2}_{n+2})\to [-1,1].
        \]
        We consider separately the loci $h^{-1}(1)$, $h^{-1}(-1)$ and
        $h^{-1}((-1,1))$.

        If $\mathbf{i}\in h^{-1}(1)$, then $i_{n+1}i_{n+2} =1$, and it
        follows that $i_{1}i_{2}\cdots i_{n}=1$.
        So these remaining $n$ points define a point in
        $\Rep(S^{2}_{n})$. The remaining choice of $i_{n+1}$ exhibits
        $h^{-1}(1)$ as a 2-sphere bundle over $\Rep(S^{2}_{n})$. As in
        \cite{Kirk-Klassen}, we may use the induction hypothesis to
        show that a perturbation of $h$ has critical points only of
        even index near $h=1$. The situation at $h^{-1}(-1)$ is
        essentially the same: multiplying $i_{1}$ and $i_{n+2}$ by $-1$
        interchanges these two loci.

        On the locus $h^{-1}((-1,1))$, the function $h$ itself is
        Morse and its critical points can be described as follows. Let
        $i$, $j$, $k$ in $C$ be the standard unit quaternions with
        $ijk=1$. Given any element of $h^{-1}((-1,1))$ we can use the
        action of conjugation to uniquely put in standard form with
        $i_{n+1}=i$ and $i_{n+2}$ lying in the interior of the
        semicircle $\gamma$ which joins $i$ to $-i$ and passes through
        $j$. In this standard form, there is a circle action on $h^{-1}((-1,1))$ which fixes
        $i_{n+1}$ and $i_{n+2}$ and rotates the points $i_{1},\dots,
        i_{n}$ about the axis through $k$. The function $\theta =
        \cos^{-1}(h)$ is smooth on this locus and is the moment map of
        the circle action. 
        The critical points of $h$
        are therefore precisely the fixed points of this circle
        action. These fixed points
        are the points which in standard form have $i_{n+1}=i$,
        $i_{n+2}=j$ and $i_{m}=\pm k$ for all other $m$. The
        constraint $i_{1}i_{2}\cdots i_{n+1}=1$ means that $i_{m}=-k$
        for an even number of indices $m$ in the range $1,\dots, n$.
        As a general property of moment maps, because these fixed
        points are isolated, they are Morse critical points for $h$,
        of even index.        
        \end{proof}

        \begin{remark}
            The proof of the lemma above gives a little bit more, for
            we can easily identify the indices of the critical points,
            and hence establish the recursive formula for the
            Poincar\'e polynomial of $\Rep(S^{2}_{n})$ which is given
            in \cite{Street}. The loci $h^{-1}(1)$ and $h^{-1}(-1)$,
            which are the 2-sphere bundles over $\Rep(S^{2}_{n})$
            inside $\Rep(S^{2}_{n+2})$, are the minima and maxima of
            $h$ and together make a contribution
            \[
                                 (1+t^{2})^{2} P_{n}(t) 
            \]
            to the Poincar\'e polynomial $P_{n+2}$ for
            $\Rep(S^{2}_{n+2})$. Using the symmetries of
            $\Rep(S^{2}_{n+2})$ obtained by multiplying an even
            number of the $i_{l}$ by $-1$, it is easy to see that the
            remaining critical points in $h^{-1}((-1, 1))$ all have
            the same index and that this index is the middle
            dimension $(n-1)$. There are $2^{n-1}$ of these critical points,
            so we recover the recursive formula
            \[
                        P_{n+2}(t) = (1+t^{2})^{2} P_{n}(t) +
                          (2t)^{n-1}
            \]
            from \cite{Street}.
        \end{remark}

              Atiyah and Bott \cite{AB}
        described standard generators for the cohomology ring of
        representation varieties of surfaces in the non-orbifold case (a smooth
        surface of genus $g$), and there is an extension of those
        techniques for the orbifold case, developed in
        \cite{Biswas-Raghavendra}. For the specific case of
        $S^{2}_{n}$, the results are given in \cite{Street}.

        In this description, the generators of the cohomology ring
        $H^{*}(\Rep(S^{2}_{n} ; \Q))$ are classes
        \begin{equation}\label{eq:AB-classes}
                   \begin{aligned}
                    \alpha &\in H^{2}(\Rep(S^{2}_{n} ); \Q) \\
                    \beta &\in H^{4}(\Rep(S^{2}_{n} ); \Q) \\
                    \delta_{p} &\in H^{2}(\Rep(S^{2}_{n} ); \Q),\quad
                    p\in \pi,
                   \end{aligned}
        \end{equation}
        which are the restrictions to $\Rep(S^{2}_{n})$ of classes
        defined on the space of irreducible bifold connections,
        $\bonf^{*}(S^{2}_{n})$ arising from the slant product
        construction \eqref{eq:slant-mu}. More specifically, the
        classes $\alpha$ and $\beta$ arise from the fundamental
        2-dimensional class $[S^{2}_{n}] \in H_{2}(S^{2}_{n})$ and
        the point class $[w] \in H_{0}(S^{2}_{n})$ respectively, while
        $\delta_{p}$ is defined as in \eqref{eq:delta-p-def}:
        \begin{equation}\label{eq:AB-generators}
        \begin{aligned}
        \alpha &= -(1/4)p_{1}^{\orb}(\mathbb{E}) / [S^{2}_{n}] \\
        \beta &= -(1/4)p_{1}^{\orb}(\mathbb{E}) / [w] \\
        \delta_{p} &= \frac{1}{2} e(\mathbb{V}_{p}).
        \end{aligned}
        \end{equation}
        We will sometimes write
        \[
                \delta_{1} , \dots, \delta_{n}
        \]
        for the classes $\delta_{p_{i}}$ as $p_{i}$ runs through
        $\pi$.

        The classes $\alpha$ and $\beta$ can also be seen as arising
        from the K\"unneth
        decomposition in $H^{4}(\bonf^{*}(S^{2}_{n}) \times S^{2}_{n};
        \Q)$,
        \[
            -(1/4)p_{1}^{\orb}(\mathbb{E}) = \beta\times 1 + \alpha\times v
        \]
        where $v$ is the generator of $H^{2}(S^{2}_{n};\Q)$.
        The
        generator $\beta$ is redundant, because of the relation
        \[
                  \delta_{p}^{2} = -\beta, \quad \forall p \in \pi,
        \]
        which is a restatement of \eqref{eq:2-d-relation-coh} in the
        current situation.
        
        In the rational cohomology ring of $\bonf^{*}(S^{2}_{n})$, there
        are no further relations: the cohomology ring is the algebra
        \begin{equation}\label{eq:bonf-algebra}
            H^{*}(\bonf^{*}(S^{2}_{n});\Q) = \Q[\alpha, \delta_{1}, \dots,
            \delta_{n}] / \langle
            \delta_{k}^{2}-\delta_{l}^{2} \rangle_{k,l}.
        \end{equation}
        
        We have a  surjective homomorphism
        \begin{equation}\label{eq:restriction-ker}
               \phi : H^{*}(\bonf^{*}(S^{2}_{n});\Q) \to
               H^{*}(\Rep(S^{2}_{n}); \Q).
        \end{equation}

        \begin{definition}\label{def:lower-jn}
        We write $A_{n}$ for the algebra
        \[
        \begin{aligned}
        A_{n} &= H^{*}(\bonf^{*}(S^{2}_{n});\Q) \\
                      &= \Q[\alpha, \delta_{1}, \dots,
            \delta_{n}] / \langle
            \delta_{k}^{2}-\delta_{l}^{2} \rangle_{k,l}
            \end{aligned}
        \]
          and we write
        \[
                j_{n}  \subset A_{n}
        \]
       for the kernel of the surjective homomorphism $\phi$.
        \end{definition}
        
        Generators for the ideal $j_{n}$
        are described in detail in \cite{Street}, which leads to a
        complete description of the cohomology ring, 
\begin{equation}\label{eq:coho-Rep}
                 H^{*}(\Rep(S^{2}_{n}); \Q) =A_{n} /
            j_{n}.
\end{equation}
       See also Proposition~\ref{prop:little-w-generate}.

\subsection{The representation variety of $Z_{n}$}
 \label{subsec:Rep3d}

       The flat bifold connections on $Z_{n}$ are of two sorts,
        which we call the ``plus'' and ``minus'' components which can
        be distinguished by examining the holonomy of the flat
        connection along the $S^{1}$ factor in $Z_{n}=S^{1}\times
        S^{2}_{n}$. The
        representations in the plus component are pulled back from
        $S^2_n$. The representations in the minus component are
        obtained from these by multiplication by a flat real line
        bundle with holonomy $-1$ on the $S^1$ factor.
        Thus we have
	\begin{equation}\label{eq:RepZn0}
	\begin{aligned}		
	   \Rep(Z_{n}) &=  \Rep(Z_{n})_+ \cup  \Rep(Z_{n})_- \\
	           &= \Rep(S^2_n) \cup \Rep(S^2_n).
	\end{aligned}
	\end{equation}

        Because of this, the description
        \eqref{eq:coho-Rep} of the cohomology ring of
        $\Rep(S^{2}_{n})$ leads immediately to a description of the
        cohomology of $\Rep(Z_{n})$. We are also eventually
        interested in the cohomology of the representation variety
        with constant coefficients $\cR$ rather than $\Q$ (because of
        our interest in instanton homology with local coefficients
        $\Gamma$). With this in mind,
        let
        \[
            \epsilon : H^{*}(\Rep(Z_{n});\cR) \to H^{*}(\Rep(Z_{n});\cR)
        \]
        be the map obtained from interchanging the two copies, so
        that $\epsilon^{2}=1$. 
        We write $\cA_{n}$ for the algebra
        \begin{equation}\label{eq:cA-newdef}
               \cA_{n} = \cR[\alpha, \delta_{1}, \dots, \delta_{n}, \epsilon]\bigm/ \bigl\langle \epsilon^{2}-1,
                 \delta_{k}^{2}-\delta_{l}^{2}\bigr\rangle_{k,l}.
        \end{equation}
        That is, we extend the coefficient ring of the algebra
        \eqref{eq:coho-Rep} from $\Q$ to $\cR$, and we adjoin the
        element $\epsilon$ with square $1$. This provides us with the
        following description. In the statement below, we write
        \[
                1_{+} \in H^{0}(\Rep(Z_{n}))
        \]
        for the element Poincar\'e dual to the fundamental class of
        the component $\Rep(Z_{n})_{+}$.

        \begin{proposition}
\label{prop:coho-Z-cyclic}
        The cohomology of the representation variety $\Rep(Z_{n})$
        with coefficients in $\cR$ is a cyclic module for the algebra
        $\cA_{n}$ with generator the element $1\in
        H^{0}(\Rep(Z_{n});\cR)$. We have
           \begin{equation}\label{eq:HRep-as-quotient}
                H^{*}(\Rep(Z_{n}) ;\cR) \cong \cA_{n} / J_{n}
        \end{equation} 
        where
        \[
            J_{n} = (j_{n} + \epsilon j_{n})\otimes \cR
        \]
        and $j_{n}$ is the ideal in \eqref{eq:coho-Rep}. Using
        Poincar\'e duality, we can equivalently describe the homology
        $H_{*}(\Rep(Z_{n});\cR)$ as a cyclic $\cA_{n}$-module with
        generator the class $[\Rep(Z_{n})_{+}]$, with the classes
        $\alpha$ and $\delta_{k}$ acting by cap product.
\end{proposition}

        We regard $\cA_{n}$ as a graded algebra
        with the generators $\alpha$
        and $\delta_{k}$ in grading $1$ (not $2$) and $\epsilon$ in grading $0$. From the
        grading, $\cA_{n}$ obtains an increasing filtration, which for future
        reference we record as
        \begin{equation}\label{eq:filtration}
                \cA_{n}^{(0)} \subset   \cA_{n}^{(1)} \subset
                \cA_{n}^{(2)} \subset \cdots \subset \cA_{n},
        \end{equation}
        where $\cA_{n}^{(s)}$ is the $\cR$-submodule generated by
        elements in grading less than or equal to $s$.

       From the explicit description of the generators of $j_{n}$
       given in \cite{Street} (for rational coefficients),
       we can read off that there are no
       relations between the generators up to the middle dimension of
       $\Rep(Z_{n})$:

       \begin{proposition}\label{prop:below-middle}
         For $s\le (n-3)/2$, we have $J_{n}\cap \cA_{n}^{(s)}=\{0\}$.
         \qed
       \end{proposition}

       \subsection{The instanton homology of $Z_{n}$}

               The instanton homology $I(Z_{n};\Q)$ with rational
        coefficients was described, together with its ring structure,
        by Street \cite{Street} drawing on work of Boden \cite{Boden}
        and Weitsman \cite{Weitsman}. We summarize part of these
        results here, adapted to the case of $I(Z_{n})$ (by which we
        continue to mean the instanton homology with local
        coefficients). 

        The representation variety $\Rep(Z_{n})$ is a Morse-Bott
        critical locus for the Chern-Simons functional. By
        Lemma~\ref{lem:Morse-even}, there is a Morse function on
        $\Rep(Z_{n})$ with critical points only in even index. The
        proof of that lemma allows one to construct such a Morse
        function as a linear combination of traces of holonomies
        around loops in $Z_{n}$.  We may use such
        a Morse function as a holonomy perturbation for the
        Chern-Simons functional, so that the critical points of the
        perturbed Chern-Simons functional correspond to the critical
        points of the Morse function on $\Rep(Z_{n})$. After making such a perturbation, the
        set of critical points forms a natural basis both for the
        ordinary homology of $\Rep(Z_{n})$ as a $\Q$-vector space, and for the instanton
        homology $I(Z_{n})$ as an $\cR$-module. We therefore obtain an
        isomorphism
        \[
                I(Z_{n})= H_{*}(\Rep(Z_{n}))\otimes \cR.
        \]
        
        In the $\Z/4$ grading of the instanton homology, the minus component
        $\Rep(S^{2}_{n})_{-}$
        is shifted by $2$  relative
        to the plus component. This is established in \cite{Street}
        for rational coefficients, but the argument extends to any
        coefficients, including our local coefficient system
        $\Gamma$.  We record this in the following proposition.
 
        \begin{proposition}\label{prop:IZ-two-copies}
        As $\cR$-modules with $\Z/4$ grading, we have an isomorphism,
	\[
	      \Lam: I_{*}(Z_{n}) = H_*(\Rep(S^2_n);\cR) \oplus H_*(\Rep(S^2_n);\cR)[2]
	\]
        for all odd $n\ge 1$. In particular, the instanton homology is
        a free $\cR$-module and is non-zero only in even degrees mod
        $4$.
	\end{proposition}

        The isomorphism $\Lam$ in the above proposition depends on the choice
        of perturbation (at least a priori), because the isomorphism
        goes by identifying both sides with the free $\cR$-module
        generated by the critical points. The following two
        propositions
        add some additional structure. In the statement of the first
        proposition below, we write $\1_{+}\in I(Z_{n})$
        for the relative invariant of the 4-dimensional
            orbifold $D^{2}\times S^{2}_{n}$ with boundary $Z^{n}$:
            \[
            \1_{+} = I(D^{2}\times S^{2}_{n}).
            \]

        \begin{proposition}\label{prop:instanton-cyclic}
            The instanton homology $I(Z_{n})$ is
            a cyclic module for the filtered algebra $\cA_{n}$
            \eqref{eq:cA-newdef}, with cyclic generator the element
            $\1_{+}$. 
        \end{proposition}

        This proposition (whose proof is given below) prompts the
        following definition.
        
              \begin{definition}
        We write $\J{n}\subset \cA_{n}$ for the annihilator of the
        cyclic module $I(Z_{n})$, so that
        \[
                I(Z_{n}) \cong \cA_{n} /\J{n}.
        \]
        \end{definition}

        From this description, the instanton homology $I(Z_{n})$
        inherits an increasing filtration from the filtration of
        $\cA_{n}$:
        \[
                I(Z_{n})^{(m)} = ( \cA_{n}^{(m)} + \J{n}) / \J{n} .
        \]
        
        \begin{proposition}\label{prop:assoc-graded}
            The isomorphism $\Lam$ of
            Proposition~\ref{prop:IZ-two-copies} respects the
            filtrations, and the isomorphism on the associated graded
            is an isomorphism of 
            $\cA_{n}$-modules, independent of the choice of
            perturbations.
        \end{proposition}

We begin the proof of the two propositions above by describing the
$\cA_{n}$-module structure of $I(Z_{n})$.
Recall from
        that the $\cA_{n}$-module
        structure of 
        $H_{*}(\Rep(Z_{n});\cR)$ arises from operators $\alpha$,
        $\delta_{1},\dots, \delta_{n}$ (acting by cap product) and
        $\epsilon$.
        The instanton homology
        $I(Z_{n})$ carries parallel operators which we now make
        explicit.
        
        First, the classes $\alpha$, $\beta$ and $\delta_{p}$ in
        $H^{*}(\bonf^{*}(Z_{n});\Q)$ correspond to operators on the
        Floer homology $I(Z_{n})$ by the general construction
        \eqref{eq:IWa}. We write these operators as
	\begin{equation}\label{eq:ops-delta-beta}
       \begin{aligned} 
       \opalpha: I_{*}(Z_{n}) &\to I_{*-2}(Z_{n}),\\
       \opbeta : I_*(Z_{n}) &\to I_{*-4}(Z_{n}) = I_{*}(Z_{n})
       ,\\
        \opdelta_{p} : I_*(Z_{n}) &\to I_{*-2}(Z_{n}) , 
        \end{aligned}
	\end{equation}
        where the subscripts denote the mod 4 grading. In the notation
        of \eqref{eq:IWa}, these are the operators
        \[
              \begin{aligned}
                \opalpha &= I([0,1]\times Z_{n}, [S^{2}_{n}]) &&\\
                \opbeta &= I([0,1]\times Z_{n}, [w]), &\qquad [w]&\in
                H_{0}([0,1]\times Z_{n}), \\
                \opdelta_{p} &=  I([0,1]\times Z_{n}, [p]), &\qquad
                [p]&\in
                H_{0}([0,1]\times K_{n}).\\
              \end{aligned}
        \]
  
        \begin{remark}
        According to the results of
         \cite{Obstruction}, the operator $2 \opdelta_{p}$ can be realized
         as the map corresponding to a cobordism $W_1$ from $Z$ to $Z$,
         derived from the product cobordism $I\times Z$ by summing a
         standard torus to $I\times K$ at the point $(1/2, p)$. The
         local orientation of $K$ is used to fix a homology orientation
         of the torus.
         \end{remark}
         
 The counterpart of the operator $\epsilon$ is a special case of the
 construction of $I(W,a)^{e}$. Specifically, following Street
 \cite{Street}, it is
        the map \eqref{eq:IWa-e} in the special case that
        $W$ is the cylindrical cobordism, the element $a$ is $1$, and
        $e$ is the class $[\{\mathrm{point}\} \times S^{2}_{n}]$:
        \[
        \opepsilon = I\bigl([0,1]\times S^{2}_{n}\bigr)^{e}.
        \]

In order for the operators $\opalpha$, $\opdelta_{p}$ and $\opepsilon$ to
make the instanton homology $I(Z_{n})$ into a module over the algebra
$\cA_{n}$, we need to see that they satisfy the relations that are
baked into the definition of $\cA_{n}$. We turn to this next.
The relation in Proposition~\ref{prop:2-d-relation}
        specializes to the following:

        \begin{lemma}\label{lemma:delta-squared}
           With $\cR= \Q[\tau^{\pm 1}]$ as usual,
           the actions of the operators $\opdelta_{p}$ and $\opbeta$ on
           the $\cR$-module $I(Z_{n})$ are related by
           \[
                            \opdelta^{2}_{p} = -\opbeta + \tau^{2}  +
                            \tau^{-2}.
           \]
           In particular, $\opdelta^{2}_{p}$ is independent of the
           chosen point $p$ on the singular set of $Z_{n}$. \qed
        \end{lemma}

        The element $\epsilon$ in $\cA_{n}$ has square $1$ by
        definition, so we need the following lemma also.

        \begin{lemma}\label{lem:epsilon-squared}
        The operator $\opepsilon : I(Z_{n})\to I(Z_{n})$ has square
        $1$, and under the isomorphism of
        Proposition~\ref{prop:IZ-two-copies} it corresponds to the
        interchange of the two summands.
        \end{lemma}

        \begin{proof}[Proof of the lemma]
        This is proved in \cite{Street} for rational coefficients,
        except that an ambiguity in the orientation of the moduli
        spaces left the sign of $\opepsilon^{2}$ unresolved there.
        (See also the proof of Proposition~\ref{prop:leading-terms}
        below.) In our
        present context we have
        \[
        \begin{aligned}
                  \opepsilon^{2} &= I\bigl([0,1]\times S^{2}_{n}\bigr)^{e} \circ
                   I\bigl([0,1]\times S^{2}_{n}\bigr)^{e} \\
                   &= I\bigl([0,1]\times S^{2}_{n}\bigr)^{2e} \\
                   &= (-1)^{e\cdot e} I\bigl([0,1]\times S^{2}_{n}\bigr) \\
                   &= 1,
        \end{aligned}
        \]
        where the equality in the second line is by functoriality and
        the equality in the third line is from \eqref{eq:e-sign}.
        \end{proof}

        The relations in Lemmas~\ref{lemma:delta-squared} and
        \ref{lem:epsilon-squared} are the same relations
        satisfied by the elements $\epsilon$ and $\delta_{k}$ in the algebra
        $\cA_{n}$, so we can indeed use these
        operators to define an $\cA_{n}$-module structure on
        $I(Z_{n})$ by
        \begin{equation}\label{eq:a-to-alpha}
               \begin{aligned}
                \alpha &\mapsto \opalpha, \\
                \delta_{i} &\mapsto \opdelta_{i}, \quad i=1,\dots,n,\\
                \epsilon & \mapsto \opepsilon.
               \end{aligned}
        \end{equation}

Having described the module structure of $I(Z_{n})$, the fact that
it is a cyclic module generated by $\1_{+}$
(Proposition~\ref{prop:instanton-cyclic}) and the assertions of
Proposition~\ref{prop:assoc-graded} are both consequences of the fact
that, under the isomorphism of Proposition~\ref{prop:IZ-two-copies},
the operators $\opalpha$,
$\opdelta_{p}$ and $\opepsilon$ agree with the operators $\alpha$,
$\delta_{p}$ and $\epsilon$ on $H_{*}(\Rep(S^{2}_{n}))$ in their
leading terms. This is the assertion of the proposition below, which
is the final proposition of this subsection.

\begin{proposition}\label{prop:leading-terms}
    Let $\Lam$ be the isomorphism of
    Proposition~\ref{prop:IZ-two-copies}. Then for any $\xi\in
    I(Z_{n})^{(m)}$ and $u\in \cA_{n}^{(k)}$, we have
    \begin{equation}\label{eq:leading-comparison}
                \Lam (u\xi) = u\Lam(\xi) \quad \bmod
                I(Z_{n})^{(m+k-1)},
    \end{equation}
    and $\Lam(\1_{+})=1_{+}$.
\end{proposition}

\begin{proof}
   It is enough to verify \eqref{eq:leading-comparison} in the case that $u$ is one of the
   generators, $\alpha$, $\delta_{p}$ or $\epsilon$. The essential
   point is that $u\xi$ is defined using instantons on the cylinder
   $\R\times Z_{n}$ and that the leading term is defined by
   (perturbations of) the flat connections, while the non-leading
   terms are defined by instantons with positive action. 

   In more detail, let us write $\Rep(Z_{n}) =  R_{+} \cup R_{-}$, as
   an abbreviation for the components $\Rep(Z_{n})_{\pm}$. Before any
   perturbations are made, we have seen that the two
          components $R_{+} \cup R_{-}$ are copies
          of the representation variety $\Rep(S^{2}_{n})$ of the
          orbifold sphere (Proposition~\ref{eq:RepZn0}). For each
          $\kappa>0$, let us write
          \[
                M_{\kappa}(R_{\pm}, R_{\pm})
          \]
          for the moduli space of (unperturbed) instanton trajectories from one
          component of $\Rep(Z_{n})$ to another, with action $\kappa$.

          \begin{lemma}\label{lemma:dimensions-kappa}
            \begin{enumerate}
                \item The moduli spaces $M_{\kappa}(R_{+}, R_{+})$ and
                $M_{\kappa}(R_{-}, R_{-})$
                are
                non-empty only for $\kappa \in (1/2)\Z$.
                \item \label{item:kappa-quantized} The moduli spaces $M_{\kappa}(R_{+}, R_{-})$ and
                $M_{\kappa}(R_{-}, R_{+})$
                are
                non-empty only for $\kappa \in (1/2)\Z + (1/4)$.
                \item The formal dimension of the moduli space, in
                every case, is $8\kappa + (2n - 6)$.
            \end{enumerate}
          \end{lemma}

          \begin{proof}[Proof of the lemma]
          The moduli spaces $M_{\kappa}(R_{+}, R_{+})$ and
          $M_{\kappa}(R_{-}, R_{-})$ are non-empty when $\kappa=0$,
          consisting then of constant trajectories on the cylinder and
          forming a regular moduli space of dimension $2n-6$ (the
          dimension of the representation variety). For other values
          of $\kappa$, these moduli spaces are related to each other
          by glueing in instantons and monopoles, which
          will change $\kappa$ by multiples of $1/2$ while always
          changing the formal dimension by $8\kappa$
          \cite{KM-gtes1,KM-yaft}.

          The formal dimension and action $\kappa$ for the moduli
          spaces $M_{\kappa}(R_{+}, R_{-})$ and $M_{\kappa}(R_{-},
          R_{+})$ are the same as for moduli spaces on the closed
          bifold $S^{1}\times Z_{n} = T^{2}\times S^{2}_{n}$ for a
          bundle with marking data where $w_{2}(E)$ is dual to the
          class $T^{2}\times\{\mathrm{point}\}$. The action in this
          case is equal to $n/4$ modulo $1/2$, or in other words
          belongs to $(1/4) + (1/2)\Z$ since $n$ is odd. (In the
          language of \cite{KM-gtes1}, the monopole number on each
          of the $n$ components of the singular set is a
          half-integer.) The formula for the formal dimension in terms
          of the action $\kappa$ is unchanged.
          \end{proof}

 After perturbation of the Chern-Simons functional, the manifolds
 $R_{+}$ and $R_{-}$ each become a finite set of non-degenerate
 critical points, $\Crit_{+}$ and $\Crit_{-}$. The action of the
 perturbed instantons will be close to integer multiples of $1/4$ is
 the perturbation is small, so for critical points $c$ and $c'$ and
 $\kappa\in (1/4)\Z$ we continue to write $M_{\kappa}(c,c')$ for the
 perturbed moduli spaces. We have the dimension formula
 \[
            \dim M_{\kappa}(c,c') = 8\kappa + \ind(c) - \ind(c')
 \]
where $\ind$ denotes the ordinary Morse index for the Morse function
on $R_{\pm}$. Furthermore, the moduli space is non-empty only if
$\kappa\in (1/2)\Z$ in the case that $c, c'$ both belong to
$\Crit_{+}$ or to $\Crit_{-}$, and only if $\kappa\in(1/4) + (1/2)\Z$
otherwise.

Consider now the operator $\opalpha$ for example. (The case of
$\opdelta_{p}$ is no different.) When $\kappa=0$, the moduli space
$M_{0}(c,c')$ between critical points $c,c'\in \Crit_{+}$ or $c,c'\in
\Crit_{-}$ coincides with a perturbation of the space of ordinary
Morse trajectories between the critical points in $R_{\pm}$. The
construction of $\opalpha$ means that we can write it as a sum
\begin{equation}\label{eq:opalpha-expansion}
        \opalpha = \sum_{\substack{
        \kappa\in (1/4)\Z \\ \kappa
        \ge 0}
}        \opalpha_{(\kappa)}
\end{equation}
according to the contributions of the different moduli spaces
$M_{\kappa}$.  The matrix entry of $\opalpha_{(0)}$ is the evaluation
of the cohomology class $\alpha$ on the Morse trajectory space
$M_{0}(c,c')$ between critical points on $R_{+}$ or $R_{-}$ with
$\ind(c)-\ind(c')=2$. This is the cap product by the class $\alpha$,
under the isomorphism between Morse homology and singular homology.
Thus we have
\[
                \Lam (\opalpha_{(0)}\xi) = \alpha\Lam(\xi)
\]
where $\xi$ is the class corresponding to the critical point $c$. The
dimension formula shows that the remaining terms
$\Lam(\opalpha_{(\kappa)}\xi)$ for positive $\kappa$ correspond to
2-dimensional moduli spaces $M_{\kappa}(c,c'')$ where the index
difference $\ind(c)-\ind(c'')$ is $4$ or more.

   In the case of
   $\opepsilon$, the equality \eqref{eq:leading-comparison} holds
   exactly. This is the content of Lemma~\ref{lem:epsilon-squared}. In
   the present context it can be understood by the same argument as
   applies to $\opalpha$ and $\opdelta_{p}$, but with the additional
   observation that the moduli spaces of positive action contribute
   zero because of action of translation on these moduli spaces.
\end{proof}

       If we keep track of the difference between $R_{+}$ and $R_{-}$
       which is
       highlighted in part \ref{item:kappa-quantized} of
       Lemma~\ref{lemma:dimensions-kappa}, then we can extract a
       slightly more detailed statement from the proof of the
       proposition above.        Recall that $J_{n}\subset \cA_{n}$ is the annihilator
        of $H_{*}(\Rep(Z_{n})$. (See
        Proposition~\ref{prop:coho-Z-cyclic}.) In
        the following corollary, we also write
        \[
                \cA_{n}^{+}\subset \cA_{n}
        \]
        for the subalgebra generated over $\cR$ by $\alpha$ and
        $\delta_{1},\dots, \delta_{n}$, so that
        \[
                \cA_{n} = \cA_{n}^{+} + \epsilon\cA_{n}^{+}.
        \]

        \begin{corollary}\label{cor:relations-coh-to-inst}
            For any element $w\in J_{n}\cap \cA_{n}^{(m)}$, there
            exists $\omega\in \cJ_{n}\cap \cA_{n}^{(m)}$ with
            \[
                            \omega - w \in \cA_{n}^{(m-1)}.
            \]
            More particularly, if $w$ is a homogeneous element of
            degree $m$ in the graded algebra $\cA_{n}$, then $\omega$
            can be taken to have the form
            \[
            \begin{aligned}
                        \omega = w(0) &+ w(2) + w(4) + \cdots\\
                        &\null + \epsilon( w(1) +  w(3) +
                        \cdots).
            \end{aligned}
            \]
            where $w(0)=w$ and $w(i)\in \cA_{n}^{(m-i)}\cap \cA_{n}^{+}$ is
            homogeneous of degree $m-i$ for all
            $i$.
            Furthermore, if $m\le (n-1)/2$, then $\omega$ is uniquely
            determined by $w$. 
        \end{corollary}

        \begin{proof}
         This follows from the proposition above and
         Proposition~\ref{prop:below-middle}.
          \end{proof}

         \subsection{The instanton homology of $Z_{n,-1}$}

        We now examine the bifold $Z_{n,-1}$
        (see Definition~\ref{def:Znq}). The singular locus
        $K(Z_{n,-1})$ in this
        case is a knot in $S^{1}\times S^{2}$, with winding number
        $n$. We still require $n$ to be odd, so that this is an
        admissible bifold. We can view $K(Z_{n,-1})$ as the closure
        of a braid in $S^{1}\times D^{2}\subset S^{1} \times S^{2}$
        whose braid diagram has $n-1$ negative crossings. There is
        therefore a cobordism  $W$ of bifolds, from $Z_{n,-1}$ to
        $Z_{n}$, obtained by smoothing each of the crossings.
        We can write $W$ as a composite
        of $(n-1)$ cobordisms, $W_{1}, \dots, W_{n-1}$, in the order
        illustrated in Figure~\ref{fig:cobordisms}.
\begin{figure}
    \begin{center}
       \includegraphics[width=5in]{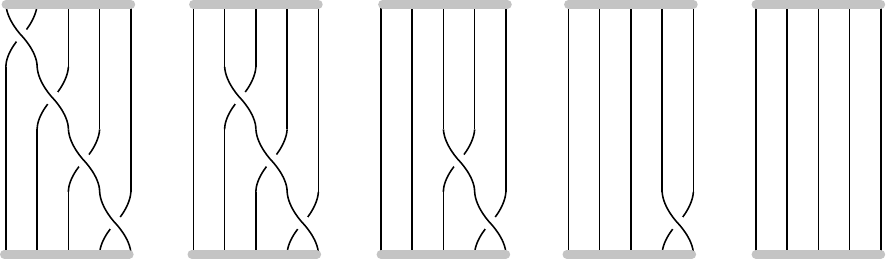}
    \end{center}
    \caption{\label{fig:cobordisms} The composite cobordism from
    $Z_{n,-1}$ to $Z_{n}$, illustrated for $n=5$.}
\end{figure}
The intermediate
        bifolds  each correspond to braids with $k$ ``straight''
        strands and $n-k$ braided strands: a side-by-side
        juxtaposition of $Z_{k}$ and $Z_{n-k,-1}$, which we
        temporarily denote by $Z_{k}* Z_{n-k,-1}$ (with the
        understanding that $Z_{0}$ is $S^{1}\times S^{2}$ with an
        empty link). So we have
        \[
                   I( W_{k} ) :I( Z_{k-1}* Z_{n-k+1,-1})\to
                    I(Z_{k}*Z_{n-k,-1} ),\qquad (k=1,\dots,n-1).
        \]
        (Note that, when $k=n-1$, we have $Z_{k}* Z_{n-k,-1} \cong
        Z_{n}$.)

        \begin{proposition}\label{prop:inclusion-free}
            For each odd $n$ and each $k \le n-1$, the induced map
            $I(W_{k})$ is an inclusion of one free $\cR$-module in
            another, as a direct summand.
        \end{proposition}

        \begin{proof}[Proof of Proposition~\ref{prop:inclusion-free}]
            As an inductive hypothesis, let us suppose that
            $I(Z_{j}*Z_{n'-j,-1} )$ is a free $\cR$-module for all
            odd $n'<n$ and all $j\le n'-1$. We assume also that, in
            this range, the module is supported in even degrees in the
            mod 4 grading.             
            All this is true when $n=3$,
            because the groups referenced in the hypothesis are all
            zero. We also recall that $I(Z_{n})$ is free
            and supported in even gradings.

            The cobordism $W_{k}$ is
            one map in a skein exact triangle \cite{KM-ibn1,
            KM-unknot}, in which the third
            instanton homology group is $I(X_{n,k})$, where $X_{n,k}$
            is a braid as shown in Figure~\ref{fig:thirdbraid}.
            \begin{figure}
    \begin{center}
       \includegraphics{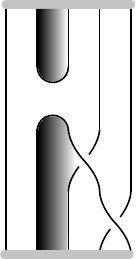}
    \end{center}
    \caption{\label{fig:thirdbraid} The third braid $X_{n,k}$ in the exact
    triangle, illustrated in the case $n=5$ and $k=2$. The shaded region (which is connected in a projection of
    $S^{1}\times S^{2}$) can be eliminated by a Reidemeister-I move.}
\end{figure}
            Thus,
            \[
                \dots \to    I( Z_{k-1}* Z_{n-k+1,-1})\to
                    I(Z_{k}*Z_{n-k,-1} ) \to
                I(X_{n,k}) \to \dots
            \]
            is a long exact sequence.
            
            After
            an isotopy, we have, for $k\le n-2$, 
            \begin{equation}\label{eq:Xnk}
                    X_{n,k} = 
                           Z_{k-1} *Z_{n-2-k+1,-1}, \qquad k \le n-2 .
            \end{equation}
            From our inductive hypothesis, $I(X_{n,k})$ is free
            in this range. The case $k=n-1$ is slightly
            different: in this case $X_{n,n-1}$ is the connected sum
            of $Z_{n-2}$ and the bifold obtained from an unknot in
            $S^{3}$. (See Figure~\ref{fig:thirdbraid} again.) From
            another application of the skein triangle, we have an
            exact sequence
            \[
                \dots \to    I(Z_{n-2}) \to I(X_{n,n-1}) \to
                I(Z_{n-2}) \to \dots.
            \]

            All of these exact sequences are sequences of
            $(\Z/2)$-graded  modules, in which just one of the three maps always
            has odd degree. We therefore have short exact sequences,
            \begin{equation}\label{eq:Xnnminus1}
               0 \to    I(Z_{n-2}) \to I(X_{n,n-1}) \to
                I(Z_{n-2}) \to 0,
            \end{equation}
            and 
            \begin{equation}\label{eq:short-exact}
                0\to    I( Z_{k-1}* Z_{n-k+1,-1})\to
                    I(Z_{k}*Z_{n-k,-1} ) \to
                I(X_{n,k}) \to 0.
            \end{equation}
            From \eqref{eq:Xnnminus1}, we see that $I(X_{n,k})$ is
            free when $k=n-1$. It follows that the sequence
            \eqref{eq:short-exact} splits when $k=n-1$. We already
            observed that $I(X_{n,k})$ is free for $k<n-1$, so all the
            sequences split and all the maps $I( Z_{k-1}*
            Z_{n-k+1,-1})\to I(Z_{k}*Z_{n-k,-1} )$ are split
            inclusions of free modules.
        \end{proof}

        \begin{proposition}\label{prop:lattice-points}
            The representation variety of $Z_{n,-1}$ is
            non-degenerate and
            consists of $(n^{2}-1)/4$ points.
           \end{proposition}

        \begin{proof}
            The orbifold $Z_{n,-1}$ is a fiber bundle over the circle,
            with fiber the orbifold sphere $S^{2}_{n}$. The
            restriction map to the fiber,
            \[
                \Rep(Z_{n,-1}) \to \Rep(S^{2}_{n}),
            \]
            has image the set of representations in $\Rep(S^{2}_{n})$
            which are invariant under the action $h_{*}$ of the
            monodromy of the circle bundle, $h: S^{2}_{n} \to
            S^{2}_{n}$. The latter is the map which rotates the sphere
            through $2\pi/n$. The restriction map is two-to-one, just
            as it is for $Z_{n}$, and for the same reason.

            The fixed points of $h_{*}$ are representations of the
            orbifold fundamental group of the quotient $\Sigma =
            S^{2}_{n}/\langle h \rangle$. This orbifold surface has
            one orbifold point of order $2$ and two orbifold points of
            order $n$. For a spherical orbifold with three singular
            points, the representation variety consists of isolated
            points, and this is essentially the situation considered
            in \cite{FS} (for example). The enumeration of
            representations, as in \cite{FS}, becomes an enumeration
            of lattice points in a region. (The same conclusion can
            also be reached by identifying the representations  with
            stable parabolic bundles on a curve of genus 0 with
            appropriate parabolic structure at the orbifold points.
            See section~\ref{subsec:loci})
            In this particular case, the number of representations of
            the orbifold fundamental group of $S^{2}_{n}/\langle h
            \rangle$ is $(n^{2}-1)/8$, and $\Rep(Z_{n,-1})$ therefore
            consists of $(n^{2}-1)/4$ points. The non-degeneracy of
            the former leads to the non-degeneracy of the latter.
        \end{proof}

        The following corollary summarizes the conclusions of the
        previous two propositions.
        
        \begin{corollary}\label{cor:rank-Znm1}
             The instanton homology $I(Z_{n,-1})$ with local
            coefficients is a free $\cR$-module of rank $(n^{2}-1)/4$,
            supported in even degrees mod 4. 
            The cobordism $W: Z_{n,-1} \to Z_{n}$ induces a map
            $I(W)$ on instanton homology with local coefficients,
\[
            I(W) : I(Z_{n,-1}) \to I(Z_{n})
\]
            which
            is an inclusion of this free $\cR$-module as a direct
            summand. \qed
        \end{corollary}

      The bifold obtained from $Z(n,-k)$ by
       reversing the orientation is $Z(n,k)$, and by dualizing the
       above corollary we obtain:

           \begin{corollary}\label{cor:rankZn1}
             The instanton homology $I(Z_{n,1})$ with local
            coefficients is also a free $\cR$-module of rank
            $(n^{2}-1)/4$. 
            The cobordism $W^{\dag}: Z_{n} \to Z_{n,1}$ induces a
            surjective map $I(W^{\dag})$ on these free modules. \qed
        \end{corollary}
        
         On the other hand, we have
         Lemma~\ref{lem:orientation-Knq} which identifies $Z_{n,-1}$
         and $Z_{n,1}$ in an orientation-preserving manner by an
         isotopy.
         So we have another variant of the corollary:

          \begin{corollary}\label{cor:surjectZ}
             There is a surjective homomorphism of free $\cR$-modules
             from $I(Z_{n})$ to $I(Z_{n,-1})$ obtained from a
             cobordism between the links $K(Z_{n})$ and
             $K(Z_{n,-1})$ inside $[0,1]\times S^{1}\times S^{2}$. 
          \qed
        \end{corollary}

        Like $Z_{n}$, the bifold $Z_{n,-1}$ contains a copy $S$ of
        the orbifold sphere $S^{2}_{n}$ intersecting the singular
        locus in $n$ points. By the general constructions
        of section~\ref{subsec:operators}, this gives rise to operators
        $\opalpha$, $\opdelta_{1},\dots, \opdelta_{n}$ and $\opepsilon$,
        acting on $I(Z_{n,-1})$ just as in the case of $I(Z_{n})$,
        making $I(Z_{n,-1})$ also an $\cA_{n}$-module. Note that the $n$
        points of intersection with $S$ all lie on the same component
        of the singular locus $K(Z_{n,-1})$ (which is now a knot, not a link). The
        operators $\opdelta_{p}$ are therefore all equal on
        $I(Z_{n,-1})$, and we will sometimes write this operator as
        $\opdelta$.

        \begin{proposition}\label{prop:Zn1-as-quotient}
            With the instanton module structure in which $\alpha, \delta_{i},
            \epsilon\in \cA_{n}$ act by the operators $\opalpha$, $\opdelta$
            and $\opepsilon$, the instanton homology $I(Z_{n,-1})$ is a cyclic module
            for the algebra $\cA_{n}$ and can therefore be described
            as a quotient,
                        \[
                    I(Z_{n,-1}) \cong \cA_{n} / \JOne{n}.
             \]
            The ideal $\JOne{n}$ contains the ideal $\J{n}$ as well
            as the elements $\delta_{i}-\delta_{j}$.
        \end{proposition}

        \begin{proof}
            We have seen that there is a cobordism from $Z_{n}$ to
            $Z_{n,-1}$ inducing a surjection on instanton homology
            (Corollary~\ref{cor:surjectZ}). The proposition follows
            from this and the above remark that the actions of the
            $\opdelta_{i}$ are all equal.
        \end{proof}

        It is helpful here to introduce the smaller algebra
        \[
                \bar\cA = \cA_{n} / \langle
                \delta_{i}-\delta_{j}\rangle_{i,j}
        \]
        which we can write simply as
        \begin{equation}\label{eq:barA-def}
                    \bar\cA = \cR[\alpha,\delta,\epsilon] /
                    \langle \epsilon^{2}-1\rangle,
        \end{equation}
        where $\delta$ denotes the image of the $\delta_{i}$ in the
        quotient ring. The algebra $\bar\cA$ described this way is
        independent of $n$.
        The above proposition then can be recast as,
        \begin{equation}\label{eq:Zn1-raw}
                                     I(Z_{n,-1}) \cong \bar \cA /
                                     \barJOne{n},
        \end{equation}
        where $\barJOne{n}$ is the image of $\JOne{n}$ in
        $\bar\cA$.

        Our main goal in this paper is to identify $I(Z_{n})$ and $I(Z_{n,-1})$
        completely, by describing the ideals $\J{n}\subset \cA_{n}$ and $\barJOne{n} \subset
        \bar\cA$. In particular, as described in the introduction, we will eventually provide a set of
        generators of $\barJOne{n}$ in closed form, as minors of an
        explicit matrix.

          \section{Relations in ordinary cohomology}
            \label{sec:relations-ordinary}
          
          \subsection{Loci in families of parabolic bundles on
          $S^{2}_{n}$}
          \label{subsec:loci}
          
          Recall from Proposition~\ref{prop:coho-Z-cyclic} the description of
          the cohomology ring of the representation variety
          \[
               \Rep(Z_{n}) = \Rep(S^{2}_{n}) \cup \Rep(S^{2}_{n})
          \]
          as a quotient $\cA_{n}/J_{n}$, where $J_{n}$ is an ideal.
          (The coefficient ring here, as in
          Proposition~\ref{prop:coho-Z-cyclic}, is $\cR$, though at
          this point our calculations will involve only $\Q$, so
          rational coefficients would suffice.) The betti numbers of
          $\Rep(S^{2}_{n})$ were calculated recursively by Boden
          \cite{Boden}, and a full presentation of the cohomology ring
          (in a more general case) is described in \cite{Earl-Kirwan}.
          Generators for the ideal of relations in the specific case
          of $\Rep(S^{2}_{n})$ are given by Street \cite{Street}. We
          shall describe a particular source of such relations,
          arising from a mechanism first pointed out by Mumford
          in the smooth case \cite{AB}. (In \cite{Earl-Kirwan} it
          is shown that essentially the same mechanism gives rise to a
          complete set of relations in the orbifold case.) 

          As stated earlier, although we have taken $\SO(3)$
          connections as our starting point, the representation
          variety $\Rep(S^{2}_{n})$ can be identified with the space
          of flat $\SU(2)$ connections having monodromy of order $4$
          at each of the $n$ punctures. In turn, this representation
          variety can be identified
          with a moduli
          space of stable parabolic bundles by the results of
          \cite{Mehta-Seshadri}. We adopt the following conventions to
          make this more specific in
          the rank-2 case, following \cite{KM-gtes1, KM-gtes2}.

          We consider a compact Riemann surface $S$ equipped with a
          set of
          distinguished points $\pi = \{p_1,\dots, p_n\}$, and a parameter
          $\alpha\in(0,1/2)$. Given a fixed holomorphic line bundle
          $\Theta\to S$ (usually trivial in our case), we study rank-2
          holomorphic bundles $\cE\to S$ with $\Lambda^2\cE=\Theta$,
          together with a filtration of the rank-2 fiber at each
          $p\in\pi$
          determined by a choice of a one-dimensional subspace (a line)
          $\cL_p \subset \cE_{p}$. The data $(\cE,
          \cL_{p_{1}},\dots,\cL_{p_{n}},\alpha)$ is a bundle with parabolic
          structure. Given a line subbundle $\cF\subset \cE$, the
          \emph{parabolic degree} of $\cF$ is defined by
	\begin{equation}\label{eq:pdeg-def}
	           \pdeg \cF = c_1(\cF)[S] + \sum_{\pi}\pm \alpha
	\end{equation}
        where we take $+\alpha$ in the sum when $\cF$ contains $\cL_p$
        at $p$ and $-\alpha$ when it does not. The parabolic bundle
        is \emph{semi-stable} if \[ \pdeg \cF \le 1/2\deg\Theta \] for
        every line subbundle $\cF$, and is \emph{stable} if strict
        inequality holds. At present we will take $\Theta$ to be
        trivial and \emph{we are only concerned with the special case
        $\alpha=1/4$}. In this case, when $n$ is odd, all semi-stable
        bundles are strictly stable, and the moduli space of stable
        parabolic bundles is a projective variety of complex dimension
        $3g-3+n$. In the case of genus 0, we write $\cM(S^{2}_{n})$
        for this projective variety: the moduli space of stable
        parabolic bundles, with parabolic structure at the $n$ marked
        points and $\alpha=1/4$.

        With this notation understood, the theorem of
        \cite{Mehta-Seshadri} identifies the representation variety
        $\Rep(S^{2}_{n})$ for odd $n$ with the moduli space of stable
        parabolic bundles:
         \[
           \Rep(S^{2}_{n}) \cong \cM(S^{2}_{n}).
         \]

        Suppose now that we have a family of parabolic bundles on
        $S^{2}_{n}$ parametrized by a space $T$. This means that we
        have a rank-2 bundle,
        \[
           \fE  \to T\times S^{2}
        \]
	with $\Lambda^{2}\fE \cong \Phi \boxtimes \Theta$ (with $\Theta$
        still trivial on $S^{2}$ at the moment, but $\Phi$ a
        non-trivial line bundle on the base $T$), together with line
        subbundles
        \[
               \fL_{p}\subset \fE |_{T\times p}, \quad p\in\pi.
        \]
        The bundle $\fE$ is equipped with a holomorphic structure on
        each $\{t\} \times S^{2}$, giving rise to parabolic bundles
        $\cE_{t}$.

        In such a family over $T$, we can consider the locus of those
        $t\in T$ where the parabolic bundle $\cE_{t}$ is unstable (for
        $\alpha=1/4$).
        From the definition at~\eqref{eq:pdeg-def}, being
        unstable means the following.
        \begin{enumerate}
            \item \label{item:filt1} We have a holomorphic line bundle $\cF\to S^{2}$, of
            degree degree $f$ say, necessarily the bundle $\cO(f)$.
            \item We have a subset $\eta \subset \pp$, whose cardinality we denote by $h$.
            \item There is a non-zero holomorphic map $\iota : \cF\to
            \cE_{t}$ such that $\iota(\cF|_{p}) \subset \cL_{t}|_{p}$
            for all $p\in \eta$.
            \item \label{item:degree} We have $f + (1/4)(2h-n) > 0$.
        \end{enumerate}

        Altering this slightly, given any $\lambda\in \R$, we make the
        following definition.

        \begin{definition}\label{def:loci-def}
          Let $\eta\subset \pp = \{p_{1},\dots, p_{n}\}$ be any
          subset, and write $h=|\eta|$ for its cardinality. Let
          $\lambda$ be an odd multiple of $1/4$ satisfying the
          additional constraint that
          \begin{equation}\label{eq:h-parity}
             h   = (n-4\lambda)/2 \pmod{2}.
        \end{equation}
        This being so, there is
        $f\in \Z$ such that
            \begin{equation}\label{eq:constraint-lambda}
                    f + (1/4)(2h - n) = -\lambda. 
            \end{equation}
            Let $\cF\to S^{2}$ be the line bundle $\cO(f)$. Given a
            family of parabolic bundles on $S^{2}_{n}$ parametrized by
            $T$ as above, we define
            \begin{equation}\label{eq:T-stratum}
                    T^{\eta}_{\lambda}\subset T
            \end{equation}
            to be the locus of points $t\in T$ such that there is a
            non-zero
            holomorphic map $\iota: \cF \to \cE_{t}$ with
            $\iota(\cF|_{p}) \subset \cL_{t}|_{p}$ for all $p\in
            \eta$.
        \end{definition}
        This definition is set up so that
          the unstable locus is the union
         \[
                \bigcup_{\lambda\le -1/4} T^{\eta}_{\lambda}.
         \]

        The definition of the locus $T^{\eta}_{\lambda}$ is
        readily rephrased as the statement that a certain Fredholm
        operator $P_{t}$ (determined by the parabolic bundle $\cE_{t}$
        and the choice of $\lambda$ and $\eta$) has non-zero kernel.
        If we suppose that the resulting map
        \[
            P : T \to \mathrm{Fred}
        \]
        is transverse to the stratification of the space of Fredholm
        operators by the dimension of the kernel, then the locus
        $T_{\lambda}^{\eta}\subset T$ will itself be a stratified space
        whose Poincar\'e dual is a cohomology class that one can
        calculate using the index theorem for families. With slight
        abuse of notation, we write \eqref{eq:T-stratum} as
        \[
              T_{\lambda}^{\eta} =  T \cap U_{\lambda}^{\eta},
        \]
        where $U_{\lambda}^{\eta}$ denotes the locus where the
        Fredholm operator has kernel.
        It will also be useful to group together the different subsets
        $\eta$ according to their size $h = |\eta|$, so that we write
        (with a slight further abuse of notation),
        \[
        \begin{gathered}
        U_{\lambda}^{h} = \bigcup_{|\eta|=h} U_{\lambda}^{\eta}\\
        T_{\lambda}^{h} = T \cap U_{\lambda}^{h}. 
        \end{gathered}
        \]
        Again, this locus is non-empty only if $h$ satisfies
        the parity condition \eqref{eq:h-parity}.

        We now compute the Chern classes of the index of the family of
        operators $P$ in order to derive a formula for the class dual
        to the stratum $T^{\eta}_{\lambda}$. Note that if $P$ is a
        family of complex Fredholm operators of index $-k+1$, then
        (assuming transversality) the locus where $P_{t}$ has kernel
        is dual to
        \begin{equation}\label{eq:cp}
            c_{k}(-\ind(P)) \in H^{2k}(T).
        \end{equation}
        (This is the first case of Porteous's formula in the case of
        Fredholm maps
        \cite{Porteous,Koschorke}.)
        
        It is evident from the definition that the locus
        $T^{\eta}_{\lambda}$ is unchanged if the family of bundles
        $\fE$ is modified by tensoring with a line bundle pulled back
        from the base $T$. Recall that we have written
        $\Lambda^{2}\fE= \Phi\boxtimes\Theta$, where $\Phi\to T$ is a
        line bundle and $\Theta$ is taken to be trivial. If $\Phi$
        has a square root, we may tensor by $\Phi^{-1/2}$ to make
        $c_{1}(\fE)=0$. Although a square root will not exist in
        general, the calculation below is not invalidated by assuming
        that $c_{1}(\fE)=0$, and we will make this simplification from
        here on. This means in particular that $c_{2}(\fE) =
        -p_{1}(\ad \fE)/4$.
         Let us then write
         \[
         c_{2}(\fE) =
               \beta\times 1 + \hat \alpha \times
                v \in H^{4}(T\times S^{2}),
         \]
        where $v$ is the unit volume form on $S^{2}$. From the
        binomial theorem, we have
        \begin{equation}\label{eq:c2-powers}
         c_{2}(\fE)^{r} =
                 \beta^{r}\times 1 + r \hat{\alpha} \beta^{r-1} \times
                v.
        \end{equation}
        The class $\hat{\alpha}$ here does not quite correspond to the
        class $\alpha$ in \eqref{eq:AB-generators}, because the latter was
        defined using the orbifold Pontryagin class. The relation
        between the two is:
        \begin{equation}\label{eq:a-to-tilde-a}
                \hat \alpha = \alpha - \frac{1}{2} \sum_{p\in\pi}
                \delta_{p}.
        \end{equation}
        
        For each $p\in \pp$ we also have the line subbundle
        $\fL_{p}$ and the quotient line bundle $\fQ_{p}=
        (\fE|_{T\times p})/\fL_{p}$, and from these we
        obtain the cohomology class
        \[
             \delta_{p} = \frac{1}{2}(c_{1}(\fQ_{p}) - c_{1}(\fL_{p}))
        \]
        The definition is set up so that $\delta_{p}$
        coincides with the Euler class of the oriented rank-2
        subbundle of $\ad(\fE|_{p\times T})$ determined by
        $\fL_{p}$).

        Fix a holomorphic line bundle $\cF \cong \cO(f)$ on $S^{2}$.
        We are seeking a non-zero holomorphic map $\iota:\cF \to \cE_{t}$ such
        that the composite with the quotient map,
        \[
             \cF \to \cE_{t} \to \cQ_{(t,p)},
        \]
        vanishes for all $p \in \eta$. So, for the family of Fredholm
        operators $P$ that we are interested in,
        \[
            \ind(P) = \ind(\bar\partial_{\cF^{*}\otimes \fE}) -
            \sum_{p\in \eta}
            [\fQ_{p}],
        \]
        where the first part is the ordinary family $\bar\partial$
        operators. From the index theorem for families, we have
        \begin{equation}\label{eq:ch-index}
            \ch(\ind(P)) = \left(\left( \mathrm{Todd}(S^{2}) \cupprod
            \ch(\cF^{*}\otimes \fE) \right) \big/ [S^{2}] \right)-
            \sum_{p\in\eta} \ch(\fQ_{p}).
        \end{equation}

        To compute the Chern characters that appear on the right-hand
        side of this formula,
         we introduce formal Chern roots $\pm \rho \in H^{2}(T\times
         S^{2};\Q)$ so that $c_{2}(\fE)=-\rho^{2}$. Then we can
         write $\ch(\fE)=e^{-\rho}+e^{\rho} = 2\cosh(\sqrt{-c_{2}(\fE)})$,
         and a short calculation using \eqref{eq:c2-powers} yields
         \[ 
                    \ch(\fE) = 2 \cosh(\sqrt{-\beta}) - v
                    \frac{\sinh(\sqrt{-\beta})}{\sqrt{-\beta}}\hat{ \alpha}.
         \]
         We also have
        \[
         \ch(\cF^{*}) = 1- f\,v,
        \]         
        and
        \[
          \ch(\fQ_{p}) = e^{\delta_{p}}.
        \]
        Finally on the right-hand side of \eqref{eq:ch-index} we have
        $\mathrm{Todd}(S^{2}) = 1 + v$. Assembling these and
        calculating the slant product by $[S^{2}]$, we find
        \[
               \ch\bigl(\ind(P)\bigr) = (2-2f-h) \cosh(\sqrt{-\beta}) -
                  \frac{\sinh(\sqrt{-\beta})}{\sqrt{-\beta}}\left(
                  \hat \alpha + \sum_{p\in
                  \eta} \delta_{p}\right),
        \]
        where $h$ is the number of elements of $\eta$.
        If we use the fact that we are assuming equality in item
        \ref{item:degree} above, and if we substitute $\alpha$ for $\hat \alpha$
        using the relation \eqref{eq:a-to-tilde-a}, we obtain:
        \begin{equation}\label{eq:ch-minus-ind}
        \begin{aligned}
             \ch\bigl(-\ind(P)\bigr) = (n/2-2\lambda -2 ) & \cosh(\sqrt{-\beta})\\
                  &\null +
                  \frac{\sinh(\sqrt{-\beta})}{\sqrt{-\beta}}\left( \alpha +
                  \frac{1}{2}\sum_{p\in
                  \eta} \delta_{p} - \frac{1}{2}\sum_{p\not\in
                  \eta} \delta_{p}\right).
       \end{aligned}            
        \end{equation}
        If we recall that $\delta_{p}^{2}=-\beta$ for all $p$, then we can
        equivalently write this formula as
              \begin{equation}\label{eq:ch-minus-ind-d1}
        \begin{aligned}
             \ch\bigl(-\ind(P)\bigr) &= (n/2-2\lambda -2 )  \cosh(\delta_{1})\\
                 & \qquad\qquad \null +
                 \frac{\sinh(\delta_{1})}{\delta_{1}}\left( \alpha +
                  \frac{1}{2}\sum_{p\in
                  \eta} \delta_{p} - \frac{1}{2}\sum_{p\not\in
                  \eta} \delta_{p}\right), \\
      \end{aligned}            
        \end{equation}
        or in abbreviated form as
       \begin{equation}\label{eq:ch-minus-ind-abbrev} 
                 \ch\bigl(-\ind(P)\bigr) = \iA_{\lambda}\cosh(\delta_{1}) +
                  \frac{\sinh(\delta_{1})}{\delta_{1}} B_{\eta},
   \end{equation}
        where $\iA_{\lambda}$ and $B_{\eta}$ are the indicated
        subexpressions of \eqref{eq:ch-minus-ind-d1}. Note that
        $\iA_{\lambda}$ is minus the numerical index of $P$.
        
        The above formula defines a graded infinite sum of elements of
        the algebra
        \[
        \begin{aligned}
        A_{n} &= \Q[\alpha, \delta_{1}, \dots, \delta_{n}]\bigm/ \bigl\langle 
                 \delta_{i}^{2}-\delta_{j}^{2}\bigr\rangle_{i,j} \\
                 &=
                 H^{*}(\bonf^{*}(S^{2}_{n});\Q)
                 \end{aligned}
                 \]         
        (see Definition~\ref{def:lower-jn}), thus an element of the formal
        completion \[
         \widehat
         H^{*}(\bonf^{*}(S^{2}_{n});\Q)\supset
         H^{*}(\bonf^{*}(S^{2}_{n});\Q). \]
        By the usual formulae
        expressing elementary symmetric polynomials in terms of power
        sums, there is a map
        \[
               \mathfrak{c}_{k} :\widehat
                  H^{*}(\bonf^{*}(S^{2}_{n});\Q) \to
                  H^{2k}(\bonf^{*}(S^{2}_{n});\Q)
        \]
        such that $\mathfrak{c}_{k}(\ch(V)) = c_{k}(V)$ for any $V$, 
        and so we have explicit formulae for
        \[
                       c_{k}(-\ind(P)) \in  H^{*}(\bonf^{*}(S^{2}_{n});\Q),
        \]
        given as $\mathfrak{c}_{k}(r)$, where $r$ is the right-hand
        side of \eqref{eq:ch-minus-ind-d1}. 
        The case we are interested in from \eqref{eq:cp} is the Chern
        class $c_{k}$, where $-k+1$ is the numerical index of $P$.
        From the constant term in the formula for the Chern character
        above, we read
        \begin{equation}\label{eq:lambda-to-k}
                  k = n/2  -2 \lambda - 1.
        \end{equation}
        So we make the following definition.
        
        \begin{definition}\label{def:omega}
            Given $\lambda$ an odd multiple of $1/4$ and given a
            subset $\eta\subset \pi= \{p_{1},\dots, p_{n}\}$ of size $h$,
            where $h$ satisfies the parity condition
            \eqref{eq:h-parity}, let $k$ be the integer given by
            \eqref{eq:lambda-to-k}, and denote by
            \[
                        \oo^{k}_{n,\eta} \in
                        H^{*}(\bonf^{*}(S^{2}_{n});\Q) \subset \cA_{n}
            \]
            the element $\mathfrak{c}_{k}(r)$, where $r$ is the
            right-hand side of 
            \eqref{eq:ch-minus-ind-d1}. 
        \end{definition}

        To illustrate the general shape of the answers here, we take
        $n=5$. When $\lambda=-1/4$, the value of $k$ is $2$. The parity condition allows the
        size of $\eta$ to be 1, 3 or 5, and we have 
        \[
           \oo^{2}_{5,\eta}=
           \frac{1}{2} \left(\bigl(\alpha+\frac{1}{2} \left(\pm \delta_1 \pm \delta_2
           \pm \delta_3 \pm \delta_4 \pm \delta_5\right)\bigr){}^2-\delta_1^2\right)
        \]
        where the sign is $+$ when $p_{i}\in \eta$ and $-$ otherwise.
        When $\lambda=1/4$, the value of $k$ is $1$, and the parity
        condition allows the size of $\eta$ to be 0, 2 or 4. We have,
        \[
          \oo^{1}_{5,\eta}   =    \alpha+\frac{1}{2} \left(\pm \delta_1 \pm \delta_2
           \pm \delta_3 \pm \delta_4 \pm \delta_5\right)
        \]
        
        Our definition means in particular that, in $H^{*}(T;\Q)$, we
        have $c_{k}(-\ind P) = \phi(  \oo^{k}_{n,\eta})$, where
        $\phi
        :A_{n} \to H^{*}(T;\Q)$ is the natural map (given, with slight
        abuse of notation, by $\alpha\mapsto \alpha$ and $\delta_{p}\mapsto \delta_{p}$).

        \begin{corollary}\label{cor:stratumPD}
            Let $(\fE, \fL) \to T\times S^{2}$ be a family of
            parabolic bundles on $S^{2}_{n}$ parametrized by $T$. Let
            $\lambda$ and $\eta$ be given, satisfying the conditions
            in Definition~\ref{def:omega}, and let
            $T^{\eta}_{\lambda}\subset T$ be the locus defined by
            \eqref{eq:T-stratum}. Assume that the corresponding family
            of Fredholm operators $P$ is transverse to the
            stratification by the dimension of the kernel.  Then the
            cohomology class dual to this stratum is given by
             \[
            \PD [T^{\eta}_{\lambda}] = \phi( \oo^{k}_{n,\eta}) 
             \]
             where $\phi$ is the natural linear map $A_{n}\to
             H^{*}(T;\Q)$, and $k$ is given in terms of $n$ and
             $\lambda$ by \eqref{eq:lambda-to-k}.
        \end{corollary}

        \begin{remarks}
            In Definition~\ref{def:loci-def},  the loci
            $T^{\eta}_{\lambda}$ are characterized by the existence of
            a holomorphic map $\iota: \cF\to\cE$ satisfying additional
            constraints at the distinguished points $\eta\subset \pi$.
            In the language of parabolic bundles, we can regard $\cF$
            as a line bundle with parabolic structure described by a
            subsheaf $\cF_{1}\subset \cF$ such that in a
            neighborhood $\cU_{p}$ of each $p\in\pi$ we have
            \[
\begin{aligned}
                \cF_{1}|_{\cU_{p}} &= \cF|_{\cU_{p}}, & p&\in\eta \\
                                \cF_{1}|_{\cU_{p}} &=
                                (\cF\otimes\cO[-p])|_{\cU_{p}}, &
                                p&\notin\eta. \\
\end{aligned}
            \]
            In these terms, what $T^{\eta}_{\lambda}$ describes is the
            existence of a map $\cF\to\cE$ of parabolic bundles: i.e.
            a map which respects the filtrations. When regarded as a
            line bundle with parabolic structure in this way, we shall
            call $\eta \subset \pi$ the set of ``hits'' for $\cF$. 
            \end{remarks}

        \subsection{The Mumford relations}
        \label{subsec:mumford}

         As a consequence of Corollary~\ref{cor:stratumPD}, we have
         the following statement, which is the essential mechanism in
         Mumford's relations. (See the discussion in
         \cite{AB} for the earlier history of such relations.)

         \begin{proposition}
            Let $(\fE, \fL)$ be a family of parabolic bundles on
            $S^{2}_{n}$ parametrized by a space $T$ as in the previous
            subsection. Suppose that for every $t\in T$ the parabolic
            bundle $(\cE_{t}, \cL_{t})$ on $S^{2}_{n}$ is stable (with
            $\alpha=1/4$ as always). Then for any $\lambda$ and
            $\eta$ 
            satisfying the conditions
            in Definition~\ref{def:omega},  with $\lambda<0$, 
            we have \[ \phi(\oo^{k}_{n,\eta})=0 \] in $H^{2k}(T;\Q)$, where
            $k = n/2  -2 \lambda - 1$ and $\phi:
            H^{*}(\bonf^{*}(S^{2}_{n});\Q) \to H^{*}(T;\Q)$
            is the natural map determined by the characteristic
            classes of\/ $\fE$ and $\fL$.
         \end{proposition}

         \begin{proof}
            When $\lambda<0$, the stratum $T^{\eta}_{\lambda}$
            consists of unstable parabolic bundles, so the hypothesis
            of the Proposition means that such strata are empty. The
            transversality condition is then vacuously satisfied and
            the result follows from Corollary~\ref{cor:stratumPD}.
         \end{proof}
        
          \begin{proposition}\label{prop:Mumford-relations}
          Let $\lambda=-1/4$ and let $\eta\subset \pi=\{p_{1}, \dots ,
          p_{n}\}$ be a subset whose size $h$ satisfies
          \begin{equation}\label{eq:h-parity-plus}
            \begin{gathered}
                h = (n+1)/2 \bmod 2,\\
                0 \le h \le n.
            \end{gathered}
          \end{equation}
          (The first condition is the parity condition
          \eqref{eq:h-parity} for $\lambda=-1/4$.) As in
          Definition~\ref{def:lower-jn}, let $j_{n}$ be the kernel
          of the restriction map to cohomology of the representation variety,
          $H^{*}(\Rep(S^{2}_{n});\Q)$. Then we have
          \[
                \oo^{m}_{n,\eta}\in j_{n},
          \]
          for $ m=(n-1)/2$ . That is, $\oo^{m}_{n,\eta}$ is a relation in
          the cohomology ring of\/ $\Rep(S^{2}_{n})$.
          \end{proposition}

          \begin{proof}
            This follows from the previous proposition by specializing
            to the case $\lambda=-1/4$, because $\Rep(S^{2}_{n}) \cong
            \cM(S^{2}_{n})$ parametrizes a family of stable parabolic bundles.
          \end{proof}

          \begin{definition}\label{def:Mumford-rel}
            Let $j_{n}\subset A_{n}$ be again the ideal of relations in the
            cohomology of $\Rep(S^{2}_{n})$.
            With $m=(n-1)/2$ and $\eta\subset\pi$ a subset whose size
            $h$ satisfies the parity condition
            \eqref{eq:h-parity-plus}, we refer to the relation
            $w^{m}_{n,\eta}\in j_{n}$ as a \emph{Mumford relation}.
            The collection of all these, as $\eta$ varies, are the
            \emph{Mumford relations} in the cohomology ring of
            $\Rep(S^{2}_{n})$.
          \end{definition}
          
         \subsection{Explicit formulae}

         The elements $w^{m}_{n,\eta}\in A_{n}$ appearing as the
         Mumford relations, and more generally the cohomology classes
         $w^{k}_{n,\eta}$ have been described using a characterization
         that does not immediately yield explicit formulae. In
         particular, $w^{k}_{n,\eta}$ is defined in terms of a Chern
         \emph{class} of an index element, while the explicit formula
         \eqref{eq:ch-minus-ind-d1} provides the Chern \emph{character} in closed form.

         As a first step towards  closed formula for
         $w^{k}_{n,\eta}$, as in
        \cite{Xie-Zhang} for
        example, and following \cite{Zagier}, a
        formula for the total Chern class can be derived as
        the formal series
        \begin{equation}\label{eq:total-Chern-class}
               \sum_{k=0}^{\infty} c_{k}(-\ind(P)) =
                       (1 + \beta)^{\iA_{\lambda}/2}\left(
                       \frac {1 + \delta_{1}}{1-\delta_{1}}
                       \right)^{B_{\eta}/(2\delta_{1})}
        \end{equation}
        where $\iA_{\lambda}$ and $B_{\eta}$ are as in
        \eqref{eq:ch-minus-ind-d1}:
        \begin{equation}\label{eq:AB-extracted}
\begin{aligned}
    \iA_{\lambda} &= (n/2-2\lambda -2 ) \\
    B_{\eta} &=  \alpha +
                  \frac{1}{2}\sum_{p\in
                  \eta} \delta_{p} - \frac{1}{2}\sum_{p\not\in
                  \eta} \delta_{p}.
\end{aligned}
        \end{equation}
        (See \cite{Zagier} for the
        interpretation of the right-hand side of this formula.) We can therefore write
        \begin{equation}\label{eq:chern-class-kth-deriv-proto}
            w^{k}_{n,\eta} = \frac{1}{k!}\left(\frac{
            d^{k}}{dt^{k}}\right)\bigg|_{t=0} \left(
            (1 + t^{2}\beta)^{\iA_{\lambda}/2}\left(
                       \frac {1 + t \delta_{1}}{1-t \delta_{1}}
                       \right)^{B_{\eta}/(2\delta_{1})}\right).
        \end{equation}
        Note here that $\iA_{\lambda}$ is minus the numerical index of
        $P$, and that in the definition of $w^{k}_{n,\eta}$ the
        integer $k$ is $-\ind(P)+1$, so we can write
                       \begin{equation}\label{eq:chern-class-kth-deriv}
            w^{k}_{n,\eta} = \frac{1}{k!}\left(\frac{
            d^{k}}{dt^{k}}\right)\bigg|_{t=0}\left(
            (1 + t^{2}\beta)^{(k-1)/2}\left(
                       \frac {1 + t \delta_{1}}{1-t \delta_{1}}
                       \right)^{B_{\eta}/(2\delta_{1})}\right).
        \end{equation}
         The following proposition gives a closed
         formula for this $k$'th term in the power series.

         \begin{proposition}\label{prop:explicit-formulae}
            We have
            \[
                  k!  w^{k}_{n,\eta} = \prod_{\substack{j=-k+1 \\ j =
                    -k+1 \bmod 2 }}^{k-1} (B_{\eta} + j \delta_{1}).
            \]
         \end{proposition}

         \begin{proof}
            Let us write
            \[
            C_{k} =   k!  w^{k}_{n,\eta} 
                          = \left(\frac{
            d^{k}}{dt^{k}}\right)\bigg|_{t=0} G_{k-1} (t)
            \]
            where
            \[
            G_{k-1}(t) =(1 + t^{2}\beta)^{(k-1)/2}\left(
                       \frac {1 + t \delta}{1-t \delta}
                       \right)^{B/(2\delta)},
            \]
            and we have abbreviated
            \[
            \begin{aligned}
            B &= B_{\eta} \\
            \delta &= \delta_{1}
            \end{aligned}
            \]
            to streamline the notation.
            
            Let $\hat C_{k}$ denote the right-hand side in the
            proposition,
            \[
                \hat C_{k} =  \prod_{\substack{j=-k+1 \\ j =
                    -k+1 \bmod 2 }}^{k-1} (B + j \delta),
            \]
            so that what we aim to prove is that $C_{k}$ and $\hat
            C_{k}$ are equal.
            We shall prove in fact that 
            \begin{equation}\label{eq:G-relation}
	\frac{d^{k}}{dt^{k}} G_{k-1}(t)=\hat C_k G_{-k-1}(t)
	\end{equation}
        which yields the desired equality $C_{k}=\hat C_{k}$ on
        substituting $t=0$, since $G_{l}(0)=1$ for all $l$.

 We prove \eqref{eq:G-relation} by induction on $k$: specifically, assuming
 that \eqref{eq:G-relation} holds for $k$, we prove the result for
 $k+2$. The seed cases, $k=0,1$, are clear.
Note first
            that $\hat C_{k}$ satisfies a recurrence relation
            \begin{equation}\label{eq:C-recurrence}
            \begin{aligned}
            \hat C_{k+2} &= (B^{2} + (k+1)^{2} \beta)\hat  C_{k} \\
                                &=  (B^{2} - (k+1)^{2} \delta^{2})
                               \hat C_{k}.
                                \end{aligned}
                                \end{equation}
Next we examine the first two derivatives of $G_{k}(t)$: by induction
on $k$ and using the fact that $G_{k}(t)= (1 + t^{2}\beta) G_{k-2}$,
we obtain the following identity for the first derivative:
\begin{equation}
\begin{aligned}
	\frac{d}{dt}G_k(t) & = (B-k \delta^2 t)G_{k-2}(t).
\end{aligned}
	\end{equation}
Applying this twice, we obtain an identity for the second derivative:
\begin{equation}
\begin{aligned}
\frac{d^2}{dt^{2}}G_k(t) &=
\left(B^2-q \delta^2 -2(k-1) B \delta^2 t +k(k-1)C^4 t^2\right)G_{k-4}(t).
\end{aligned}
	\end{equation}        
Using these identities for the first two derivatives, together with
the induction hypothesis \eqref{eq:G-relation}
and the recurrence relation \eqref{eq:C-recurrence}, we compute:        
\[
\begin{aligned}
	\frac{d^{k+2}}{dt^{k+2}} G_{k+1}(t) & = \frac{d^{k+2} (1-\delta^2t^2)G_{k-1}(t)}{dt^{k+2}}\\
	&  \begin{aligned} = (1-\delta^2t^2)\frac{d^{k+2}
        G_{k-1}(t)}{dt^{k+2}}  -2(k+2)\delta^2t & \frac{d^{k+1}
        G_{k-1}(t)}{dt^{k+1}}\\
       & -(k+2)(k+1)\delta^2\frac{d^{k}
        G_{k-1}(t)}{dt^{k}}
        \end{aligned}
        \\
	&= \Bigl((1-\delta^2t^2)\frac{d^2}{dt^2}-2(k+2)\delta^2t
        \frac{d}{dt}-(k+2)(k+1)\delta^2\Bigr)\frac{d^{k} G_{k-1}(t)}{dt^{k}}\\
	&=  \Bigl((1-\delta^2t^2)\frac{d^2}{dt^2}-2(k+2)\delta^2t
        \frac{d}{dt}-(k+2)(k+1)\delta^2\Bigr)\hat C_kG_{-k-1}(t)\\
	&= \left((B^2+\delta^2 (k+1)+2(k+2) B \delta^2t +(k+1)(k+2)\delta^4 t^2)\right.\\
	& \qquad\qquad\qquad -2(k+2)\delta^2t (B+(k+1)
        \delta^2t) \\
        &\qquad\qquad \qquad\qquad \left. -(k+2)(k+1)\delta^2(1-\delta^2t^2)\right)\hat C_kG_{-k-3}(t)\\
	&= (B^2-(k+1)^2\delta^2)\hat C_kG_{-k-3}(t)\\
	&= \hat C_{k+2}G_{-k-3}(t)
\end{aligned}
\]
as required.  
         \end{proof}

          \subsection{The Mumford relations as generators of the
          ideal}

         In \cite{Street}, a presentation of the cohomology ring of
         $\Rep(S^{2}_{n})$ is given by exhibiting a
         complete set of generators for the ideal of
         relations
         $j_{n}\subset A_{n}$, all of which have degree $m=(n-1)/2$.
         We now show that the elements $w^{m}_{n,\eta}$ also generate
         the ideal, by relating them to the generators in
         \cite{Street}.

         \begin{remark}
            The statement that the elements $w^{s}_{n,\eta}$, for
            $s\ge m$, generate the ideal is a counterpart of Kirwan's
            result \cite{Kirwan} in the case of a (non-orbifold)
            surface of genus $g$. Kirwan's result was strengthened by
            Kiem \cite{Kiem}, who showed that the relations in the
            middle dimension (i.e.~the case $s=m$ in our context) are
            sufficient. The results of \cite{Kirwan} were extended to
            the case of parabolic bundles on surfaces of genus $g\ge
            2$ with one marked point by Earl and Kirwan
            \cite{Earl-Kirwan}.
         \end{remark}

\begin{proposition}\label{prop:little-w-generate}
    Fix $n\ge 3$ odd, and write $n=2m+1$. Then as $\eta$ runs through
    all subsets of $\pi$ whose size $h=|\eta|$ satisfies the parity
    condition \eqref{eq:h-parity-plus}, the elements
    $\oo^{m}_{n,\eta}\in A_{n}$ form a set of generators of the
    ideal $j_{n}$. That is, the elements $\oo^{m}_{n,\eta}$ form a
    complete set of relations for the cohomology of\/ $\Rep(S^{2}_{n})$
    as a quotient of the algebra $A_{n}$.
\end{proposition}

\begin{proof}
    From the results of \cite{Street}, in the ideal $j_{n}$, there
    is an element $r^{m}$ that has degree $m$ and belongs to the
    subalgebra $\Q[\alpha, \beta] \subset A_{n}$, where
    $\beta=-\delta_{p}^{2}$. The element $r^{m}$ is unique up to
    scale.
    According to \cite[Corollary
    1.6.2]{Street}, the ideal $j_{n}$ is generated by the elements
    \[
            R_{\zeta}^{m} = r^{m-|\zeta|}(\alpha,\beta) \prod_{p\in
            \zeta} \delta_{p}
    \]
    where $\zeta$ runs through all subsets of $\pi$ of size $0\le
    |\zeta| \le
    m$. These elements all have degree $m$.
 
    As we vary $\eta$, we
    obtain $2^{n-1}$  expressions $w^{m}_{n,\eta}$, all of which
    are elements of $j_{n}$ of degree $m$. Because $m$ is the lowest
    degree in which relations exist, each $w^{m}_{n,\eta}$ is a
    $\Q$-linear combination of the generators $R^{m}_{\zeta}$. The
    number of generators $R^{m}_{\zeta}$ is also $2^{n-1}$; so in
    order to see that the elements $w^{m}_{n,\eta}$ generate the ideal
    $j_{n}$, it will be enough if we show that they are linearly
    independent over $\Q$.

    The fact that the elements $w^{m}_{n,\eta}$ are linearly
    independent can be deduced through a direct
    examination of the formulae which define it, as follows. Let us
    specialize the formulae by setting
    $\beta=0$, 
    in which case the expression
    \eqref{eq:total-Chern-class} for the total
    Chern class of $-\ind(P)$ simplifies to
    \[
                (1 + 2\delta_{1})^{B_{\eta}/(2\delta_{1})} = \exp B_{\eta}
    \]
    because $\delta_{1}^{2}=0$. The element $w^{m}_{n,\eta}$ therefore
    specializes to $B^{m}_{\eta}/m!$, or if we further specialize by
    setting $\alpha=1$, to
    \[
          (1/m!) \bigl(  1 + \sum_{p}\eta_{p}\delta_{p}  \bigr)^{m}
    \]
    where $\eta_{p}=1$ for $p\in\eta$ and $\eta_{p}=-1$
    otherwise. We can expand this as
    \[
                    \sum_{|\zeta|\le m} C_{\eta,\zeta} \biggl(\prod_{p \in
                    \zeta}\delta_{p}\biggr)
    \]
    where the rational coefficient $C_{\eta,\zeta}$ is given by
    \[
      C_{\eta,\zeta} = \frac{1}{\bigl(m-|\zeta|\bigr)!}
      \biggl(\prod_{p\in\zeta}\eta_{p}\biggr).
    \]
    We wish to see that the matrix $C = (C_{\eta,\zeta})$, which is square
    of size $2^{n-1}$, is non-singular. To do this, we compute the dot product
    of the columns of $C$ corresponding to different subsets
    $\zeta_{1}$ and $\zeta_{2}$. For fixed $\eta$ we have
    \[
               C_{\eta,\zeta_{1}}
             C_{\eta,\zeta_{2}} = \frac{1}{\bigl(m-|\zeta_{1}|\bigr)!
             \bigl(m-|\zeta_{2}|\bigr)! } \times
             \begin{cases}
            +1, & \text{if $|\eta \cap (\zeta_{1}\ominus \zeta_{2})^{c}|$
            is even} \\
            -1, & \text{if $|\eta \cap (\zeta_{1}\ominus \zeta_{2})^{c}|$
            is odd} 
           \end{cases}
    \]
    where the superscript $c$ denotes the complement. Since $\zeta_{1}$
    and $\zeta_{2}$ are distinct subsets of size strictly less than
    $n/2$, their symmetric difference is a non-empty proper subset of
    $\pi$. The number of subsets $\eta$ of a given parity for which
    the intersection is even and
    the number for which it is odd are therefore equal, and we see
    that
    \[
                \sum_{\eta} C_{\eta,\zeta_{1}}
             C_{\eta,\zeta_{2}}  = 0.
    \]
    The columns are therefore orthogonal, showing that the square
    matrix $C$ is
    non-singular, as required.
\end{proof}

\begin{remarks}
    An alternative verification of the linear independence of the
    elements $w^{m}_{n,\eta}$, not depending on an examination of the
    formula, will be seen later, in the remarks at the end of
    section~\ref{subsec:proof-subleading}.
\end{remarks}

    \section{Relations in instanton homology}
     \label{sec:relations-instanton}
     
          \subsection{Deforming the Mumford relation with instanton
          terms}

          The element $\oo^{m}_{n,\eta} \in j_{n}$ in
          Proposition~\ref{prop:Mumford-relations} is a relation in
          the ordinary cohomology ring $H^{*}(\Rep(S^{2}_{n});\Q)$.
          Via its description in terms of the multiplicative
          generators $\alpha$ and $\delta_{p}$, as an
          explicit element of the ring
          \[
                  \Q[\alpha, \delta_{1}, \dots, \delta_{n}]\bigm/ \bigl\langle 
                 \delta_{i}^{2}-\delta_{j}^{2}\bigr\rangle_{i,j} ,
          \]
          we may regard $\oo^{m}_{n,\eta}$ also as an element of the ideal
          $J_{n} \subset \cA_{n}$ of
          Proposition~\ref{prop:coho-Z-cyclic}, where it is a relation in
          the ordinary cohomology ring $H^{*}(\Rep(Z_{n});\cR)$. As
          $\eta$ varies over all subsets of $\pi$ satisfying the
          parity condition, the elements $\oo^{m}_{n,\eta} \in J_{n}$
          form a set of generators of the ideal, as follows
          immediately from the corresponding statement for
          $\Rep(S^{2}_{n})$
          (Proposition~\ref{prop:little-w-generate}).

          The
          following proposition promotes $\oo^{m}_{n,\eta}$ to a relation
          between the corresponding operators on the instanton
          homology $I(Z_{n})$ by adding terms of lower degree. Recall
          that $\cJ_{n} \subset \cA_{n}$ is the annihilator of
          $I(Z_{n})$ as an $\cA_{n}$-module.
          
          \begin{proposition}\label{prop:quantum-Mumford}
          Let $n$ be odd and let $\eta\subset\pi$ be a subset whose size $h$ satisfies the parity
          condition \eqref{eq:h-parity-plus}. Write $m=(n-1)/2$ and
          let $\oo^{m}_{n,\eta}\in
          j_{n}\subset J_{n}$ be as in
          Proposition~\ref{prop:Mumford-relations}, regarded as a relation in the
          ordinary cohomology of the representation variety
          $\Rep(Z_{n})$. Then there
          is a unique element $\om^{m}_{\eta}\in\J{n}\subset \cA_{n}$
          in filtration degree $m$ whose
          leading term is $\oo^{m}_{n,\eta}$:
          \[
                \om^{m}_{\eta}=\oo^{m}_{n,\eta} \pmod
                {\cA_{n}^{(m-1)}}.
          \]
          As $\eta$ varies over all subsets satisfying the parity
          condition, these elements $\om^{m}_{\eta}$ form a set of
          generators for the ideal of relations $\J{n}$.
          \end{proposition}

                   \begin{remark}
            Our notation for $\om^{m}_{\eta}$ does not include $n$,
            since $n$ is always related to $m$ in this context by
            $n=2m+1$.
          \end{remark}

\begin{proof}[{Proof of Proposition~\ref{prop:quantum-Mumford}}]
          Corollary~\ref{cor:relations-coh-to-inst} gives
          the existence of $\om^{m}_{\eta}\in\J{n}$ with leading term
          $\oo^{m}_{n,\eta}$. The uniqueness assertion is a consequence of
          Proposition~\ref{prop:below-middle}. The fact that these are
          a complete set of generators for the ideal follows from the
          corresponding statement for the elements $w^{m}_{n,\eta}\in
          J_{n}$ together with the fact that $\cA_{n}/J_{n}$ and
          $\cA_{n}/\J{n}$ are free modules of the same rank, because
          they are respectively the ordinary homology of
          $\Rep(Z_{n})$ and the instanton homology of $Z_{n}$
          (Proposition~\ref{prop:IZ-two-copies}).
\end{proof}          

           We aim to give an algorithm for computing $W^{m}_{\eta}$ as
           a deformation of $w^{m}_{n,\eta}$, and our first main step
           will be to determine the sub-leading term. That is,
           Corollary~\ref{cor:relations-coh-to-inst} provides
           the existence
          of $\oo(1)$ with
          \[
                \om^{m}_{\eta}=\oo^{m}_{n,\eta}  + \epsilon
                \oo(1) \pmod {\cA_{n}^{(m-2)}},
          \]
           and we wish to determine $\oo(1)$.
          
          \begin{proposition}\label{prop:quantum-Mumford-subleading}
         The subleading term of $\om^{m}_{\eta}$ is given
          by $ \epsilon\tau^{n-2h}  \oo^{m-1}_{n,\eta'}$, where $\eta'$ is the complement
          $\pp\setminus \eta$ and $h=|\eta|$, so
          \[
                \om^{m}_{\eta}=\oo^{m}_{n,\eta}  + \epsilon\tau^{n-2h}             
                \oo^{m-1}_{n,\eta'} \pmod {\cA_{n}^{(m-2)}}.
          \]        
          \end{proposition}

          The proof of this proposition will require some preparation. 
          To understand how to characterize the subleading term
          $\epsilon\oo(1)$, we
          draw on the mechanism behind
          Proposition~\ref{prop:leading-terms} and
          Corollary~\ref{cor:relations-coh-to-inst}. Let $\1_{+}$ again be the
          standard cyclic generator of $I(Z_{n})$ from
          Proposition~\ref{prop:instanton-cyclic}, and let $\1_{-} =
          \opepsilon\1_{+}$. Let $\Lam$ be the $\cR$-module
          isomorphism in Proposition~\ref{prop:leading-terms}, and
          let $1_{\pm} = \Lam(\1_{\pm}) \in H_{*}(\Rep(Z_{n});\cR)$.
          Then $w(1)$ is a homogeneous polynomial of degree $m-1$ in $\alpha$
          and the $\delta_{p}$, with coefficients in $\cR$, such that
          \[
                 \Lam(\oo^{m}_{n,\eta} \1_{-}) - \oo^{m}_{n,\eta} 1_{-} =
                 w(1) 1_{+}\quad \bmod \bigoplus_{k\le
                 2(m-2)}H^{k}(\Rep(Z_{n}); \cR).
          \]
          (The right-hand side is the $(m-2)$'th step of the
          increasing filtration of $H^{*}(\Rep(Z_{n};\cR)$.)

          Recall next we have 
          an expansion of the operator $\opalpha$ according to
          the action $\kappa \in (1/4)\Z$, as in
          \eqref{eq:opalpha-expansion} and
          Proposition~\ref{prop:leading-terms}. There is a similar expansion of
          each $\opdelta_{p}$. This gives a $\kappa$-expansion of any
          monomial in $\opalpha$ and the $\opdelta_{p}$, and therefore
          of the multiplication operator of any $u\in \cA_{n}$ acting
          on $I(Z_{n})$. That is, we may write 
          \[
                    u \xi = \sum_{\kappa\in (1/4)\Z} u
                    \smileI_{\kappa} \xi.
          \]
          This description is set up so that if $u\in \cA_{n}$ is in
          grading $k$ and $\Lam(\xi) \in H^{2l}(\Rep(Z_{n});\cR)$,
          then
          \[
                    \Lam (u
                    \smileI_{\kappa} \xi) \in H^{2(l+k) -
                    8\kappa}(\Rep(Z_{n});\cR).
          \]
          The description of $w(1)$ then becomes
          \begin{equation}\label{eq:w1-interpretation}
                w(1)1_{+} =\Lam( \oo^{m}_{n,\eta}  \smileI_{1/4} 1_{-}).
          \end{equation}
          Computation of $w(1)$ in this form therefore depends
          directly on
          understanding the instantons on the cylinder $\R\times
          Z_{n}$ with action $1/4$. We address this calculation in
          the following subsection, where the proof of
          Proposition~\ref{prop:quantum-Mumford} will be completed.
          
          \subsection{Characterizing the sub-leading term}

          From the discussion above, we are interested in the moduli
          space $M_{\kappa}(\R\times Z_{n})$ of anti-self-dual
          bifold $\SU(2)$ connections on the cylinder, particularly
          for $\kappa=1/4$. By attaching a
          copy of the bifold $D^{2}\times S^{2}_{n}$ to each of the
          two ends, we form from the
          cylinder a compact bifold
          \[
                    X = S^{2}\times S^{2}_{n}.
          \]
          For clarity in distinguishing the two factors here, we will
          write
          \[
                    X = B \times C
          \]
          where $B$ is $S^{2}$ and $C$ is the bifold $S^{2}_{n}$.
          We write $M_{\kappa}(X)$ for the moduli space of
          anti-self-dual
          $\SU(2)$ connections on the bifold $X$, with action
          $\kappa$, and we write $M_{\kappa}^{e}(X)$ for the moduli
          space corresponding to $w_{2}=[e]$, where $[e]=\{b\}\times
          C$. The moduli spaces depend,
          of course, on a choice of conformal structure on $X$. The moduli space
          $M_{\kappa}(X)$ is non-empty only if $\kappa\in
          (1/2)\Z$, while $M_{\kappa}^{e}(X)$ is non-empty only if $\kappa\in
          (1/2)\Z + 1/4$. The moduli spaces have formal dimension
          \[
                d(\kappa) = 8\kappa + 2n - 6.
          \]

          For any
          element $u \in \cA_{n}$ of degree $d(\kappa)/2$ in the
          variables $\alpha$ and $\delta_{i}$, we can seek
          to evaluate a Donaldson polynomial invariant by evaluating the
          corresponding cohomology class on $M_{\kappa}(X)$ or
          $M^{e}_{\kappa}(X)$. Because we working with local
          coefficients $\Gamma$, our Donaldson invariants should also
          involve $\cR$-valued weights. By the formula
          \eqref{eq:Gamma-A}, the local system $\Gamma$ defines a
          locally constant function
          \begin{equation}\label{eq:Gamma-M-weights}
                \Gamma: M_{\kappa}(X) \to \cR^{\times}
          \end{equation}
          and so the moduli spaces are a collection of oriented,
          weighted manifolds.
          
          However,
          the bifold $X$ has $b_{2}^{+}=1$, so the appearance of
          reducibles in one-parameter families means that the
          Donaldson invariant depends on a choice of chamber in the
          space of Riemannian metrics on $X$. We consider a product
          metric in which the area of $B$ is very large compared to the
          area of $C$, and we call this the \emph{$B$-chamber}. (This
          means that the self-dual 2-form for the Riemannian metric on
          $X$ is nearly Poincar\'e dual to
          a multiple of $\PD[C]$.)
          Similarly there is a distinguished chamber, the
          \emph{$C$-chamber}, in which the area of $C$ is very large
          compared to $B$. There is then a well-defined Donaldson
          invariant $q^{B}_{\kappa}$ 
          in the $B$-chamber,
          \[
          \begin{aligned}
          u &\mapsto q^{B}_{\kappa}(X;u) \\
          \cA_{n} &\to \cR,
          \end{aligned}
          \]
          calculated using either the moduli space $M_{\kappa}(X)$ or
          the moduli space
          $M^{e}_{\kappa}(X)$, depending on whether $4\kappa$ is even
          or odd respectively. Our notation again makes no explicit
          mention of the local coefficient system, but the
          contributions of the various components of the moduli spaces
          are to be weighted by the locally constant function
          \eqref{eq:Gamma-M-weights}.

          These Donaldson invariants of $X$ are
          related to the action of $u$ on $I(Z_{n})$ by a gluing
          argument, because of the description of $X$ as the union of
          the cylinder $[-1,1]\times Z_{n}$ and the two copies of
          $D^{2}\times S^{2}_{n}$. More specifically, let $\1_{+}\in
          I(Z_{n})$ be once more the cyclic generator obtained as the
          relative invariant of the manifold $D^{2}\times S^{2}_{n}$,
          and let $\1_{+}^{\dagger}$ be the element of the instanton
          \emph{co}homology group $I^{*}(Z_{n})$ obtained by regarding
          $D^{2}\times Z_{n}$ as a manifold with boundary $-Z_{n}$.
          Then for $\kappa\in(1/2)\Z$ and $u\in \cA_{n}^{+}$, we can write
          \[
                q^{B}_{\kappa}(X;u) =  \langle u*_{\kappa} \1_{+} ,
                \1^{\dag}_{+} \rangle
          \]
          where the pairing on the right is the $\cR$-valued pairing
          between $I(Z_{n})$ and $I^{*}(Z_{n})$. For $\kappa\in (1/4)
          + (1/2)\Z$, we have
          \[
                q^{B}_{\kappa}(X;\epsilon u) =  \langle u*_{\kappa} \1_{+} ,
                \1^{\dag}_{-} \rangle.
          \]
          From this relationship and Poincar\'e duality,
          It follows that 
          \eqref{eq:w1-interpretation}
          is equivalent to
          \begin{equation}\label{eq:subleading-relation-X}
                q^{B}_{1/4}(w(0) v)  = q^{B}_{0}(\epsilon w(1) v)
          \end{equation}
          for all $v\in \cA_{n}$ of degree
          \[
          \begin{aligned}
          \deg(v) &= (1/2)d(1/4) - \deg(w(0)) \\
                        &= n - 2  - m \\
                        &=  m-1 ,
                    \end{aligned}
          \]
           where $n=2m+1$ as usual.

           The situation is somewhat simplified now because the moduli
           spaces $M_{0}(X)$ and $M^{e}_{1/4}(X)$ are \emph{compact}. This
           is because non-compactness of the moduli space arises only
           from bubbling, and bubbles decrease $\kappa$ by multiples
           of $1/2$. So for $\kappa\le 1/4$, the Donaldson invariants are simply
           evaluations on $[M_{\kappa}(X)]$ or $[M^{e}_{\kappa}(X)]$
           of ordinary cohomology
               classes in $H^{*}(\bonf^{*}(X) ; \cR)$, weighted by the
               function locally constant \eqref{eq:Gamma-M-weights}.
               We will write $[\Gamma\cdot M_{\kappa}(X)]$ and
               $[\Gamma\cdot M_{\kappa}^{e}(X)]$ for these weighted
               fundamental classes, as elements of the ordinary
               homology $H_{*}(\bonf^{*}(X); \cR)$.

               Via the
           relationship between $\cA_{n}$ and
           $H^{*}(\bonf^{*}(Z_{n});\cR)$, we have an inclusion
           \[
              \cA_{n} \hookrightarrow H^{*}(\bonf^{*}(X) ; \cR).
           \]
           The relation \eqref{eq:subleading-relation-X} can therefore
           be stated in terms of ordinary pairings,  between these
           cohomology classes and the fundamental classes of the
           moduli spaces:
          \[
                          \bigl\langle w(0) v, [\Gamma\cdot M^{e}_{1/4}(X)]
                    \bigr\rangle = \bigl\langle w(1) v, [\Gamma\cdot M_{0}(X)]
                    \bigr\rangle .
          \]
          The assertion in
          Proposition~\ref{prop:quantum-Mumford-subleading}
          concerning the value of the subleading term $w(1)$ can
          therefore be restated as the following proposition.

          \begin{proposition}\label{prop:subleading-interpret}
          Let $n=2m+1$ as usual let $v\in \cA_{n}$ be any element of
          degree $m-1$. Let $\oo_{n,\eta}^{k} \in \cA_{n}$ be the
          explicit polynomials described in
          Definition~\ref{def:omega}. Then we have
          \[
                      \bigl\langle \oo_{n,\eta}^{m} v, [\Gamma\cdot M^{e}_{1/4}(X)]
                    \bigr\rangle = \bigl\langle  \tau^{n-2|\eta|}
                    \oo_{n,\eta'}^{m-1} v,  [\Gamma\cdot M_{0}(X)]
                    \bigr\rangle ,
          \]
          where the (compact) moduli space $M^{e}_{1/4}(X)$ is
          computed using a metric on $X$ in the $B$-chamber, and
          $M_{0}(X)$ is the moduli space of flat bifold connections,
          a copy of $\Rep(S^{2}_{n})$.
          \end{proposition}

          The proof of Proposition~\ref{prop:subleading-interpret} is
          given in section~\ref{subsec:proof-subleading}, after
          a digression on the wall-crossing behavior of moduli spaces
          on $X$.

         \subsection{A wall-crossing argument}
          
          The structure of our argument up to this point is closely
          related to the work of Mu\~noz \cite{Munoz2}, in which a key
          step is the calculation of the contribution of the first
          non-trivial moduli space (our $M_{1/4}^{e}(X)$ in the
          present context). In \cite{Munoz2}, the relevant moduli
          space was of the form $M_{1/2}^{e}(S^{2}\times \Sigma_{g})$
          for a smooth surface $\Sigma_{g}$, and the key observation
          is that this moduli space is empty in one chamber (when the
          area of the $S^{2}$ factor is small, corresponding to the
          $C$-chamber in our notation) and undergoes a single
          wall-crossing where the metric passes to the $B$-chamber.
          (See \cite[Proposition 2]{Munoz2}.) 
          The description of the wall-crossing for $S^{2}\times
          \Sigma_{g}$ leads to a
          description of the moduli space on the $B$ side of
          the wall as a bundle over the Jacobian $J(\Sigma_{g})$ with
          fiber a complex projective space.

          Such a description has an
          exact parallel in our orbifold context, with the Jacobian
          $J(\Sigma_{g})$ in Mu\~noz's situation
          replaced now by the finite set of bifold line
          bundles on $S^{2}_{n}$ of a fixed bifold degree. That is,
          the wall-crossing contributes to
          $M_{1/4}^{e}(X)$ a finite number of copies of a complex
          projective space, where an explicit understanding of the
          cohomology classes allows a calculation of the Donaldson
          invariant. We now turn to the details of this calculation.

          \begin{lemma}
            In the $C$-chamber, the Donaldson invariants
            $q_{\kappa}^{C}(u)$ are zero when $\kappa$ is in $(1/4) +(1/2)\Z$.
          \end{lemma}

          \begin{proof}
            The bifold $X$ decomposes into two parts along a copy of
            $B\times S^{1} \subset B\times C$, i.e. an $S^{2}\times
            S^{1}$. The bundle has $w_{2}$ non-zero on this
            $S^{2}\times S^{1}$ when $\kappa$ is in $(1/2)\Z
            + (1/4)$, so there are no flat connections on $B\times
            S^{1}$. A stretching argument therefore shows that the
            anti-self-dual moduli space
            is empty when the metric on $X$ contains a long neck
            $[-T,T]\times B \times S^{1}$. A metric with such a long
            neck lies in the $C$-chamber, so the invariant in this
            chamber is zero.
          \end{proof}

          \begin{lemma}
            For the moduli spaces $M_{\kappa}^{e}(X)$ with $\kappa\le
            1/4$, in a 1-parameter family of product metrics on $X=B\times
            C$ passing from the $C$-chamber to the $B$-chamber,
            exactly one wall is crossed.
          \end{lemma}

          \begin{proof}
            The only non-empty moduli space $M^{e}_{\kappa}(X)$ with $\kappa\le 1/4$ is
            the moduli space $M_{1/4}^{e}(X)$, and a wall is crossed
            when the Riemannian metric allows the existence of a
            reducible anti-self-dual connection in this moduli space.
            We are therefore looking for a reduction of the bifold
            adjoint $\SO(3)$ bundle as $\R\oplus K$, where $K$ is a
            bifold 2-plane bundle. Let us write the bifold Euler class
            $\euler(K)$ as
            \[
                   \PD \euler(K) = x [B] + y [ C ].
            \]
            Here $y$ is an odd integer because $\euler(K)[B]$ is odd. On
            the curve $C$, the bundle $K$ has $n$ bifold points, and
            $n$ is odd; so $2x$ is also an odd integer. For a given
            Riemannian metric, let us write the class of the
            self-dual $2$-form as
            \[
                   \PD [\omega^{+}] = [ B] +  t [C] ,
            \]
            suitably normalized. The condition that the curvature of
            $K$ is anti-self-dual imposes the constraint that $\euler(K)$
            and $[\omega^{+}]$ are orthogonal, which is to say
            \[
                         y  = - tx. 
            \]
            The action $\kappa$ is $-\euler(K)^{2}/4$ which is $-xy/2$.
            Using the orthogonality condition, we write this as $\kappa= t
            x^{2}/2$. With $\kappa=1/4$, our constraints therefore become
\begin{enumerate}
        \item $t x$ and $2x$ are odd integers, and
        \item $ t x^{2} = 1/2$ .
\end{enumerate}
          These constraints force $x=\pm 1/2$ and $t=2$. The
          orientation of $K$ is indeterminate, and the sign of $x$ can
          therefore be taken to be positive. A path of Riemannian
          metrics passing from the $C$ chamber to the $B$ chamber is a
          path in which $t$ begins close to $0$ and ends close to
          $+\infty$, and the wall is crossed at $t=2$.
          \end{proof}
          
          The proof the lemma shows that the wall-crossing occurs when
          there is an orbifold 2-plane bundle $K$ with
          \[
                \PD \euler(K) = (1/2) [B] - [C].
          \]
          The degree of $K$ on $C=S^{2}_{n}$ is thus $1/2$. In terms
          of an $\SU(2)$ lift of on the curve $\{b\}\times C$ then, we can write
          the bundle as
          \[
                     F \oplus F^{-1}
          \]
          where $F$ is a complex line bundle with
          limiting holonomy $\pm i$ on the linking
          circles at the $n$ singular points. We orient $K$ as
          $F^{-2}$.  The Chern-Weil integral
          for the first Chern class of the singular connection on $F$
          is $-1/4$. As a parabolic bundle on
          $S^{2}_{n}$ we can write the underlying rank-2 vector bundle
          as $\cE = \cF \oplus \cF^{-1}$, and for each $p\in\pi$ the
          distinguished line $\cL_{p}\subset \cE_{p}$ is the summand
          $\cF_{p}$ if the limiting holonomy is $-i$ and
          $\cF^{-1}_{p}$ otherwise. Write $\xi \subset \pi$
          for the
          set where the holonomy is $-i$. Then
          \[
                c_{1}(\cF)[C] + |\xi|/4 - (n-|\xi|)/4 = - 1/4.
          \]
          This constraint imposes the parity condition $|\xi| =
          (n-1)/2$ mod $2$, which
          allows $2^{n-1}$ possibilities for $\xi$. We summarize this
          with another lemma.

          \begin{lemma}
            When the Riemannian metric on $X = B\times C$ lies on the
            wall between the two chambers, the moduli space
            $M^{e}_{\kappa}(X)$ consists of $2^{n-1}$ reducible
            anti-self-dual connections, corresponding to the subsets
            $\xi\subset \pi$ whose size $|\xi|$ has the same parity
            as $(n-1)/2$. 
          \end{lemma}

          Let $A_{0}$ denote any one of the reducible connections
          described in the lemma. The formal dimension of the moduli
          space $M^{e}_{1/4}(X)$ is $2n-4$. If we write the orbifold
          adjoint bundle as $\R\oplus K$ now on the whole of $X$,
          then in the deformation theory of $A_{0}$ we have a
          contribution of $1$ to the dimension of $H^{0}_{A_{0}}$ coming from the $\R$
          summand because $A_{0}$ is reducible, and there is a similar
          contribution of $1$ to the dimension of $H^{2}_{A_{0}}$
          from the $\R$ summand because $b^{+}_{2}=1$. If we assume
          that the deformation theory is otherwise unobstructed (an
          assumption which we will see later is justified for product
          metrics on $B\times C$, without
          the need for perturbing the equations), then it follows that
          $H^{1}_{A_{0}}$ has dimension $2n-2$ and that this comes
          from the $K$ summand of the adjoint bundle. With this in
          place, the standard model for wall-crossing describes the
          moduli space $M^{e}_{1/4}(X; g_{t})$ for a Riemannian metric
          $g_{t}$ whose conformal parameter $t$ is $2+\epsilon$ for
          small $\epsilon$ as a copy of $\CP^{n-2}$ in a neighborhood
          of each reducible $A_{0}$. We therefore have the following
          proposition.

          \begin{proposition}\label{prop:wall-crossing-summary}
          For a product metric on $X$ which lies in the
          $B$-chamber and is close to the wall, the moduli space
          $M^{e}_{1/4}(X)$ consists of
          $2^{n-1}$ copies of $\CP^{n-2}$.
           \end{proposition}

           As mentioned earlier, this is a close counterpart to the
           result \cite[Proposition 2]{Munoz2}, where the
           corresponding description of the moduli space of smallest
           positive action is a bundle of projective spaces
           over the Jacobian of a smooth curve.

           \subsection{A proof of
          Proposition~\ref{prop:subleading-interpret}}
          \label{subsec:proof-subleading}

          From their definition, $w^{m}_{n,\eta}$ and
          $w^{m-1}_{n,\eta'}$ represent cohomology classes dual to loci
          $U^{\eta}_{-1/4}$ and $U^{\eta'}_{1/4}$ in the space of
          bifold connections $\bonf^{*}(S^{2}_{n})$. If we select a
          fiber
          \[
                \{b_{0}\} \times S^{2}_{n} \subset B\times
                S^{2}_{n}=X,
          \]
          then we obtain by restriction corresponding loci in the
          spaces of bifold connections on $X$:
          \[
                \begin{aligned}
                    U^{\eta}_{-1/4}(b_{0})&\subset \bonf^{*}(X)^{e} \\
                    U^{\eta'}_{1/4}(b_{0})&\subset \bonf^{*}(X).
                \end{aligned}
          \]
          In this way we can interpret the equality to be proved in
          Proposition~\ref{prop:subleading-interpret} as
         \begin{equation}\label{eq:subleading-loci}
                      \bigl\langle v, [\Gamma\cdot
                      M^{e}_{1/4}(X) \cap U^{\eta}_{-1/4}(b_{0})]
                    \bigr\rangle = \tau^{n-2|\eta|}\bigl\langle
                    v, [\Gamma\cdot M_{0}(X) \cap U^{\eta'}_{1/4}(b_{0})]
                    \bigr\rangle ,
          \end{equation}
          provided that the loci are transverse to the filtration of
          the space of Fredholm operators by the dimension of the
          kernel. The moduli spaces on $X$ should be obtained from
          metrics in the $B$-chamber as always.

                    We can obtain more information about
          $M^{e}_{1/4}(X)$ and the loci on both sides of
          \eqref{eq:subleading-loci} by interpreting the
          moduli space of bifold anti-self-dual connections as a
          moduli space of stable parabolic bundles on the pair $(X,
          \Sigma)$ where $\Sigma$ is the singular locus $B\times \pi
          \subset X$. To this end, we adopt the notation and results
          of \cite{KM-gtes2} to identify $M^{e}_{1/4}(X)$ with the
          moduli space of parabolic bundles $(\cE, \cL)$ with
          $\kappa=1/4$
          satisfying the parabolic stability condition
          with parameter $\alpha=1/4$. Here we can write $\kappa$ as
          $k+l/2$ following \cite{KM-gtes1, KM-gtes2}, where in this case
          \begin{equation}\label{eq:instanton-monopole-number}
                \begin{gathered}
                    k = (c_{2}(\cE) - \frac{1}{4} c_{1}(\cE)^{2})[X]
                    \\
                    l = (\frac{1}{2} c_{1}(\cE) - c_{1}(\cL) )[ \Sigma].
                \end{gathered}
          \end{equation}
          (The quantities $k$ and $l$ are the ``instanton number'' and
          ``monopole number'' in the notation of \cite{KM-gtes1}.)
          The rank-2 bundle $\cE$ should have $c_{1}(\cE)[B]$ odd, so
          we take
          \[
                \Lambda^{2}(\cE) = \cO(1,0),
          \]
          by which we mean the holomorphic line bundle with degree $1$
          on $B$. The moduli space $M_{0}(X)$ is similarly a moduli
          space of stable parabolic bundles on $X$, now with
          $\Lambda^{2}(\cE)=\cO$ and $\kappa=0$. These bundles are the
          pull-backs of the stable parabolic bundles on the curve $C=S^{2}_{n}$.

          The loci on either side of \eqref{eq:subleading-loci} have
          the following interpretations. Let $\cF\to
          C$ be
          the parabolic line bundle whose set of hits is $\eta$ and whose
          parabolic degree is $\pdeg{\cF}=1/4$. (See the remarks at
          the end of section~\ref{subsec:loci}. The dual parabolic
          bundle $\cF^{*}$ has parabolic degree $-1/4$ and its set of
          hits is $\eta'=\pi\setminus\eta$. Given a stable
          parabolic bundle $\cE$ on $X$, let $\cE_{b}$ be the
          parabolic bundle obtained by restriction to $\{b\} \times
          C$.          

          \begin{lemma}
            Let $\cF$ be the parabolic line bundle described above and
            $\cF^{*}$ its dual. Then:
            \begin{enumerate}
            \item
                the locus $M^{e}_{1/4}(X) \cap
                U^{\eta}_{-1/4}(b_{0})$ is the locus of stable parabolic
                bundles $\cE \in M^{e}_{1/4}(X)$ such that there
                exists a non-zero holomorphic map of parabolic bundles
                \[
                    \cF\to \cE_{b_{0}};
                \]
                \item
                               the locus $M_{0}(X) \cap
                U^{\eta'}_{1/4}(b_{0})$ is the locus of stable parabolic
                bundles $\cE \in M_{0}(X)$ such that there
                exists a non-zero holomorphic map of parabolic bundles
                \[
                    \cF^{*}\to \cE_{b_{0}}.
                \]
            \end{enumerate}
          \end{lemma}

          \begin{proof}
            These statements follow directly from the definitions.
          \end{proof}

                   Going beyond the above lemma, we have the
          following constructions for the relevant bundles.

         \begin{lemma}   \label{lem:bundle-loci-par}         
            \begin{enumerate}
            \item\label{item:loci1}
                The locus $M^{e}_{1/4}(X) \cap
                U^{\eta}_{-1/4}(b_{0})$ consists of parabolic bundles
                $\cE\to X = B\times C$ which are non-split extensions
                \[
                    \cO(1)\boxtimes \cF^{*} \to \cE \to \cF
                \]
                such that the extension class vanishes on
                $\{b_{0}\}\times C$.

                \item\label{item:loci2}
                               The locus $M_{0}(X) \cap
                U^{\eta'}_{1/4}(b_{0})$ is the locus  parabolic
                bundles $\cE \in M_{0}(X)$  which are non-split extensions
                \[
                  \cF^{*} \to \cE \to \cF.
                \]
            \end{enumerate}
            \noindent
            In both cases, all bundles obtained as such extensions are
            stable in the $B$-chamber on $X$.
          \end{lemma}

          \begin{proof}
          In \ref{item:loci2},
          the bundles in $M_{0}(X)$ are pulled from the stable
          parabolic bundles on $C$, and the existence of a non-zero
          map of parabolic bundles $\iota:\cF^{*}\to \cE$ is the
          definition of the locus $U^{\eta'}_{1/4}$. The map $\iota$
          must be an inclusion of a parabolic line sub-bundle, for
          otherwise this map would destabilize $\cE$.  So $\cE$ is an
          extension of parabolic line bundles as described. The
          extension must be non-split, for otherwise $\cE$ is
          destabilized by $\iota$.

          For \ref{item:loci1}, the first task is to verify that
          that every stable
          parabolic bundle
          in $M^{e}_{1/4}(X)$ in the $B$-chamber is a non-split extension
          \begin{equation}\label{eq:G-extension}
                                  \cO(1)\boxtimes \cG^{*} \to \cE \to
                                  \cG,
          \end{equation}
          where $\pdeg{\cG}=-1/4$ and the set of hits for $\cG$ is a
          subset $\xi\subset \pi$ which is arbitrary, except for the
          parity constraint \eqref{eq:h-parity}. There are $2^{n-1}$
          choices for $\xi$, and once $\xi$ is given, the non-split
          extensions are parametrized by a projective space, in this
          case of dimension $n-2$. In this way we find $2^{n-1}$
          copies of $\CP^{n-2}$ in $M^{e}_{1/4}$, and it is
          straightforward to see that these are disjoint, because a
          given bundle $\cE$ cannot be presented as an extension of
          this sort in two different ways. The verification that these
          $2^{n-1}$ copies of $\CP^{n-2}$ comprise the \emph{entire} moduli
          space $M^{e}_{1/4}(X)$ in the $B$-chamber is the holomorphic
          analog of wall-crossing result described in
          Proposition~\ref{prop:wall-crossing-summary}, and it
          is proved in essentially the same way. This is also the
          content of \cite[Proposition 2]{Munoz2} in the slightly
          different context of that paper, which serves the
          same purpose there.

          For an extension
          such as \eqref{eq:G-extension}, the restriction to
          $\{b_{0}\}\times C$ is an extension of parabolic line
          bundles on $C$,
          \[
                   \cG^{*} \to \cE_{b_{0}} \to
                                  \cG,
          \]
          and because $\pdeg(\cF) = \pdeg(\cG)> \pdeg(\cG^{*})$, there can
          be a non-zero
          map $\cF\to\cE_{b_{0}}$ only if $\cF=\cG$ and the
          extension class is zero on $\{b_{0}\}\times C$.
          \end{proof}

          The extensions that arise in \ref{item:loci2} are
          parametrized by the projective space
          \begin{equation}\label{eq:projective-ext}
                \mathbb{P}\left( H^{1}(C; (\cF^{*})^{\otimes 2})
                \right)
          \end{equation}
          where the cohomology group is interpreted as the
          cohomology of a sheaf on a bifold. The extensions that arise
          in \ref{item:loci1} are parametrized by the subset of the projective space
          \[
                     \mathbb{P}\left(H^{0}(B;\cO(1))\otimes H^{1}(C; (\cF^{*})^{\otimes 2})
                \right)
          \]
          corresponding to elements vanishing at $b_{0}$. If
          $Z_{b_{0}}\subset H^{0}(B;\cO(1))$ is the one-dimensional
          space of sections vanishing at $b_{0}$, then this is the
          space
                  \[
                     \mathbb{P}\left(Z_{b_{0}}\otimes H^{1}(C; (\cF^{*})^{\otimes 2})
                \right)
          \]
          which is canonically identified with
          \eqref{eq:projective-ext}. Both spaces of extensions are
          copies of $\CP^{m-1}$.

          We have now seen that there is a canonical identification of
          the two loci,
          \[
                  M^{e}_{1/4}(X) \cap U^{\eta}_{-1/4}(b_{0}) =
                  M_{0}(X) \cap U^{\eta'}_{1/4}(b_{0}),
          \]
          both of which are projective spaces. Furthermore, for any
          $b\ne b_{0}$ in $B$, the restrictions of the corresponding
          bundles in these loci to $\{b\}\times C$ agree. Indeed they
          are the same family of non-split extensions of $\cF$ by
          $\cF^{*}$ on $C$.
          The cohomology classes $v$ arising from elements of the
          algebra $\cA_{n}$ can be regarded as being pulled back via
          the restriction to $\{b\}\times C$, so it follows that the
          evaluation of such a class $v$ is the same in the two cases.

          Before accounting for the weights arising from the local
          system $\Gamma$, we therefore have an equality
          \begin{equation}\label{eq:equality-without-weights}
          \bigl\langle v, [M^{e}_{1/4}(X) \cap U^{\eta}_{-1/4}(b_{0})]
                    \bigr\rangle = \bigl\langle
                    v, [M_{0}(X) \cap
                    U^{\eta'}_{1/4}(b_{0})]\bigr\rangle.
          \end{equation}
          However, while $M^{e}_{1/4}(X) \cap U^{\eta}_{-1/4}(b_{0})$
          and $M_{0}(X) \cap
                    U^{\eta'}_{1/4}(b_{0})$ are both copies of
                    $\CP^{m-1}$ and are canonically identified, the
                    (constant) functions
        \[
\begin{aligned}
            \Gamma: M^{e}_{1/4}(X) \cap U^{\eta}_{-1/4}(b_{0}) &\to \cR
            \\
            \Gamma: M_{0}(X) \cap
                    U^{\eta'}_{1/4}(b_{0}) &\to \cR
\end{aligned}
        \]
        are different. The next lemma provides these values.

        \begin{lemma}
            \begin{enumerate}
                \item On $M_{0}(X) \cap
                    U^{\eta'}_{1/4}(b_{0})$, the value of\/ $\Gamma$ is
                    $1$.
                \item On $M^{e}_{1/4}(X) \cap
                U^{\eta}_{-1/4}(b_{0})$, the value of\/ $\Gamma$ is
                    $\tau^{n-2|\eta|}$.
            \end{enumerate}
        \end{lemma}
     
          \begin{proof}
            The singular set $\Sigma\subset X$ is a collection of spheres with
            trivial normal bundle, so there is no self-intersection
            term in the formula \eqref{eq:Gamma-A}, and we simply have
            \[
                        \Gamma(A) = \tau^{\nu(A)}
            \]
           where $\nu(A)$ is a 2-dimensional Chern-Weil integral on
            $\Sigma$. In the case of $M_{0}(X)$, the connections are
            flat and $\nu(A)=0$. So $\Gamma=1$ in this case, as stated
            in the first item of the lemma.

            In the case of a closed manifold, the value $\nu(A)$ is
            $-2l$ where $l$ is the ``monopole number'' of the bundle
            \eqref{eq:instanton-monopole-number}. The bundles that
            contributes to the moduli space $M_{1/4}(X)^{e}\cap
            U^{\eta}_{-1/4}(b_{0})$ are described in
            Lemma~\ref{lem:bundle-loci-par}. From there we read off
            that $c_{1}(\cE)[\Sigma_{p}]=1$ for each of the $n$
            components $\Sigma_{p}\subset \Sigma$, so that
            $c_{1}(\cE)[\Sigma]=n$. For $p\in \eta'$, the distinguished
            line subbundle $\cL \subset \cE|_{\Sigma_{p}}$ coincides
            with the image of the subbundle $\cO(1) \boxtimes\cF^{*}$
            on $\Sigma_{p}$, which has degree $1$. For $p\in\eta$, the
            distinguished line subbundle $\cL$ on $\Sigma_{p}$ maps isomorphically to
            the restriction of $\cF$ in the extension in
            Lemma~\ref{lem:bundle-loci-par}, so has degree $0$. In all
            then,
            \[
                    c_{1}(\cL) [\Sigma] = |\eta'|.
            \]
             The formula for the monopole number $l$ in
             \eqref{eq:instanton-monopole-number} therefore gives
             $(n/2) - |\eta'|$, which is $|\eta|-(n/2)$. Since
             $\nu(A)=-2l$, we have $\nu(A)=n-2|\eta|$, as the lemma
             claims.
             \end{proof}

             From the lemma, we see that
             \[
                    [\Gamma\cdot
                      M^{e}_{1/4}(X) \cap U^{\eta}_{-1/4}(b_{0})] =
                      \tau^{n-2h} [ M^{e}_{1/4}(X) \cap
                      U^{\eta}_{-1/4}(b_{0})]
             \]
             while
             \[
                     [\Gamma\cdot M_{0}(X) \cap
                     U^{\eta'}_{1/4}(b_{0})] = [M_{0}(X) \cap
                     U^{\eta'}_{1/4}(b_{0})].
             \]
             The equality to be proved in
             Proposition~\ref{prop:subleading-interpret} now follows
             from the unweighted equality
             \eqref{eq:equality-without-weights}, and this completes the
             proof of the Proposition.

             \begin{remark}
             In the course of these arguments, we have seen first that
             $M^{e}_{1/4}(X)$ is a disjoint union of $2^{n-1}$ copies of
             $\CP^{n-2}$ and second that the class $w^{m}_{n,\eta}$
             restricts to be non-zero on exactly one of them, being
             dual to a $\CP^{m-1}$ in exactly one of the copies of
             $\CP^{n-1}$. The components $\CP^{n-2}$ of
             $M^{e}_{1/4}(X)$ are in one-to-one correspondence with
             the subsets $\eta \subset \pi$ of the correct parity, so
             let us write them as $\CP^{n-2}_{\eta}$. If we choose a
             class $v$ which has non-zero pairing (say $1$) with each
             $\CP^{m-1}\subset \CP^{n-2}_{\eta}$, then we have
             \[
                        \bigl \langle w^{m}_{n,\eta}\cupprod v,
                        [\CP^{n-2}_{\xi}] \bigr\rangle= 
                        \begin{cases}
                            1,  & \eta=\xi \\
                            0,&  \text{otherwise,}
                        \end{cases}
             \]
             from which it follows that the classes $w^{m}_{n,\eta}$
             are linearly independent in $\cA_{n}$. This provides an
             alternative verification of the result used in the proof
             of
             Proposition~\ref{prop:little-w-generate}.
             \end{remark}

             \subsection{Passing to $Z_{n,-1}$}

          Recall that the algebra $\bar\cA$ is defined as the quotient
          of $\cA_{n}$ in which all the $\delta_{i}$ are equal (see
          equation
          \eqref{eq:barA-def}), and let $\oo^{k}_{n,\eta}\in \cA_{n}$
          be the elements from Definition~\ref{def:omega}. The image
          of $\oo^{k}_{n,\eta}$ in $\bar\cA$ depends only on the cardinality
          of the subset $\eta\subset\pi$, not
          otherwise on its elements, and we write this element of
          $\bar\cA$ as
          \begin{equation}\label{eq:baroo}
                    \bar \oo^{k}_{n,h} = \oo^{k}_{n,\eta} + \langle
                    \delta_{i} -
                    \delta_{j}\rangle_{i,j} \in \bar\cA
          \end{equation}
           when $|\eta|=h$.
          Recall from \eqref{eq:Zn1-raw} that we can write $I(Z_{n,-1})$ as $\cA_{n}/\JOne{n}$
          or as $\bar\cA/\barJOne{n}$ and that
          $\JOne{n}$ contains $\J{n}$
          (Proposition~\ref{prop:Zn1-as-quotient}).
          Proposition~\ref{prop:quantum-Mumford} and
          Proposition~\ref{prop:quantum-Mumford-subleading} therefore yield the
          following version for $Z_{n,-1}$.
          
          \begin{proposition}\label{prop:Mumford-Zn1}
          Write $n=2m+1$
          let $h$ be an integer satisfying the conditions
          \eqref{eq:h-parity-plus}, and let $\bar\oo^{m}_{n,h}$ be
          defined as above.
          Then there
          is an element $\bar\om^{m}_{h}\in\barJOne{n}$ of the
          filtered algebra $\bar\cA$ in filtration degree $m$ whose
          leading term is $\bar \oo^{m}_{n,h}$.
          The subleading term of $\bar\om^{m}_{h}$ is given
          by $\epsilon \oo^{m-1}_{n,h'}$, where $h'=n-h$. Thus
          \[
                \bar\om^{m}_{h}=\bar\oo^{m}_{n,h}  + \epsilon
                \bar\oo^{m-1}_{n,h'} \pmod {\bar\cA^{(m-2)}}.
          \]
          The element $\bar\om^{m}_{h}$ in $\bar\cA$ is the image of\/
          $\om^{m}_{\eta}\in\J{n}$ under the quotient map $\cA_{n}\to\bar\cA$.
          \qed
          \end{proposition}

          We have not yet established that $\barJOne{n}$ is the image
          of $\J{n}$, so we do not know yet that the elements
          $\bar\om^{m}_{h}$ generate the ideal of relations
          $\barJOne{n}$ for $I(Z_{n,-1})$. We turn to this next.
          
          \begin{proposition}\label{prop:omegas-generate}
          When $n=2m+1$,
            the elements $\bar\om^{m}_{h}$ for $h$ in the range $0\le
          h\le n$ with $h=(n+1)/2$
          mod $2$ are a set of generators for the ideal
          $\barJOne{n}\subset\bar\cA$.
          In particular, $\barJOne{n}$ is the image of $\J{n}$ in
          $\bar\cA$.          
          \end{proposition}

          \begin{proof}
            The quotient $\bar\cA/\barJOne{n}$ is $I(Z_{n,-1})$ which
            we know to be a free $\cR$-module of rank $(n^{2}-1)/4$ by
            Corollary~\ref{cor:rank-Znm1}. If $\cJ'\subset\barJOne{n} $ denotes the ideal
            generated by the elements $\bar\om^{m}_{h}$, then the
            desired equality $\cJ'=\barJOne{n}$ will follow if we can
            prove that $\bar\cA/\cJ'$ has the same rank. The leading
            $m$th-degree
            terms of the elements $\bar\om^{m}_{h}$ are the elements
            $\bar\oo^{m}_{n,h}$, so let us denote by $\bar
            J_{n}\subset\bar\cA$ the ideal generated by these leading
            terms. (This is the image in $\bar\cA$ of the ideal of
            relations
            $J_{n}\subset\cA_{n}$ for the ordinary cohomology ring $H^{*}(\Rep(Z_{n});\cR)$
          \eqref{eq:HRep-as-quotient}.) It will therefore suffice to
          show that $\cA/\bar J_{n}$ has rank $(n^{2}-1)/4$, and this
          is the content of the lemma below, which completes the
          proof.
          \end{proof}
          
          \begin{lemma}\label{lem:power-maximal}
          Write $n=2m+1$ again and let $\bar J_{n} \subset
          \bar\cA$ be as above, generated by the elements
          $\bar\oo^{m}_{n,h}$.
          Then $\bar J_{n}$ is the $m$-th
          power $\langle \alpha, \delta\rangle^{m}$ of the ideal $\langle \alpha, \delta
          \rangle$. In particular, the rank of $\bar\cA/ \bar J_{n}$
          is $m(m+1)$, which is also equal to $(n^{2}-1)/4$.
          \end{lemma}

          \begin{remark}
            The quotient of a polynomial algebra in two variables by
            the $m$-th power of the maximal ideal at $0$ has rank
            $m(m+1)/2$. The extra factor of two in the lemma arises
            because of the extra generator $\epsilon$ in the algebra
            $\bar\cA$.
          \end{remark}

          \begin{proof}[Proof of the Lemma]
          Recall that $\oo^{m}_{n,\eta}$ arises from the formal computation
          of $c_{m}(-\ind(P))$, where $P$ is a family of Fredholm
          operators, Definition~\ref{def:omega}.
          The formula \eqref{eq:ch-minus-ind-d1} for the Chern
          character of $-\ind(P)$ becomes the following, after passing
          to the formal completion of the quotient ring $\bar\cA$ in
          which all the $\delta_{i}$ are equal:
                         \begin{equation}\label{eq:ch-minus-ind-bar}
           (m-1 ) \cosh(\delta)
                  \null + \frac{\sinh(\delta)}{\delta}\left( \alpha +
                  (h-n/2)\delta\right).
        \end{equation}
        Passing from the Chern character to the $m$-th Chern class, we
        find that the image of $c_{m}(-\ind(P))$ in $\bar\cA$ has the form
        \[
                    V_{m} ( B_{h}, \delta)
        \]
        where $V_{m}(x,y)$ is a homogeneous polynomial of degree $m$ in two
        variables and $B_{h} = \alpha + (-h+n/2)\delta$. Furthermore the
        coefficient of $x^{m}$ in $V_{m}$ is $1/m!$.

        Thus $\bar
        J_{n}$ is generated by the elements $V_{m}(B_{h},\delta)$, for
        $h$ in the range $0\le
          h\le n$ with $h=(n+1)/2$
          mod $2$. The assertion of the lemma is equivalent to the
          statement that the homogeneous polynomials $V_{m}( x +
          (h - n/2)y, y)$ in $\Q[x,y]$ span the space of homogeneous
          degree-$m$ polynomials. This in turn is true because $h-n/2$
          runs through $m+1$ distinct values in $\Q$ as $h$ runs
          through its allowed range.  (This is the same assertion as the
          statement that any $m+1$ distinct translates of a polynomial
          $f(x)$ of degree $m$ are necessarily independent.)
           \end{proof}

          \section{Calculation of the ideals}
          \label{sec:calculation}
          
                    \subsection{Hilbert schemes of points in the
                    plane} \label{subsec:hilb}

           We present here and in section~\ref{subsec:syz} below some
           material on Hilbert schemes of points in the plane,
           specialized to the particular situation for which we have
           application. General references are \cite{Miller-Sturmfels}
           for section~\ref{subsec:hilb}
           and \cite{Eisenbud-Geometry-of-Syzygies} for
           section~\ref{subsec:syz}.

           Let $A$ be the algebra $k[x,y]$, with $k$ a field, which we
           may take to be $\C$. Let
           $A_{n} \subset A$ be the
           subspace of homogeneous polynomials of degree $n$, and let
           $A^{(n)}=\oplus_{k\le n} A_{k}$. Let $\m\subset A$ be the
           maximal ideal $\langle x,y\rangle$, and consider the $m$'th
           power
           $\m^{m}$, which has generators
           \begin{equation}\label{eq:max-m-generators}
                \m^{m} =\langle x^{m}, x^{m-1}y, \dots, y^{m}\rangle.
           \end{equation}
           The colength of $\m^{m}$ (the dimension of the quotient
           $A/\m^{m}$ as a $k$-vector space) is $N = m(m+1)/2$, and a
           vector space complement is the linear subspace $A^{(m-1)}$:
           \[
                    A = \m^{m} \oplus A^{(m-1)}
           \]

           We
           can consider $\m^{m}$ as defining a point in the Hilbert
           scheme $\Hilb^{N}$ which parametrizes all ideals of
           colength $N$ in $A$. In the Hilbert scheme, there is an
           open neighborhood $U\ni \m^{m}$ defined as
           \begin{equation}\label{eq:U-hilb}
                    U = \{\, I\in \Hilb^{N} \mid A= I \oplus A^{(m-1)}
                    \,\}.
            \end{equation}
            For $I\in U$, there is the projection  to the second factor,
            $A\to A^{(m-1)}$ with kernel $I$:
           \[
                          \phi_{I}:  A \to A^{(m-1)}.
            \]
           It is an elementary matter to check that the restriction of
           $\phi_{I}$ to $A_{m}$
           completely determines $I$, and that $I$ is in fact
           generated by
           \[
                    I = \langle \, a - \phi_{I}(a) \mid a\in A_{m}
                    \,\rangle.
           \]
           We have in particular $a = \phi_{I}(a)$ mod $I$, for all $a\in
           A_{m}$.
           
           The map $\phi = \phi_{I}$ is constrained by the condition
           that its kernel is an ideal rather than just a codimension-$N$
           linear subspace in $A$. To see how, consider elements
           $a,a'\in A_{m}$ with
           \[
                x a = y a'.
           \]
           We have $a =  \phi(a)$ mod $I$, and therefore $xa =
           x\phi(a)$, and applying $\phi$ again
           \[
                       xa = \phi(x\phi a) \pmod I.
           \]
           Similarly with $ya'$ so $\phi(y\phi a') = \phi(x\phi a)$ mod $I.$
        However both sides of the last equality lie in
           the complementary subspace $A^{(m-1)}$, so in fact
           \begin{equation}\label{eq:phi-constraint}
                \phi(y\phi a') = \phi (x \phi a).
           \end{equation}
           Conversely, 
           if we are given a linear map $\psi: A_{m}\to A^{(m-1)}$
           satisfying the constraint \eqref{eq:phi-constraint}, then
           there exists a unique (well-defined) extension to a linear
           map $\phi:A\to A^{(m-1)}$ characterized by $\phi(x^{i}y^{j}
           a) =  \phi(x^{i}y^{j}\phi(a))$, and the kernel of $\phi$ is
           then an ideal $I$ belong to $U\subset \Hilb^{N}$.
           
           To expand on the constraint \eqref{eq:phi-constraint}, write
           \[
                \phi|_{A_{m}} = \phi_{1} + \phi_{2} + \dots + \phi_{m}
           \]
           where $\phi_{r}: A_{m}\to A_{m-r}$, and use the fact that
           $\phi|_{A_{k}}=1$ for $k<m$ to  obtain
           \[
                \phi(y \phi_{1}(a')) + y \phi_{2}(a') + \dots
                y\phi_{m}(a') = \phi(x \phi_{1}(a)) + x \phi_{2}(a) + \dots
                x\phi_{m}(a)  .
           \]
           Finally compare terms of like degree to see that
           \begin{equation}\label{eq:r-part}
                        y \phi_{r+1}(a') - x \phi_{r+1}(a) =
                        -\phi_{r}(y \phi_{1}(a') + \phi_{r}(x
                        \phi_{1}(a)
           \end{equation}
            for all $r\ge 1$ and all $a,a'\in A_{m}$ with $ya'=xa$.
            If we write $a'=xb$ and $a=yb$ for $b\in A_{m-1}$, the
            constraint becomes
            \[
                    y \phi_{r+1}(xb) - x\phi_{r+1}(yb) =
                                -\phi_{r}(y \phi_{1}(xb)) + \phi_{r}(x
                        \phi_{1}(yb))
            \]
            which we can express as
            \begin{equation}\label{eq:L-op}
                        L_{r}(\phi_{r+1}) = Q_{r}(\phi_{1}, \phi_{r}),
            \end{equation}
            where $L_{r} : \Hom(A_{m}, A_{m-r-1}) \to \Hom(A_{m-1},
            A_{m-r})$ is a linear map and $Q_{r}$ is a bilinear
            expression. It is easy to verify that the operator $L_{r}$
            is injective (see below), so the constraints determine $\phi_{r+1}$
            once  $\phi_{r}$ and
           $\phi_{1}$ are known. 

           We have shown:

           \begin{lemma}
            Given a $k$-linear map $\phi_{1}: A_{m}\to A_{m-1}$ there
            exists at most one linear map $\phi = \phi_{1} + \phi_{2}
            + \dots + \phi_{m}$, with $\phi_{r}: A_{m} \to A_{m-r}$,
            such that constraints \eqref{eq:r-part} hold. The ideal
            $I$ generated by the elements $\{\, a - \phi(a) \mid a\in
            A_{m}\,\}$ then belongs to the open set $U\subset
            \Hilb^{N}$. Every ideal in $U$ arises in this way. \qed
           \end{lemma}

           The lemma exhibits $U$ as a closed subset of the vector
           space $\Hom_{k}(A_{m}, A_{m-1})$, which has dimension
           $m(m+1) = 2N$. This subset is also invariant under the
           action by scalars. It will follow that $U \cong \Hom_{k}(A_{m},
           A_{m-1})$ if it can be shown that $U$ has dimension $2N$.
           To do this, one can show that $U$ contains an ideal $I$
           whose zero set consists of $N$ distinct points in the plane
           $k^{2}$. Such an ideal can be realized as the
           ``distraction'' of $\m^{m}$. This is the ideal $I$ generated
           by the elements
           \[
                        u_{h} = \left(\prod_{0\le j < h}
                        (x-j)\right)\left( \prod_{0\le l
                        < m-h}
                        (y-l)\right), \quad h=0,\dots, m,
           \]
           (allowing that one of the products may be empty).
           Its zero-set is the set of lattice points $(j,l)$ in the
           first quadrant with
           $j+l < m$.

            \begin{proposition}
            Given a $k$-linear map $\phi_{1}: A_{m}\to A_{m-1}$ there
            exists exactly one linear map $\phi = \phi_{1} + \phi_{2}
            + \dots + \phi_{m}$, with $\phi_{r}: A_{m} \to A_{m-r}$,
            such that the ideal
            $I$ generated by the elements $\{\, a - \phi(a) \mid a\in
            A_{m}\,\}$ has colength $N$. The matrix entries of
            $\phi_{r}$ $(r \ge 2)$ can be expressed as  polynomials in
            the matrix entries of $\phi_{1}$. \qed
           \end{proposition}

           The proposition tells us that, at each stage $r$ in the
           equations \eqref{eq:L-op}, the right-hand side
           $Q_{r}(\phi_{1},\phi_{r})$ is in the image of the linear
           operator $L_{r}$. If we choose a right-inverse $P_{r}$ for
           $L_{r}$, then we can express the iterative solution as
           \begin{equation}\label{eq:P-iter}
                    \phi_{r+1} = P_{r} Q_{r}( \phi_{1},\phi_{r}).
           \end{equation}
           To give $P_{r}$ explicitly,
           let us temporarily make our polynomials inhomogeneous by
           setting $y=1$, so identifying $A_{m}$ with the polynomials
           in $x$ of degree at most $m$, and let us write
           \[
                    u_{k} = \phi_{r+1}(x^{k})
           \]
           as a polynomial of degree at most $m-r-1$ in $x$. Then the
           equations \eqref{eq:L-op} take the form
           \[
                     u_{k+1} - x u_{k} = v_{k}
           \]
           for $k=0,\dots, m-1$, where $v_{k}$ is a given polynomial in
           $x$ of degree at most $m-r$ and the equations are to be
           solved for $u_{k}$ of degree at most $m-r-1$. If a solution
           exists, then
           \[
                u_{m} = v_{m-1} + x v_{m-2} + \dots + x^{m-1}v_{0} +
                x^{m} u_{0}.
           \]
           Since all polynomials $u_{k}$ and $v_{k}$ here have
           degree less than $m$, this equation determines the
           coefficients of $u_{0}$ as linear combinations of the
           coefficients of the $v_{k}$:
           \[
                    u_{0} = - (x^{-m}v_{m-1} + x^{-m+1} v_{m-2} +
                    \dots + x^{-1}v_{0})_{+},
           \]
           where the subscript $+$ means to discard the negative
           powers of $x$. Having found $u_{0}$, we can
           express the complete solution, if it exists, by the recurrence
           \[
                    u_{k+1} = \mathrm{trunc}_{m-r-1}(v_{k} + x u_{k})
           \]
           where $\mathrm{trunc}_{m-r-1}$ is the truncation of the polynomial
           to the given degree. Whether or not a solution exists,
           this process defines $u_{k}$ as a linear function of the
           $v$'s, and so defines a right inverse $P_{r}$ for the
           linear map $L_{r}$. In this
           form, the coefficients of $P_{r}$ are integers,
           and this allows us to pass to any ring. These leads to the
           following  version.

           \begin{proposition}\label{prop:Hilbert-R}
            Let $R$ be a Noetherian ring, let $A=R[x,y]$ and let
            $I\subset A$ be an ideal such that 
            \begin{itemize}
                \item $A/I$ is a free $R$-module of rank $N=m(m+1)/2$;
                \item there is an $R$-module homomorphism $\phi :
                A_{m}\to A^{(m-1)}$ such that $a-\phi(a)\in I$ for all
                $a\in A_{m}$.
            \end{itemize}
            Then $I$ is generated by the elements $a-\phi(a)$ for
            $a\in A_{m}$. Furthermore, if we write
            \[
                        \phi = \phi_{1} + \phi_{2} + \dots + \phi_{m}
            \]
            with $\phi_{r}: A_{m}\to A_{m-r}$, then $\phi_{r}$ for
            $r\ge 2$ is determined by $\phi_{1}$ through the iteration
            \eqref{eq:P-iter}. This establishes a bijection between
            ideals $I$ satisfying the above two constraints and module
            homomorphisms
            $\phi_{1} : A_{m}\to A_{m-1}$.
           \end{proposition}

           \begin{proof}
            If $I$ satisfies the second condition, then the relations
            $a=\phi(a)$ mod $I$ show that the map $A^{(m-1)}\to A/I$
            is surjective. The first of the two conditions tells us
            that these are free $R$-modules of equal rank, and it
            follows that the map is an isomorphism because $R$ is
            Noetherian. Thus we have a
            direct sum decomposition $A=I \oplus A^{(m-1)}$. As before, the
            constraints then lead to the relations \eqref{eq:P-iter} which
            determine $\phi_{r}$ for $r\ge 2$.  
           \end{proof}

           \subsection{Syzygies}
\label{subsec:syz}
           Proposition~\ref{prop:Hilbert-R}, which determines $\phi$
           entirely in terms of $\phi_{1}$,
           will be applied in section~\ref{subsec:application-hilb} to
           see that the generators $\bar\om^{m}_{h}$ of the ideal
           $\barJOne{n}$ can be determined completely in terms of the
           leading and subleading terms. (The subleading terms are already supplied by
           Proposition~\ref{prop:Mumford-Zn1}.) This will provide a
           complete description of the instanton homology
           $I(Z_{n,-1})$. First however, we pursue further our
           discussion of the Hilbert scheme of points in the plane, to
           explain that the way in which $\phi_{1}$ determines $\phi$
           can be packaged by considering the syzygies of the
           module $A/I$. This will lead to quite explicit formulae for
           the generators.
           
           We return temporarily to the case $A=k[x,y]$ as above, and
           we take $k=\C$. Fix $m$ again and write $N=m(m+1)/2$. Let
           $U\subset \Hilb^{N}$ be as before \eqref{eq:U-hilb}. An
           ideal $I\in U$ contains no non-zero polynomials of degree
           less than $m$ and is generated by $m+1$ elements whose
           leading terms are a basis for $A_{m}$. Choose a basis for
           $A_{m}$ so as to identify $A_{m}=A^{\oplus (m+1)}$, say the
           monomial basis \eqref{eq:max-m-generators}. We then have
           generators for $I$ in the form
           \[
                    g_{i} = x^{m-i}y^{i} - \phi( x^{m-i}y^{i}).
           \]
           Because $A$ has dimension $2$, a resolution of $A/I$ has
           only one more step, and we therefore have a presentation of
           the ideal $I$ in the form
           \begin{equation}\label{eq:syzygy-presentation}
                0 \to A^{\oplus m} \stackrel{S}{\to} A^{\oplus (m+1)}
                \stackrel{g}{\to} I \to 0.
           \end{equation}
           Here $g = (g_{i})$ is given by the generators (the
           relations in $A/I$) and $S$ is the matrix of syzygies.

           In the special case that $I=\m^{m}$ and
           $g_{i}=x^{m-i}y^{i}$ the syzygy matrix can be taken to be
           \begin{equation}\label{eq:S0}
                    S_{0} = 
                    \begin{pmatrix}
                        -y  & 0 & 0 & \dots & 0 \\
                        \phantom{+}x & -y & 0 & \dots & 0 \\
                        0 & \phantom{+}x & -y & \dots & 0 \\
                        \vdots & \vdots & \vdots & \ddots & \vdots \\
                        0 & 0 & 0 & \dots & -y \\
                        0 & 0 & 0 & \dots & \phantom{+}x 
                    \end{pmatrix}.
           \end{equation}

           \begin{lemma}
            For a general $I\in U$, the syzygy matrix $S$ has the form
            $S=S_{0} + S_{1}$, where $S_{0}$ is as above and $S_{1}$
            is a matrix of scalars (polynomials of degree $0$).
           \end{lemma}

           \begin{proof}
            Write $g = g(0)+g(1)+ \dots + g(m)$, where $g(r)$ is a
            vector of homogeneous polynomials of degree $m-r$ and $g(0)$ is the
            basis of monomials of degree $m$. (So the entries of
            $g(r)$ are the polynomials $-\phi_{r}(x^{m-i}y^{i})$.) Let
            \[
                    g^{t} = g(0) + t g(1) + t^{2} g(2) + \cdots,
            \]
            and let $I^{t}$ be the ideal generated by the entries of
            $g^{t}$. Because the co-length of $I=I^{1}$ is the same as
            that of $I^{0}$, this is a flat family, and the syzygy
            matrix $S^{0}$ for $g^{0}$ therefore lifts to a syzygy
            matrix $S^{t}$, whose entries are polynomials in $(x,y,t)$
            and which coincides with $S^{0}$ at $t=0$. Because the
            entries of $g^{t}$ are homogeneous (of degree $m$) in
            $(t,x,y)$, we may assume that $S^{t}$ is also homogeneous.
            Since $S_{0}$ has homogeneous degree $1$, so too does
            $S^{t}$, and it follows that
            \[
                        S^{t} = S_{0} + t S_{1},
            \]
            where the entries of $S_{1}$ have degree $0$ in $(x,y)$.
            \end{proof}

            Note that in the above lemma, the matrix $S_{1}$ is
            entirely determined by the leading term $g(0)$ and the
            subleading term $g(1)$ (or equivalently by $\phi_{1} :
            A_{m} \to A_{m-1}$) via the condition
            \begin{equation}\label{eq:S1-condition}
                    g(0) \cdot S_{1} + g(1) \cdot S_{0} = 0.
            \end{equation}
            Quite concretely, taking $g(0)$ to be again the standard
            monomial basis, taking $S_{0}$ as above, and writing the
            subleading terms $g_{i}(1)$ as
            \[
                        g_{i}(1) = \sum_{j=0}^{m-1} G_{ij} x^{m-1-j}
                        y^{j}
            \]
            then
            \begin{equation}\label{eq:S1}
                        S_{1} = 
                        \begin{pmatrix}
                             -G_{1,0} &
                           -G_{2,0} & \dots &
                           -G_{m,0} \\
                            G_{0,0}-G_{1,1} & G_{1,0}-G_{2,1} & \dots &
                            G_{m-1,0}-G_{m,1}\\
                            \vdots & \vdots & \ddots & \vdots \\
                            G_{0,m-2}-G_{1,m-1} & G_{1,m-2}-G_{2,m-1}
                            & \dots & G_{m-1,m-2} - G_{m, m-1} \\
                            G_{0,m-1} &
                            G_{1,m-1} & \dots &
                            G_{m-1,m-1} & 
                        \end{pmatrix}
            \end{equation}

            \begin{proposition}
                Let $S=S_{0}+S_{1}$ be the syzygy matrix as above, so
                that $S_{0}$ is the matrix of syzygies of the standard
                monomial
                ideal $\m^{m}$ and $S_{1}$ is determined by the
                subleading terms $g_{i}(1)$ by \eqref{eq:S1}. Then the generators
                $g_{0}, \dots, g_{m}$ of the ideal $I$ are precisely
                the $m\times m$ minors of the $(m+1)\times m$ matrix
                $S$ (i.e. the determinants of the matrices obtained by
                deleting a single row of $S$, with alternating sign).
            \end{proposition}

            \begin{proof}
                Let $h=(h_{0}, h_{1}, \dots, h_{m})$ be the minors. We
                have both $h\cdot S=0$ (by standard properties of
                determinants) and $g\cdot S=0$ (by construction), and
                it follows that $a h = b g$ for some $a$ and $b$ in
                $A$, because the rank of the kernel of $S^{\mathrm{T}}$ is $1$.
                On the other hand, by inspection, the leading
                term of $h_{i}$ is the same as that of $g_{i}$, namely
                $x^{m-i}y^{i}$. So $h=g$.
            \end{proof}

            Finally, we can pass from the case of $k[x,y]$ to more
            general coefficients without difficulty. The next
            proposition summarizes the situation.

            \begin{proposition}\label{prop:syz-R}
            As in Proposition~\ref{prop:Hilbert-R}, 
            let $R$ be a Noetherian ring, let $A=R[x,y]$ and let
            $I\subset A$ be an ideal such that 
            \begin{itemize}
                \item $A/I$ is a free $R$-module of rank $N=m(m+1)/2$;
                \item there is an $R$-module homomorphism $\phi :
                A_{m}\to A^{(m-1)}$ such that $a-\phi(a)\in I$ for all
                $a\in A_{m}$.
            \end{itemize}
            Let $(g_{0}(0),\dots ,g_{m}(0))$ be a basis for $A_{m}
            \cong
            A^{\oplus (m+1)}$ and let
            \[
                        \begin{aligned}
                            g_{i} &= g_{i}(0) - \phi (g_{i}(0)) \\
                                  &= g_{i}(0) + g_{i}(1) + g_{i}(2) +\dots +
                                  g_{i}(m)
                        \end{aligned}
            \]
            where $g_{i}(j)$ is homogeneous of degree $m-j$. 
            Then the elements $(g_{0}, \dots ,
            g_{m})$ are generators of the ideal $I$.
            Furthermore, let $S_{0}$ be a matrix of
            syzygies for the leading parts $g_{i}(0)$, with entries
            which are homogeneous of degree $1$, and let $S_{1}$
            be the matrix of scalars determined by the subleading
            parts $g_{i}(1)$ via equation \eqref{eq:S1-condition}.
            Then:
            \begin{enumerate}
                \item the matrix $S=S_{0}+S_{1}$ is the matrix of syzygies
            for the generators $(g_{0}, g_{1},\dots, g_{m})$ of the
            ideal $I$ ;
            \item if $h_{0},
            \dots, h_{m}$ are the $m\times m$ minors of the matrix
            $S$, then $(h_{0}, h_{1},\dots, h_{m})$ is a set of
            generators for $I$;
            \item if $S_{0}$ is chosen so that its minors are the
            leading terms $(g_{0}(0), \dots, g_{m}(0))$, then the
            generators $g_{i}$ for $I$ are equal to the minors $h_{i}$
            of $S$.
            \end{enumerate}
            In this way, the generators $g$ are determined by their
            leading and subleading terms, $g(0)$ and $g(1)$.
           \end{proposition}

           \begin{proof}
            We may take it that $g(0)$ is the standard monomial basis
            and that $S_{0}$ is given \eqref{eq:S0}. The matrix
            $S_{1}$ is then given by \eqref{eq:S1} where the terms
            $G_{i,j}$ are the coefficients of the subleading terms
            $g(1)$. According to Proposition~\ref{prop:Hilbert-R},
            the lower terms in the entries of $g$ are expressible as
            universal polynomials in the coefficients of $g(1)$. On
            the other hand, the recipe in terms of the minors of $S$
            expresses the lower terms of $g$ as polynomials in the
            coefficients of $g(1)$, at least when $R$ is a field $k$.
            The polynomials occurring in the minors have integer
            coefficients, and must coincide with the polynomials in
            Proposition~\ref{prop:Hilbert-R}.
           \end{proof}
           
           \subsection{Equations for the curve $D_{n}$}
           \label{subsec:application-hilb}

           Let $\cR=\Q[\tau,\tau^{-1}]$. Let $R$
           temporarily denote the ring \[ R = \cR[\epsilon]/\langle
           \epsilon^{2}-1 \rangle. \] The algebra $\bar\cA$ in
           \eqref{eq:barA-def} is
           $R[\alpha,\delta]$ and the instanton homology $I(Z_{n,-1})$
           is described as a quotient $\bar\cA/ \barJOne{n}$ in
           \eqref{eq:Zn1-raw}. We know that $I(Z_{n,-1})$ is a free
           $\cR$-module of rank $(n^{2}-1)/4 = m(m+1)$ from
           Corollary~\ref{cor:rank-Znm1}, and it is a free $R$-module
           of rank $m(m+1)/2$.  We know that there are elements
           $\bar\om^{m}_{h}$ in $\barJOne{n}$ of degree $m$ in
           $(\alpha,\delta)$ having the form
         \begin{equation}\label{eq:model-om}
         \begin{aligned}
         \bar\om^{m}_{h} &= w(0)_{h} + \epsilon w(1)_{h} +
                                \dots \\
                                &=\bar\oo^{m}_{n,h}  + \epsilon
                \bar\oo^{m-1}_{n,h'} + \cdots
                \end{aligned}
                \end{equation}
           (see Proposition~\ref{prop:Mumford-Zn1}).
           The leading and subleading terms $w(0)$
           and $\epsilon w(1)$ are known from
           Proposition~\ref{prop:Mumford-Zn1} and
           Definition~\ref{def:omega}.  We also know that
           the leading terms $w(0)_{h}$ are a basis for the $m$'th
           power of the maximal ideal, $\langle\alpha, \delta\rangle^{m}$,
           by Lemma~\ref{lem:power-maximal}.

           The ideal $\barJOne{n}\subset R[\alpha,\delta]$ therefore
           satisfies the hypotheses of
           Proposition~\ref{prop:Hilbert-R} and
           Proposition~\ref{prop:syz-R}. In the notation of
           Proposition~\ref{prop:syz-R}, we know $\phi_{1}$ explicitly, as it is
           determined by the subleading terms $\epsilon w(1)_{h}$. We
           therefore have the following result as a corollary. In this
           statement, we write $n=2m+1$ as usual.

           \begin{theorem}\label{thm:Jbar-main}
           Let
           $\barJOne{n}$ be the ideal of relations for the
              instanton homology $I(Z_{n,-1})$ with local
                coefficients, and let
                \[
                        \bar\om^{m}_{h} = w(0)_{h} + w(1)_{h} + \dots
                        + w(m)_{h}, \qquad (0 \le h \le n, \; h =
                        m+1 \bmod 2),
                        \]
                be the generators for this ideal, as in \eqref{eq:model-om}.
            There are explicit polynomial formulae which express the
            coefficients of all the lower terms $w(r)_{h}$ for $r\ge
            2$  in terms of the
            leading and subleading terms
            \[
            \begin{aligned}
            w(0)_{h}& = \bar\oo^{m}_{n,h},\quad  \text{and}\\
                w(1)_{h} &= \epsilon
                \bar\oo^{m-1}_{n,n-h}
                \end{aligned}
            \]
            in
                Proposition~\ref{prop:Mumford-Zn1}. If the syzygy matrix
                \[
                        S = S_{0} + S_{1}
                \]
                is constructed as in Proposition~\ref{prop:syz-R}, as
                 a matrix whose entries are inhomogeneous linear forms
                 in $(\alpha,\delta)$ with coefficients in
                 $R=\cR[\epsilon]/\langle \epsilon^{2}-1 \rangle$,
                then the generators
                $\bar\om^{m}_{h}$ are the $m\times m$ minors of $S$.
                \qed
           \end{theorem}

           To obtain a final form for the generators, we now need to
           find an explicit formula for the syzygy matrix $S$,
           starting from our formulae for $w(0)_{h}$ and $w(1)_{h}$.
           In section~\ref{subsec:syz} above, we illustrated the
           calculation when the leading terms of the generators were the
           standard monomial basis in the polynomials in two
           variables, so that the term $S_{0}$ was the standard syzygy
           matrix \eqref{eq:S0}. The leading terms $w(0)_{h}$ are not
           monomials in our case, so we must first write down a
           suitable matrix of syzygies $S_{0}$ for these.

 From Proposition~\ref{prop:explicit-formulae},  on setting
           all $\delta_{i}$ equal to $\delta$ to pass from the ring
           $\cA_{j}$ to $\bar\cA$, we obtain an expression for
           $w(0)_{h} = \bar\oo^{m}_{n,h}$ as a product of linear
           factors. It is convenient to remove the combinatorial
           factor of $1/m!$ and write
           \[
           \begin{aligned}
           g(0)_{h} &= m! w(0)_{h} \\
                             &= m! \bar\oo^{m}_{n,h},
                             \end{aligned}
           \]
           for which Proposition~\ref{prop:explicit-formulae} yields
           the formula
           \[
           \begin{aligned}
           g(0)_{h} &= \prod_{\substack{j=-m+1 \\ j =
                    -m+1 \bmod 2 }}^{m-1} \bigl(\alpha + (2h-n - 2j)
                    \delta/2\bigr), \\
               &= \bigl(\alpha + (2h-3)\delta/2\bigr)\bigl(\alpha + (
               2h-7)\delta/2\bigr) \cdots \bigl( \alpha + ( 2h-4m+1) \delta/2\bigr)
               .
                    \end{aligned}
                    \]
           We introduce some abbreviated notation, setting
           \[
                           \begin{aligned}
                            L(k)&= (\alpha + k \delta/2),\\
                            P(k,l)&= L(k)L(k+4)L(k+8) \cdots L(l).
                           \end{aligned}
           \]
           (The latter notation will be used only when $k=l\bmod 4$.)
           Then we can write,
           \[
           g(0)_{h} = P(2h-4m+1, 2h-3).
           \]
           If we compare $g(0)_{h}$ to $g(0)_{h+2}$, only the first
           and last factors in this product differ, so we have a
           relation
           \[
                -L(2h+1) g(0)_{h} +
                 L(2h-4m+1) g(0)_{h+2}
                 = 0.
           \]
           That is, for $h'$ in the range $0\le h' \le n-2$ with $h' =
           m+1 \bmod 2$, we have
           \[
                        \sum_{h} S^{h'h}_{0} g(0)_{h} = 0,
           \]
           where
           \begin{equation}\label{eq:S0-answer}
                        S^{h'h}_{0} = 
                        \begin{cases}
                            - L(2h'+1),
                            & h=h',\\
                            L(2h'-4m+1), & h=h'+2,\\
                            0, &\text{otherwise}.
                        \end{cases}
           \end{equation}
            This is therefore the leading part $S_{0}$ of the required
            syzygy matrix $S= S_{0}+S_{1}$. It is straightforward to
            verify that the minors of $S^{h'h}_{0}$ are the terms
            $g(0)_{h}$, as required. 

            We normalize the subleading terms just as we did the
            leading terms, so that
            \[
            \begin{aligned}
                g(1)_{h} &= m! w(1)_{h} \\
                             &= m! \epsilon\tau^{n-2h}            
                             \bar\oo^{m-1}_{n,n-h},
                             \end{aligned}
            \]
            from Proposition~\ref{prop:quantum-Mumford-subleading}.
            We then have the explicit formulae again from
            Proposition~\ref{prop:explicit-formulae} (noting that
            $|\eta'| = n-h$),
                     \[
           \begin{aligned}
           g(1)_{h} &=
          m \epsilon\tau^{n-2h} 
           \prod_{\substack{j=-m+2 \\ j =
                    m \bmod 2 }}^{m-2} \bigl(\alpha + (n-2h - 2j)
                    \delta/2\bigr), \\
               &\begin{aligned} = m \epsilon\tau^{n-2h} 
               P(-2h+5, -2h + 4m-3)
              .
               \end{aligned}
                    \end{aligned}
                    \]
           To obtain the other term $S_{1}$ in the syzygy matrix, we
           need to solve  the following equations for $S_{1}^{h'h}$:
           \[
                 \sum_{h} S^{h'h}_{1} g(0)_{h} + \sum_{h} S^{h'h}_{0}
                 g(1)_{h} =0,
           \]
           ($h, h' = m+1 \bmod 2$, $0\le h \le n$ and $0\le h' \le
           n-2$). Using the formulae for $g(0)_{h}$, $g(1)_{h}$ and
           $S_{0}^{h'h}$, we write this out as
           \[
           \begin{aligned}
            0 = &\sum_{h}  S^{h'h}_{1} P(2h-4m+1, 2h-3)\\
                  \null  & - m \epsilon\tau^{n-2h'}  L(2h' +1)
                             P(-2h'+5, -2h' + 4m-3) \\
                    \null & +m\epsilon\tau^{n-2h'-4}
                    L(2h'-4m+1) 
                           P(-2h'+1, -2h' + 4m-7).
                                      \end{aligned}
                    \]
           The solution $S^{h'h}_{1}$ consisting of scalars in $R$ is
           unique,  because the terms
            $g(0)_{h}$ are a basis for the homogeneous polynomials of
            degree $m$ in $(\alpha,\delta)$.

            The last two of the three
            terms above have at least $m-2$ common linear factors
            $L(k)$, and
            have $m-1$ common factors in two edge cases. The $m-2$
            factors are the expression
            \[
                    Q(h') = P(-2h'+5, -2h'+4m-7). 
            \]
            The edge cases are $h'=0$ (which only occurs when $m$ is
            odd), and $h'=n-2$ (which occurs only when $m$ is even).
            In these two edge cases the $m-1$ common factors are
            respectively,
            \[
            \begin{aligned}
            Q_{+} &= L(1) Q(0) \\
                 &=   P(1, 4m-7)
            \end{aligned}
            \]
            and 
            \[
            \begin{aligned}
                    Q_{-} &=L(-1) Q(n-2) \\
                      &= P(-4m+7,-1).
            \end{aligned}
            \]
            We seek a solution $S_{1}^{h'h}$ to the above equations in
            the special form where, for each $h'$, the coefficients
            $S_{1}^{h'h}$ are non-zero only for those values of $h$
            for which $g(0)_{h}$ is divisible by $Q(h')$ (respectively
            $Q_{+}$ or $Q_{-}$ in the edge cases). Excluding the edge
            cases, there are three such values of $h$, namely
            \begin{equation}\label{eq:generic-positions-S1}
                    h \in \{\, n-h'-3,\;  n-h'-1,\;
                    n-h'+1\,\},
                    \quad (0<h'<n-2).
            \end{equation}
            In each of the edge cases, there are two such values of
            $h$:
            \begin{equation}\label{eq:edge-positions-S1}
\begin{aligned}
                    &h \in \{\, n-3,\;  n-1\,\},
                    &&\quad (h'=0) \\
                    &h \in \{\,   1,
                    3\,\},
                    &&\quad (h'=n-2) \\ 
\end{aligned}
            \end{equation}
            In the non-edge cases, the equations for the non-zero
            coefficients $S_{1}^{h'h}$ then take the general
            shape
            \begin{equation}\label{eq:general-shape-S1-eqn}
                    S_{1}^{h',\,n-h'-3} A+
                      S_{1}^{h',\,n-h'-1} B  +
                        S_{1}^{h',\,n-h'+1} C + D = 0,
            \end{equation}
            where $A$, $B$ and $C$ are the homogeneous quadratic polynomials in
            $(\alpha,\delta)$ given by
            \[
                        g(0)_{h} \big/ Q(h'), \qquad  h \in \{\, n-h'-3,\;  n-h'-1,\;
                    n-h'+1\,\}
            \]
            and $D$ is a quadratic polynomial
            \[
                   D = \bigl(S_{0}^{h',\,h'}g(1)_{h'} +
                   S_{0}^{h',\,h'+2}g(1)_{h'+2}\bigr)\big/Q(h').
            \]
            Explicitly,
            \[
             \begin{aligned}
                A &= L(-2h' -3)L(-2h' +1) \\
                B&= L(-2h'+1)L(-2h'+4m-3) \\
                C&= L(-2h' +4m-3) L(-2h'+4m+1) \\
        \end{aligned}
\]
        and
\[
                D= m\epsilon\tau^{n-2h'}\bigl(  -L(2h'+1)   L(-2h' + 4m -3)   +
                \tau^{-4}   L(2h'-4m+1)   L(-2h'+1) \bigr).
              \]
            The three
            polynomials $A$, $B$ and $C$ are independent, and we know
            there to be a unique solution, which we can now find by
            equating coefficients of $\alpha^{2}$, $\alpha\delta$ and
            $\delta^{2}$ in \eqref{eq:general-shape-S1-eqn}.
            The two edge cases are similar. Thus
            in the case $h'=0$, the equations for the two unknown
            coefficients of $S_{1}$ take the form
            \begin{equation}\label{eq:edge-1-shape-S1-eqn}
                    S_{1}^{0,\,n-3} X+
                      S_{1}^{0\,n-1} Y = Z,
            \end{equation}
            where $X$, $Y$ and $Z$ are homogeneous linear forms in
            $(\alpha,\delta)$, while in the case $h'=n-2$ we have
            similar
            equations
                \begin{equation}\label{eq:edge-2-shape-S1-eqn}
                    S_{1}^{n-2,\,1} X'+
                      S_{1}^{n-2,\,3} Y' = Z'.
            \end{equation}
       
            Solving the equations
            (\ref{eq:general-shape-S1-eqn}--\ref{eq:edge-2-shape-S1-eqn})
            for the coefficients $S_{1}^{h'h}$ leads to the following
            answer, valid for all $h'$,  whether or not we are in an
            edge case. We find:
            \begin{equation}\label{eq:S1-answer}
                  S^{h'h}_{1} = 
                        \begin{cases}
                            \epsilon\tau^{n-4-2h'}(-n+2+h'),
                            & h=n-h'-3,\\
                            \epsilon\tau^{n-4-2h'}(m-h'-1+(m-h')\tau^{4}),
                            & h=n-h'-1,\\
                           \epsilon\tau^{n-2h'}h',
                            & h=n-h'+1,\\
                             0, &\text{otherwise},
                        \end{cases} 
            \end{equation}
            for all $h', h$ in the range $0\le h \le n$ and $0\le h'
            \le n-2$ with the parity constraint $h=h'=m+1 \pmod 2$.
            So we have obtained the desired closed form for the
            generators of the ideal $\barJOne{n}$ for the instanton
            homology $I(Z_{n,-1})$:

            \begin{theorem}\label{thm:barJn-via-syz}
                  Let $S = S_{0} + S_{1}$ be an $m\times (m+1)$
                  with rows indexed by $h'$ and columns indexed by $h$
                  in the range $0\le h \le n$ and $0\le h'
            \le n-2$ with the parity constraint $h=h'=m+1 \pmod 2$.
             Let the entries of $S_{0}$ be given by
             \eqref{eq:S0-answer} and the entries of $S_{1}$ be given
             by \eqref{eq:S1-answer}, so that the entries of $S$
             belong to the ring $\bar\cA = \Q[\tau,\tau^{-1},\epsilon,\alpha,\delta]/ \langle
             \epsilon^{2}=1 \rangle$.  Then the normalized  generators
             $m!\,\bar\om^{m}_{h}$ of the ideal $\barJOne{n}$ are
             given by the $m\times m$ minors of $S$.    
            \end{theorem}
            
            \begin{remark}
The matrix the matrix $S$ has $m+1$ different $m\times m$
minors, and explicitly the generators of the ideal can be expressed as
\[
                    m!\,\bar \om^{m}_{h} = \pm \det S[h],
\]
where $S[h]$ is obtained from $S$ by deleting the column indexed by
$h$. (Recall again that the indexing of the columns is by only those
integers
$h$ with $h=m+1$ mod $2$.) The signs alternate as usual. Although there
are $m+1$ generators in this description, in fact only two generators suffice, as
the following proposition states.

\begin{proposition}\label{prop:two-minors}
    The ideal $\barJOne{n}$ is generated by the two elements
    $\bar\om^{m}_{m-1}$ and $\bar\om^{m}_{m+1}$, or equivalently by
    the two determinants
    \[
    \begin{gathered}
    \Gen{n}{1}=\det S[m-1]\\
            \Gen{n}{2}=\det S[m+1].
            \end{gathered}
    \]        
    \end{proposition}

    \begin{proof}
        It is sufficient to show that the matrix $S[m-1, m+1]$
        obtained by deleting both columns $h=m-1$ and $h=m+1$ has full
        rank $m-1$. To do this, let us examine the $(m-1)\times (m-1)$
        matrix $T$ obtained from $S[m-1,m+1]$ by deleting either the first or
        last row, according as $m$ is odd or even respectively. An
        inspection of the entries of $S$ reveals first that the entries of
        $T$ on the contra-diagonal are all units in $\bar\cA$: they
        are non-zero integers
        times powers of $\tau$. Furthermore, a reordering of the rows
        and columns makes $T$ triangular, with these same units on the
        diagonal. The determinant of $T$ is therefore non-zero, which
        shows that $S[m-1,m+1]$ indeed has full rank as desired.
    \end{proof}

As illustration, when $m=3$ (i.e.~$n=7$) the two generators
$\Gen{7}{1}$ and $\Gen{7}{2}$  are:
\begin{multline*}
(1/48) \biggl( 8 \alpha^3+36 \alpha^2 \delta +22 \alpha  \delta^2
-21 \delta^3 +24\epsilon\tau^3 \alpha^2 
-72 \epsilon \tau^3 \alpha  \delta + 30 \epsilon\tau^3 \delta^2\\
- (88 \tau^2 + 16  \tau^{-2}) \alpha  -
   (52 \tau^2 + 56 \tau^{-2})\delta  -24\epsilon
   \tau^5-96\epsilon \tau \biggr)
   \end{multline*}
and
\begin{multline*}
(1/48) \biggl(8 \alpha^3-12 \alpha^2 \delta -26 \alpha  \delta^2 +15 \delta^3
 +24\epsilon \tau^{-1} \alpha^2 
 +24\epsilon\tau^{-1} \alpha  \delta - 18 \epsilon\tau^{-1} \delta^2 \\
 -(40 \tau^2 +64 \tau^{-2})\alpha 
   +(68 \tau^2 -32 \tau^{-2}) \delta  -72 \epsilon\tau- 48 \epsilon\tau^{-3}\biggr).
\end{multline*}
            \end{remark}
           
           \subsection{Relating different values of $n$}

        Theorem~\ref{thm:barJn-via-syz} provides a complete
        description of the instanton homology of $Z_{n,-1}$ with local
        coefficients, but  we have not yet presented a full description for
        the case of $Z_{n}$. As preliminary material for this, we
        describe how the functoriality of instanton homology can be
        used to obtain relations between the ideal of relations in
        $Z_{n}$ for different values of $n$.
                   
        The fact that the ideal $\J{n}$ annihilates $I(Z_{n})$
        leads, via a standard approach, to the interpretation of the
        elements of $\J{n}$ as universal relations that hold for the
        maps defined by general bifold cobordisms.
       To spell this out,
        let $W$ be a homology-oriented
        bifold cobordism from $Z^{0}$ to $Z^{1}$, both of
        which are admissible. We have seen in
        section~\ref{subsec:operators} that $W$ gives rise to
        homomorphisms of $\cR$-modules
        \[
                I(W,a) : I(Z^{0})\to I(Z^{1})
        \]
        depending linearly on
        \[
                a \in \mathrm{Sym}_{*}\biggl( H_{2}(W; \Q) \oplus
                          H_{0}(\Sigma(W); O)\biggr) \otimes \cR,
        \]
        where $O$ is the orientation bundle of the singular set
        $\Sigma(W)$ with coefficients $\Q$. Further, given a distinguished
         2-dimensional class $e$ we can use marked connections with
         non-zero $w_{2}$ to define maps
         \[
              I(W,a)^{e} : I(Z^{0})\to I(Z^{1}).
          \]
         Using $\delta_{p}$ to denote the generator of the symmetric
         algebra corresponding the homology class of a
         point $p\in \Sigma(W)$ with local orientation, let us
         imitate the definition of $\cA_{n}$ and write
        \[
                    \cA(W) = \left(\mathrm{Sym}_{*}\Bigl( H_{2}(W; \Q) \oplus
                          H_{0}(\Sigma(W); O)\Bigr)
                            \otimes  \cR[\epsilon] \right) \bigm/ \bigl\langle \epsilon^{2}-1,
                 \delta_{p}^{2}-\delta_{q}^{2}\bigr\rangle_{p,q}
        \]
        where the indexing in the ideal runs through all pairs of
        points $p,q$ in $\Sigma(W)$.        We obtain a linear map
        \begin{equation}\label{eq:Psi}
             \Psi:   \cA(W) \to \mathrm{Hom} ( I(Z^{0}),  I(Z^{1})) 
        \end{equation}
        by
        \[
                a_{1} + \epsilon a_{2} \mapsto I(W, a_{1}) + I(W,
                a_{2})^{e}.
        \]
        This construction has been phrased so that, in the special
        case that $W$ is the product cobordism from $Z_{n}$ to
        itself and $e$ is the generator of $H_{2}$, the algebra
        $\cA(W)$ coincides with $\cA_{n}$ as defined above, and the
        map $\Psi$ is the action of the algebra $\cA_{n}$ on the
        module $I(Z_{n})$ via the instanton module structure.

        Continuing with the case of a general cobordism $W$, we
        suppose now that we have an embedded orbifold sphere $S\subset
        W$ meeting the singular set in $n$ orbifold points $\{p_{1},
        \dots, p_{n}\}$. 
        After choosing an orientation for $S$ we obtain local
        orientations for the singular set at the $n$ points of
        intersection, and hence we obtain elements $\delta_{p_{k}} \in
        \cA(W)$, where for the class $e$ in the definition of $\cA(W)$
        we take the fundamental class $[S]$. 
        Let us suppose
        that the normal bundle of $S$ is trivial so that the boundary
        of the tubular neighborhood of $S$ is a copy of $Z_{n}$.
        From the definitions, there is a natural map
        \[
                i_{*}:\cA_{n} \to \cA(W)
        \]
        arising from the inclusion; the map is defined so that
        $i_{*}(\delta_{k})=\delta_{p_{k}}$ for all $k$ and
        $i_{*}(\alpha)=[S] \in H_{2}(W)$.

        \begin{proposition}\label{prop:universalW}
            For an embedded orbifold sphere $S\subset W$ as above, the
            ideal $\J{n}$ lies in the kernel of the map $\Psi$ defined
            at \eqref{eq:Psi}. That is, for $a=a_{1} + \epsilon
            a_{2}\in \J{n}\subset \cA_{n}$, we have
            \[
                I(W, i_{*}(a_{1})) + I(W, i_{*}(a_{2}))^{e}=0.
            \]
            More generally, if $b$ is another class in $\cA(W)$ which
            an be expressed as a polynomial in cycles disjoint from
            $S$, then we have
            \[
                I(W, i_{*}(a_{1})b) + I(W, i_{*}(a_{2})b)^{e}=0.
            \]
        \end{proposition}

        \begin{proof}
            This is a standard argument based on the observation that
            we can factor the cobordism $W$ as a composite cobordism
            in which the first factor is the cobordism from $Z^{0}$ to
            $Z^{0} \amalg Z_{n}$. For the disjoint union, we can construct
            the instanton homology as a tensor product, and then we apply
            functoriality. See \cite{KM-Tait} and \cite{Street}, for
            example, for similar arguments.
            \end{proof}

        Our application of Proposition~\ref{prop:universalW} is
        equivalent to \cite[Corollary~2.6.8]{Street}. (A
        closely-related result appears in \cite{Munoz1}.) Suppose that
        \[
                n = n' + 2f,
        \]
        where $f \ge 0$. Consider an embedding of the orbifold sphere
        $S=S^{2}_{n}$ in the trivial cobordism $W=[0,1]\times
        Z_{n'}$, representing the generator in homology.
        This means that $S$ meets the singular locus $[0,1]
        \times K(Z_{n'})$ geometrically in $n'+2f$ points, while the
        algebraic intersection number is $n'$. There are therefore $2f$
        signed intersection points that cancel in pairs. Such a
        sphere $S\subset [0,1]\times Z_{n'}$ can be constructed as by
        taking the standard generating sphere $S'\subset Z_{n'}$ and
        introducing $2f$ extra intersection points by doing $f$
        ``finger moves'' to the sphere $S'$. We take these extra
        intersection points
        to be the orbifold points numbered $n'+1,\dots,n'+2f$ in
        $S \cong S^{2}_{n}$, and we suppose that they all lie on the component
        $[0,1]\times K^{n'}  \subset [0,1]\times K(Z_{n'})$. Among
        these $2f$ points, there are $f$ of them that have negative
        intersection number, and we can take it that these are the
        points numbered $n'+f+1, \dots, n'+2f$ in $S^{2}_{n}$. There is
        a corresponding map
        \[
                i_{*}^{n,n'} : \cA_{n} \to \cA_{n'}, \qquad (n=n'+2f),
        \]
         and our choice of numbering  means that it is given by
         \[
           i_{*}^{n,n'}(\alpha) = \alpha,
         \]
         and
         \[
           i_{*}^{n,n'}(\delta_{k}) = 
           \begin{cases}
            \hphantom{+}\delta_{k}, & 1 \le k\le n', \\
            \hphantom{+}\delta_{n'}, & n'+1\le k \le n'+f, \\
            -\delta_{n'}, & n'+f+1 \le k \le n'+2f.
           \end{cases}
                   \]
          Proposition~\ref{prop:universalW} now yields the following.

          \begin{corollary}[\protect{\cite[Corollary~2.6.8]{Street}}]
          \label{cor:recursive-1}
            When $n=n'+2f$ and $i_{*}^{n,n'} : \cA_{n} \to \cA_{n'}$ is
            defined as above, we have an inclusion of ideals,
            \[
                    i_{*}^{n,n'}\J{n}    \subset \J{n'}.
            \]
            \qed
          \end{corollary}

          With a little more work and an examination of the explicit
          formulae for the leading and subleading terms of the
          generators of $\J{n}$
          (Proposition~\ref{prop:explicit-formulae}), we can
          strengthen the above corollary as follows.

          \begin{proposition}
            In the situation of Corollary~\ref{cor:recursive-1} above,
            we have inclusions
            \[
               (\tau^{4}-1)^{f}  \J{n'} \subset  i_{*}^{n,n'}\J{n}
               \subset \J{n'}.
            \]
            In particular, the ideals $i_{*}^{n,n'}\J{n}$ and $\J{n'}$
            become equal after tensoring with the field of fractions of the
            ring $\cR=\Q[\tau,\tau^{-1}]$.
          \end{proposition}

          \begin{proof}
            It is sufficient to treat the case $f=1$, so $n'=n-2$. Let
            $\eta_{0}\subset{1,\dots, n-2}$, and let
            $\eta_{1},\eta_{2}\subset \{1, \dots, n\}$ be respectively
            the same as $\eta_{0}$ and $\eta_{0}\cup \{n-1,n-2\}$.
            From the explicit formulae, we see
            \[
                    i_{*}^{n,n-2}(w^{m}_{n,\eta_{1}}) =
                    i_{*}^{n,n-2}(w^{m}_{n,\eta_{2}}) ,
            \]
            because $i_{*}^{n,n-2}(B_{\eta_{1}})=
            i_{*}^{n,n-2}(B_{\eta_{2}})$. 
            Similarly
            \[
                                    i_{*}^{n,n-2}(w^{m-1}_{n,\eta'_{1}}) =
                    i_{*}^{n,n-2}(w^{m-1}_{n,\eta'_{2}}) .
            \]
            We therefore have (using the general shape of the
            subleading term),
            \[
            \begin{aligned}
            i_{*}^{n,n-2}(W^{m}_{\eta_{1}} -
                        W^{m}_{\eta_{2}}) &= \epsilon(
                        \tau^{n-2h}-\tau^{n-2h-4})
                        i_{*}^{n,n-2}(w^{m-1}_{n,\eta'_{1}}) +
                        \text{lower terms} \\
                        &= u 
                        (\tau^{4}-1)i_{*}^{n,n-2}(w^{m-1}_{n,\eta'_{1}}) +
                        \text{lower terms}.
                        \end{aligned}
            \]
            where $u$ is a unit in $\Q[\tau,\tau^{-1}]$.
            By the previous corollary, these belong to
            $\J{n-2}$. It is now enough to show that the
            elements $i_{*}^{n,n-2}(w^{m-1}_{n,\eta'_{1}})$ generate the ideal
            $j_{n-2}$ of relations in the ordinary cohomology of
            $\Rep(S^{2}_{n-2})$, because the statement about instanton
            homology will follow as before. From the formulae in
            Proposition~\ref{prop:explicit-formulae}, we see that
            this is the same as showing that the elements
            $w^{m-1}_{n-2,\eta'_{0}}$  generate the ideal $j_{n-2}$, which has
            already been established (as the case $n-2$) in
            Proposition~\ref{prop:little-w-generate}.
          \end{proof}

          The homomorphism $i_{*}^{n,n'}$ does not pass to a
          homomorphism between the quotient rings $\bar\cA$.
          But we can at least compose with the quotient
          map $\cA_{n'}\to\bar\cA$ to get the following immediate
          corollary. In the statement of the corollary, we note that
          the choices of sign in the definition of $i_{*}^{n,n'}$ are
          arbitrary and can be replaced by a more general phrasing.

          \begin{corollary}\label{cor:recursive}
            Let $\nu \in \{\pm 1\}^{n}$ be any choice of signs. Write
            $n' = \sum \nu_{i}$ and assume $n'\ge 1$.
            Consider the homomorphism
            $\bar\imath_{\nu} : \cA_{n} \to \bar\cA$ defined
            by $\bar\imath_{\nu}(\delta_{i})=\nu_{i}\delta$ for all
            $i$. Then we have an inclusions of ideals in
            $\bar\cA = \cR[\delta,\alpha,\epsilon]/\langle
            \epsilon^{2}-1\rangle$,
            \[
                  (\tau^{4}-1)^{(n-n')/2} \barJOne{n'}   \subset
                  \bar\imath_{\nu}(\J{n}) \subset  \barJOne{n'}.
            \]
            \qed
          \end{corollary}

           We refer to the relations between the ideals in
          Corollaries~\ref{cor:recursive-1} and \ref{cor:recursive} as
          ``finger-move relations'', because of the interpretation of
          the sphere $S$ as having been obtained from the standard
          sphere $S'\subset W$ by finger moves.
          
          \begin{remark}
            A second application of Proposition~\ref{prop:universalW}
            will be given in the proof of
            Proposition~\ref{prop:divisible} later in this paper.
          \end{remark}
          
                    \subsection{Decomposition of the instanton curve}

           We are now ready to harness our understanding of
           $I(Z_{n,-1})$ from Theorem~\ref{thm:barJn-via-syz}
           to obtain a description of $I(Z_{n})$.         
           Write \[V_{n} = \Spec \Q[\tau,\tau^{-1}, \alpha, \delta_{1},
           \dots, \delta_{n}, \epsilon].\] The set of complex-valued points
           $V_{n}(\C)$ is
           $\C^{\times}\times \C^{n+2}$, with $\tau$ a coordinate on
           the first factor. We can describe the $\cA_{n}$-module
           $I(Z_{n})$ geometrically as the coordinate ring of the
           closed subscheme
           \[
                    C_{n} \subset V_{n}
           \]
           defined by the vanishing of the elements of the ideal
           $\J{n}$ together with the additional relations that define
           the algebra $\cA_{n}$, namely the vanishing of
           $\delta_{i}^{2}-\delta_{j}^{2}$ and $\epsilon^{2}-1$.
           We can write $C_{n} = \Spec(I(Z_{n}))$, where $I(Z_{n})$ is
           considered as a quotient ring of the algebra $\cA_{n}$. To
           describe $I(Z_{n})$ as an $\cA_{n}$-module, we can
           therefore use geometrical language to describe the
           subscheme $C_{n}$. Note that the relation $\epsilon^{2}=1$
           means that $C_{n}$ is contained in the union of the two
           hyperplanes $\epsilon=1$ and $\epsilon=-1$, so we may write
           \[
                    C_{n} = C_{n}^{+} \cup C_{n}^{-}.
           \]

           In a similar way, let us write
           \[
                \bar V = \Spec \Q[\tau,\tau^{-1}, \alpha,
                \delta,\epsilon],
           \]
           so that the instanton homology group $I(Z_{n,-1}$ defines,
           (via its ideal of relations $\barJOne{n}$ and
           the relation $\epsilon^{2}=1$), a subscheme
           $D_{n}=\Spec(I(Z_{n,-1}))$, which is a closed
           subscheme of $\bar V$:
           \begin{equation}\label{eq:Dn}
                    D_{n} = D_{n}^{+} \cup D_{n}^{- }\subset \bar V.
           \end{equation}
            
           We can interpret  
           Corollary~\ref{cor:recursive} as describing a relation
           between the curves $C_{n}$ for $I(Z_{n})$ and $D_{n}$ for
           $I(Z_{n,-1})$. 
           First, given any choice of signs $\nu\in \{\pm 1\}^{n}$, define
           a morphism
           \[
                       \bar\imath^{*}_{\nu}: \bar V \to V_{n}
           \]
           by $\delta_{i}\mapsto \nu_{i}\delta$. Write
           \[
                    V_{n,\nu}\subset V_{n}
           \]
           for the image of $\imath^{*}_{\nu}$. This is the subvariety
           cut out by the linear relations $\nu_{i}
           \delta_{i}=\nu_{j}\delta_{j}$. Their union is the
           subvariety defined by $\delta_{i}^{2}=\delta_{j}^{2}$ for
           all $i$, $j$; so we have
           \[
                    C_{n} \subset \bigcup_{\nu} V_{n,\nu}.
           \]
          Given $\nu$ as above, write $n'=n'(\nu)=\sum \nu_{i}$ and suppose
           that this odd integer $n'$ is positive.
          We have an
            isomorphic copy of the affine scheme $D_{n'}$ as its image
            under the embedding $\imath^{*}_{\nu}$:
            \begin{equation}\label{eq:subscheme-nu}
                        \imath^{*}_{\nu}( D_{n'}) \subset V_{n,\nu}.
            \end{equation}
            just the image of $\bar C_{n'}$ under the embedding.
            
           \begin{proposition}\label{prop:union}
             The subscheme $C_{n}\subset V_{n}$ is the union of the
             subschemes \eqref{eq:subscheme-nu} as $\nu$ runs through
             all choices of sign $\{\pm 1\}^{n}$ with $n'(\nu)>0$:
             \begin{equation}\label{eq:C-is-union}
                        C_{n} = \bigcup_{\substack{ \nu \\ n'=n'(\nu)>0}}
                        \bar\imath^{*}_{\nu}(D_{n'}).
             \end{equation}
             \end{proposition}

             The curves $D_{n'}$ are completely known via their
             defining equations from Theorem~\ref{thm:barJn-via-syz},
             so the proposition above is a complete characterization
             of the curve $C_{n}$ for $I(Z_{n})$.
             In the language of the defining ideals, this proposition
             is a converse to Corollary~\ref{cor:recursive}. In other
             words, we have the following:

             \begin{corollary}\label{cor:intersection}
                In the notation of Corollary~\ref{cor:recursive}, the
                defining ideal $\J{n}$ for $I(Z_{n})$ can be
                characterized as:
                \[
                                \J{n} = \{ \, w\in \cA_{n} \mid
                                \bar\imath_{\nu}(w) \in
                                \barJOne{n'(\nu)}, \forall \nu \,\}.
                \]
                Thus $I(Z_{n})$ is
                determined as an $\A_{n}$-module by the finger-move
                constraints, once $I(Z_{n',-1})$ is known for all odd
                $n'\le n$.
             \end{corollary}

           \begin{proof}[Proof of Proposition~\ref{prop:union}]
           Let us write $C'$ for the union on the right hand side of
           \eqref{eq:C-is-union}.
           The inclusion of
           ideals $\bar\imath_{\nu}(\J{n}) \subset  \barJOne{n'}$ in
           Corollary~\ref{cor:recursive} says that the curve $C_{n}$
           contains $C'$

           The coordinate ring of the scheme on the left hand side of
           \eqref{eq:C-is-union} is $I(Z_{n})$, and if we temporarily
           write $I'$ for the coordinate ring of the affine
           scheme $C'$, then the inclusion of schemes means that
           we have a surjection of rings,
           \[
                            I(Z_{n}) \to I'.
           \]
           We know that $I(Z_{n})$ is a free $\cR$-module of finite
           rank, where $\cR=\Q[\tau,\tau^{-1}]$.  So to prove that the
           rings are isomorphic, and to complete the proof of the
           proposition, it will suffice to prove that these two
           $\cR$-modules have the same rank, or in geometrical
           language,
           \[
                        \deg C_{n} = \deg C'
           \]
           where $\deg$ denotes the degree of the projection to the
           $\tau$ coordinate. (The inclusion one way means that we
           already have
           $\deg C_{n} \ge \deg C'$.)

           To prove this last equality we note that
           \begin{equation}\label{eq:inequality-intersection-nu}
                    \deg C_{n} \le \sum_{\substack{\nu\\ n'(\nu)>0} }\deg ( C_{n} \cap
                    V_{n,\nu})
           \end{equation}
           with equality if and only if the schemes $C_{n}\cap
           V_{n,\nu}$ for different $\nu$ have no common component of
           positive degree.  The two-way inclusions of
           Corollary~\ref{cor:recursive} tell us that $C_{n}\cap
           V_{n,\nu}$ and $i^{*}_{\nu}(\bar C_{n'})$ coincide over
           locus where $\tau^{4}-1$ is non-zero. In particular,
           \[
          \deg \bigl(C_{n}\cap
           V_{n,\nu}\bigr)= \deg \bigl(i^{*}_{\nu}(D_{n'})\bigr)
           \]
           and if the schemes on the left have no common component of
           positive degree for different $\nu$, then the same is true
           of the schemes on the right. From
           \eqref{eq:inequality-intersection-nu} we therefore obtain
            \begin{equation}\label{eq:inequality-bar-nu}
                    \deg C_{n} \le \sum_{\substack{\nu\\ n'(\nu)>0} }
                    \deg D_{n'}
           \end{equation}
           with equality if and only if the schemes on the right hand
           side of \eqref{eq:C-is-union} have no common component of
           non-zero degree.

                               In terms of instanton homology, the
                               inequality
           \eqref{eq:inequality-intersection-nu} can be restated as
                   \begin{equation}\label{eq:inequality-of-ranks-I}
                        \rank_{\cR} I(Z_{n})
                       \le \sum_{\substack{\nu \\ n'(\nu)>0}}
                        \rank_{\cR} I(Z_{n',-1}).
           \end{equation}
           On the other hand we can verify directly that we have
           equality here:
           \begin{equation}\label{eq:rank-equality}
                        \rank_{\cR} I(Z_{n})
                        = \sum_{\substack{\nu \\ n'(\nu)>0}}
                        \rank_{\cR} I(Z_{n',-1}).
           \end{equation}
           Indeed, the right hand side can be calculated by
           Corollary~\ref{cor:rank-Znm1}, and is
           \[
           \sum_{f=0}^{(n-1)/2}  \binomial{n}{f}
             \bigl((n-2f)^{2}-1\bigr)/4.
           \]
           The left-hand side of \eqref{eq:rank-equality}
           is twice the rank of the ordinary cohomology of
           the representation variety  $\Rep(S^{2}_{n})$ calculated by
           Boden
           \cite{Boden}, and can be expressed as
           \[
           \begin{aligned}
           \rank_{\cR} I(Z_{n}) &= 2^{n-3}(n-1) \\
                                &= (F''(1)-F(1))/8
           \end{aligned}
           \]
           where
           $F(t)=(t+t^{-1})^{n}$. Equality with the right-hand side of
           \eqref{eq:rank-equality} can be seen easily from the
           binomial expansion of $F(t)$.

           It follows that the parts making up the union $C'$ on the
           right-hand side of \eqref{eq:C-is-union} have no common
           components of positive degree, and we therefore have
           \[
           \begin{aligned}
           \deg C' &= \sum_{\nu} \deg D_{n'(\nu)} \\
                        &= \deg C_{n}
                        \end{aligned}
            \]
           as required.
           \end{proof}
           
           \begin{remark}
           In the course of the proof, we have seen that $C_{n}$ has
           pure dimension $1$, and we refer to it as the instanton
           curve for $Z_{n}$. Although it has no embedded points, we
           have not shown that the curve $C_{n}$ is reduced: it may
           perhaps
           have components with multiplicity larger than $1$, but the
           authors have not seen this arise in calculations.
           \end{remark}

          \subsection{Equations for the curve $C_{n}$}

           We now have a geometric description of
           $I(Z_{n})$ as a module, namely as the coordinate ring of an
           affine curve $C_{n}$. The curve $C_{n}$ is a union of
           curves each of which is isomorphic to some $D_{n'}$.
           However, although we have an explicit description of the
           defining relations for the $D_{n'}$, the resulting
           description of $C_{n}$ does not immediately provide
           explicit generators for the corresponding ideal $\cJ_{n}
           \subset \cA_{n}$. Instead, it describes the ideal
           $\cJ_{n}$ as an intersection of known ideals (expressed
           essentially in Corollary~\ref{cor:intersection}).

          To practically compute the intersection of the ideals in
          this particular context, we can leverage what we know about
          $\J{n}$. 
          From Propositions~\ref{prop:quantum-Mumford} and
          \ref{prop:quantum-Mumford-subleading}, we know the ideal
          $\J{n}$ is generated by elements $\om^{m}_{\eta}$ which can be
          written in the form
          \begin{equation}\label{eq:omega-expansion}
               \om_{\eta}^{m} = w(0) + \epsilon w(1) + w(2) + \epsilon w(3) +
               \cdots
          \end{equation}
          where $w(i)$ is a homogeneous polynomial of degree $m-i$ in
          $(\alpha,\delta_{1}, \dots, \delta_{n})$, and furthermore
          \[
               \begin{aligned}
                w(0) &= w^{m}_{n,\eta} \\
                w(1) &= w^{m-1}_{n,n-\eta},
               \end{aligned}
          \]
          Furthermore, the element $\om_{\eta}^{m}$ is the unique
          element of the ideal having leading term $w(0)$. The lower
          terms in $\om_{\eta}^{m}$ are therefore uniquely
          characterized by the linear constraints of
          Corollary~\ref{cor:intersection}, namely that
          $bar\imath_{\nu}(\om_{\eta}^{m})$ belongs to the known ideal
                                $\barJOne{n'(\nu)}$, for all $\nu$.
                                Solving this large linear system
                                provides the generators.
                                
           There is an alternative way to package the calculation of
           $\om_{\eta}^{m}$, which does not explicitly pass through a
           determination of the ideals $\barJOne{n}$, albeit the same
           ingredients are used. To set this up, 
          he terms in \eqref{eq:omega-expansion}
          which are as yet \emph{unknown} are the terms which belong
          to a lower part of the increasing filtration of $\cA_{n}$:
          we write, 
          \[
               L_{\eta}^{m} = w(2) + \epsilon w(3) +
               \cdots ,
          \]
          so that
          \begin{equation}\label{eq:omega-to-B}
          \begin{gathered}
          \om_{h}^{m} = w(0) + \epsilon w(1) + 
                    L_{\eta}^{m} \\
                    L_{\eta}^{m} \in \cA_{n}^{(m-2)}.
          \end{gathered}
                    \end{equation}
                    
          There is some symmetry that can be usefully exploited. The
          braid group $B_{\pi}$ for the $n$-element subset
          $\pi\subset S^{2}$ acts on $I(Z_{n})$ because of its
          interpretation as a mapping class group. This action factors
          through the symmetric group $S_{\pi}$, as one can see from
          the description of $I(Z_{n})$ as a cyclic module for the
          algebra $\cA_{n}$. Indeed, given a permutation
          $\sigma\in S_{\pi}$, we obtain an automorphism $\sigma_{*}:
          \cA_{n}\to \cA_{n}$ permuting the generators $\delta_{p}$
          and preserving the ideal $\J{n}\subset \cA_{n}$, so
          establishing the automorphism $\sigma_{*}: I(Z_{n})\to
          I(Z_{n})$. From this, we can see that
          \[
                    \sigma_{*}(W^{m}_{\eta}) = W^{m}_{\sigma(\eta)}.
          \]
          In particular, the element $W^{m}_{\eta}\in \cA_{n}$ is
          invariant under the action of group of permutations
          $S_{\eta}\times S_{\eta'} \subset S_{\pi}$.

          The lower terms $L_{\eta}^{m}$ therefore have the same
          symmetry. Furthermore, it will be enough if we determine
          $L_{\eta}^{m}$ for just one subset $\eta\subset \pi$ of each
          cardinality $h$ satisfying the parity condition
          \eqref{eq:h-parity}. Note also that the expression $L_{\eta}^{m}$ is empty
          unless $m$ is at least $2$ (i.e. $n$ is at least 5). 
                    
          The proposed recursive procedure for identifying the lower
          terms $L_{\eta}^{m}$ is to again use Corollary~\ref{cor:recursive-1},
          which gives us the finger-move relation
          \begin{equation}\label{eq:finger-constraints}
               i_{*}^{n,n-2}(\om_{\eta}^{m}) \in \J{n-2}   
          \end{equation}
           We would like to see that, f the ideal $\J{n-2}$ is
           already known, then the constraint
           \eqref{eq:finger-constraints} will be sufficient to
           determine the lower terms. In line with the remarks above,
           since either $\eta$ or $\eta'$ can be assumed to have at
           least $m+1$ elements (i.e. more than half), we will assume that the indices
           $\{m, m+1, \dots, n\}$ all belong either to $\eta$ or to
           $\eta'$. In particular this means that $\om^{m}_{\eta}$ and
           its lower terms $L^{m}_{\eta}$ are invariant under the
           symmetric group $S_{m+1}$ acting by permutation of the
           variables $\{\delta_{m}, \delta_{m+1},\dots, \delta_{n}\}$.
           (These indices include the three indices $\{n-2,n-1,n\}$
           which are involved in the definition of the finger move
            $i^{n,n-2}$.)

           \begin{lemma}\label{lem:unique-B}
             Write $n=2m+1$ and
             let $L\in \cA_{n}^{(m-2)}$ be an element that is
             symmetric in the variables
             $\delta_{m+1},\dots,\delta_{n-1},\delta_{n}$ (i.e.~more
             than half of the variables).
             Suppose $L$ satisfies
             \begin{equation}\label{eq:B-finger}
             i_{*}^{n,n-2}(L)  \in \J{n-2}.
             \end{equation}
             Then $L=0$. 
           \end{lemma}

           \begin{proof}
            Let $\sigma_{k}$ be the $k$'th symmetric polynomial in
            $\delta_{m+1},\dots, \delta_{n}$, and let $\sigma'_{k}$ be
            the symmetric polynomial in $\delta_{m+1},\dots,
            \delta_{n-2}$, regarded as elements of $\cA_{n}$ and
            $\cA_{n-2}$ respectively. From
            Proposition~\ref{prop:below-middle}, we now that
            $\J{n-2}\cap \cA_{n-2}^{m-2}=0$, so the hypothesis
            $i_{*}^{n,n-2}(L)  \in \J{n-2}$ actually means that
            $i_{*}^{n,n-2}(L)$ is zero. We compute what
            $i_{*}^{n,n-2}$ does to $\sigma_{k}$, and we find
            \[
                    i_{*}^{n,n-2}(\sigma_{k}) =
                    \begin{cases}
                        \sigma'_{k}, & k=0,1 \\
                        \sigma'_{k} + \beta \sigma'_{k-2},
                                    & 2\le k \le m-1 \\
                        \beta\sigma'_{k-2},
                                    & k=m, m+1,
                    \end{cases}                    
            \]
            where $\beta=-\delta_{p}^{2}$ (independent of $p$).
            Because $L$ has degree at most $m-2$, we can write it as
            \[
                       L= \sum_{k=0}^{m-2}P_{k} \sigma_{k}
            \]
            where each $P_{k}$ is an expression in $\cA_{m}$,
            i.e.~involving only $\delta_{1},\dots,\delta_{m}$. Thus
            \[
                    i_{*}^{n,n-2}(L) = \sum_{k=0}^{m-2}(P_{k} + \beta
                    P_{k+2})\sigma'(k)
            \]
            where we set $P_{j}=0$ for $j>m-2$. The injectivity of
            $i_{*}^{n,n-2}$ is now clear from the upper triangular
            nature of this linear transformation, because the
            symmetric functions $\sigma'(k)$ are non-zero in this
            range.
           \end{proof}

           The lemma tells us that the finger move constraint can be
           used to determine the lower terms $L^{m}_{\eta}$ uniquely.
           So we obtain a procedure which determines the ideals
           $\J{n}$ recursively for all odd $n$, as follows.
           \begin{enumerate}
            \item In the base case $n=1$, the ideal $\J{1}$ is
            $\langle 1\rangle$.
            \item For general $n \ge 3$ (and $n$ odd as always), assume that the ideal
            $\J{n'}$ is already known for $n'<n$.
            \item Write $m=(n-1)/2$.  According to
            Propositions~\ref{prop:quantum-Mumford} and
            \ref{prop:quantum-Mumford-subleading}, for each $\eta$
            satisfying the parity condition \eqref{eq:h-parity-plus},
            there exists an element $\om_{\eta}^{m} \in
            \J{n}$ which can be written in the form
            \eqref{eq:omega-to-B}:
\[
\begin{aligned}
\om_{h}^{m} &= w(0) + \epsilon w(1) + w(2) + \epsilon w(3) + \cdots \\
            &= w(0) + \epsilon w(1) + 
                    L_{\eta}^{m} 
                   , \qquad
                    L_{\eta}^{m} \in \cA_{n}^{(m-2)}.
                                        \end{aligned}
\]
            The first
            terms $w(0)+ \epsilon w(1)$ are known because $w(0)$ is
            the Mumford relation and
            Proposition~\ref{prop:quantum-Mumford-subleading} provides the term
            $w(1)$.
            \item According to Lemma~\ref{lem:unique-B}, the unknown
            terms $L_{\eta}^{m}$ in $\om^{m}_{\eta}$ are uniquely
            determined by the finger-move relations
            \eqref{eq:finger-constraints}, which impose linear
            conditions on the coefficients of $L_{\eta}^{m}$. Solving
            these linear equations determines $L_{\eta}^{m}$ and hence
            determines $\om_{\eta}^{m} \in \cA_{n}$.
            \item As $\eta$ runs through the subsets satisfying
            \eqref{eq:h-parity-plus}, the elements $\om_{\eta}^{m}$  generate the ideal
            $\J{n}\subset \cA_{n}$ according to
            Proposition~\ref{prop:quantum-Mumford}. So we have a
            known set of generators for $\J{n}$. This determines
            $\J{n}$ and completes the inductive step.
           \end{enumerate}

\section{Further remarks}
\label{sec:further}

\subsection{Singularities of the instanton curve}

When the local coefficient system $\Gamma$ is replaced by constant
coefficients $\Q$, we obtain a description of the instanton homology
$I(Z_{n};\Q)$ which was earlier completely determined by Street
\cite{Street}. Those results therefore provide a description of the
scheme-theoretic intersection of the curve $C_{n}$ with the hyperplane
$\tau=1$. It is shown in \cite{Street} that the simultaneous
eigenvalues of the pair of operators $(\alpha,\delta)$ on
$I(Z_{n};\Q)$ are of the form $(\lambda,\delta)$, where $\lambda$ runs
through the odd integers in the range $|\lambda| < n$. The
multiplicities of the eigenspaces is also computed.

We can apply these results to learn that the curve $D_{n}$
corresponding to $I(Z_{n,-1}; \Gamma)$ intersects the plane $\tau=1$
in the points
\[
         x_{\lambda} :   (\tau,\alpha,\delta,\epsilon) = (1, \lambda, 0, \pm 1)
\]
where $\lambda$ runs through the same odd integers, and the sign of
$\epsilon$ is $(-1)^{(\lambda+1)/2}$. We also learn that the
intersection multiplicity at $x_{\lambda}$ is
$\mu_{\lambda}=(n-|\lambda|)/2$.

Knowing the intersection multiplicity puts an upper bound on the order
of a possible singular point of the curve at $x_{\lambda}$. In
particular, it means that $D_{n}$ is smooth at the points
$x_{\lambda}$ for the two extreme values of $\lambda$, namely
$\lambda=\pm (n-2)$, because the intersection multiplicity is $1$ at
those points.

A
little experimentation suggests that equality holds at all the points
$x_{\lambda}$ where $D_{n}$ meets $\tau=1$: that is,
\begin{equation}
\begin{aligned}
\ord( D_{n} ,x_{\lambda}) &= \mu_{\lambda} \\
                                  &= (n - |\lambda|)/2.  
                                  \end{aligned}
\end{equation}
With the understanding that these results have been verified only
experimentally for modest values of $n$, one can describe the
singularity of $D_{n}$ at $x_{\lambda}$ in greater detail. First of
all, we have seen that the ideal $\barJOne{n}$ which defines $D_{n}$
has just two generators $\Gen{n}{1}$, $\Gen{n}{2}$ (Proposition~\ref{prop:two-minors}),
and it follows that the singularity of $D_{n}$ at $x_{\lambda}$ is a
local complete intersection. Indeed, each of $D_{n}^{+}$ and
$D_{n}^{-}$ is cut out as a global complete intersection inside the
variety defined by $\epsilon=\pm 1$ and $\tau\ne 0$. Experiment also
indicates that the surfaces defined by the vanishing of $\Gen{n}{1}$ and
$\Gen{n}{2}$ are both smooth at $x_{\lambda}$. Indeed, the
$\alpha$-derivative of both is non-zero. By the implicit function
theorem, the zero-sets of $\Gen{n}{1}$ and $\Gen{n}{2}$ are therefore described
in a local analytic neighborhood of $x_{\lambda}$ by
\[
\begin{aligned}
    \alpha &= \lambda + f_{n,\lambda,1}(\delta,\tau) \\
    \alpha &= \lambda + f_{n,\lambda,2}(\delta,\tau) 
    \end{aligned}
\]
for two analytic functions $f_{n,\lambda,1}$ and $f_{n,\lambda,2}$. At the singular
points -- that is, when $|\lambda|<n-2$ -- the derivatives of both
$f_{n,\lambda,1}$ and $f_{n,\lambda,2}$ vanish at
$(\delta,\tau)=(0,1)$. The singular germ $(D_{n},x_{\lambda})$ is
therefore analytically isomorphic to the germ of the analytic plane
singularity
\[
\begin{gathered}
g_{n,\lambda}(\delta,\tau)=0 \\
g_{n,\lambda}= f_{n,\lambda,1}-f_{n,\lambda,2}
\end{gathered}
\]
at $(\delta,\tau)=(0,1)$.

In computations up to $n=31$, the function
$g_{n,\lambda}$ vanishes to order $\mu_{\lambda}$ at $(0,1)$,
verifying that $\mu_{\lambda}$ is the indeed the order of the
singular point. Furthermore we find
\[
        g_{n,\lambda}(\delta,\tau) = \mathrm{const.} \bigl( \delta \pm
        2(\tau-1)\bigr)^{\mu_{\lambda}} +
        O(\delta, \tau-1)^{\mu_{\lambda}+1},
\]
where the sign depends on $\epsilon$ and $\lambda$.
This means that the tangent cone to the singular point is the line
$\delta\pm 2(\tau-1)=0$, with multiplicity $\mu_{\lambda}$.

The highest-order
singular points on the curve are the points $x_{\lambda}$ with $\lambda=\pm 1$,
where the order of the singularity is $m=(n-1)/2$. At these points,
the analytic form of the singularity is $x^{m}=y^{m+1}$ where
$x=\delta \pm 2(\tau-1)$. In particular the singularity is unibranch.
The authors have not determined (even experimentally) whether the
singularity is unibranch at other singular points. Note, however, that
the entire curves $D_{n}^{\pm}$ are reducible when $n$ is composite
(as discussed below) and it follows that the singularities are not
unibranch when $\lambda$ and $n$ have a common factor.

One further experimental observation is that the local form of the
surface $\Gen{n}{i}=0$, given by $\alpha = \lambda   +  f_{n,\lambda,i}(\delta,\tau)$
at $x_{\lambda}$,
appears to approach a smooth limit as $n$ increases with $\lambda$
fixed. Indeed, after scaling by $\lambda$, we find that the limit
is independent of $\lambda$ also. That is, there is a convergent power series
$F(\delta,\tau)$ independent of $n$, $\lambda$ and $i=1,2$, such that
\[
   \lambda + f_{n,\lambda,i} (\delta,\tau) \to \lambda F(\delta,\tau).
\]
The difference vanishes at
$(0,1)$ to order $(\delta,\tau-1)^{O(n)}$. Up to terms of degree $5$,
the series $F$ is
\begin{multline*}
F(\delta, 1+\sigma) = 1-\frac{\delta ^2}{16}+\frac{31 \delta \sigma }{4}+\frac{\sigma
^2}{4}-\frac{31 \delta \sigma ^2}{8}-\frac{\sigma ^3}{4}-\frac{5
\delta ^4}{1024}+\frac{31 \delta ^3 \sigma }{128}+\frac{5 \delta ^2
\sigma ^2}{128}+\frac{31 \delta
\sigma ^3}{32}\\+\frac{15 \sigma ^4}{64}-\frac{31 \delta ^3 \sigma ^2}{256}-\frac{5 \delta ^2 \sigma
^3}{128}+\frac{31 \delta \sigma ^4}{64}-\frac{7 \sigma ^5}{32} +
\cdots.
   \end{multline*}

\subsection{Reducibility when $n$ is composite}

The curves $D^{+}_{n}$ and $D^{-}_{n}$ arising as $\Spec(I(Z_{n,-1}))$
are irreducible when $n$ is prime in all cases that the authors have
calculated. It seems to be an
interesting conjecture whether this holds in general. For composite
$n$, however, the curves $D^{+}_{n}$ and $D^{-}_{n}$ are reducible, as
the following result implies.

\begin{proposition}\label{prop:divisible}
    If $n'$ divides the odd integer $n$, then the curves $D^{+}_{n}$
    and $D^{-}_{n}$
    contain $\psi(D^{+}_{n'})$ and
    $\psi(D^{-}_{n'})$ respectively, where $\psi$ is the map on the
    ambient space $\bar V$ given by
    $\psi(\tau,\tau^{-1},\delta,\alpha,\epsilon) =
    (\tau,\tau^{-1},\delta, (n/n')\alpha, \epsilon)$.
    \end{proposition}

 \begin{proof}
        This is an application of the general principal described by
        Proposition~\ref{prop:universalW}. In the context of that
        proposition,
        take $W$ to be the product cobordism $[0,1]\times
        Z_{n'}$. Write $l=n/n'$. We can embed a sphere
        $S\hookrightarrow W$ representing $l$ times the generator of
        $H_{2}(W)$ and meeting the singular set in $ln'$ points, all
        with the same orientation.
        The relevant map $\Psi$ in Proposition~\ref{prop:universalW}
        is then  the  homomorphism of algebras
        \[
                \Psi_{l} : \cA_{n} \to \cA_{n'}
        \]
        which is given by (with our standardly named generators, and suitably
        numbering the intersection points),
        by
        \[
                \begin{gathered}
                    \Psi_{l}(\alpha) = l \alpha \\
                    \Psi_{l}(\delta_{k})  = \delta_{(k\bmod n')}.
                \end{gathered}
        \]
       The conclusion of Proposition~\ref{prop:universalW} is that we
       have an  inclusion of ideals
            $\Psi_{l}(\J{n}) \subset \J{n'}$.

        Passing to the quotient rings $\bar{\cA}$ in which all the
        $\delta_{k}$ are equal, and using the fact that $\barJOne{n}$
        is the image of $\J{n}$ in the quotient ring
        (Proposition~\ref{prop:omegas-generate}),
        we obtain an inclusion of ideals $\psi_{l}(\barJOne{n})
            \subset \barJOne{n'}$ when $n=ln'$, where $\psi_{l}$ is
            the with $\psi_{l}(\alpha)=l \alpha$ and
            $\psi_{l}(\delta)=\delta$.
            Proposition~\ref{prop:divisible} is just a restatement of
            this inclusion of ideals, in the geometrical language of
            the subschemes that they define.
 \end{proof}

   \subsection{Interpretation as the quantum cohomology ring}

   For
   every odd $n$, the representation variety $M=\Rep(S^{2}_{n})$ is
   naturally a smooth symplectic manifold, by a standard construction
   \cite{Goldman}. If $n$ points in $\CP^{1}$ are chosen, then $M$
   becomes also a smooth complex-algebraic variety of dimension $n-3$,
   as a consequence of its interpretation as a moduli space of stable
   parabolic bundles. With the symplectic form, it is a K\"ahler
   manifold, and the cohomology class of the K\"ahler form is a
   negative multiple of the canonical class. The latter assertion is
   the statement of ``monotonicity'' for the symplectic structure. It
   can be deduced as a particularly simple case from \cite{KM-yaft},
   for example, or it can be deduced from the fact that there is only
   one class in $H^{2}$ which is invariant under the ``flip''
   symmetries \cite{Street}. This is therefore a Fano variety. (A
   concrete description is discussed in \cite{Casagrande}.)

       The quantum cohomology ring of such a Fano variety is defined
       using a deformation of the usual triple intersection product.
       Given cycles $A$, $B$, $C$, the quantum intersection product is
       a scalar which is a weighted count of isolated pseudo-holomorphic curves $u :
       \CP^{1} \to M$, with the constraint that $u$
       maps three marked points to $A$, $B$ and $C$. For our purposes,
       the weight will be of the form $\tau^{[u]\cdot T}$ for a
       suitable 2-dimensional cohomology class $T = 2\sum
       \delta_{i}$. This leads to a
       quantum cohomology ring $\mathit{QH}(M)$ which is a module over
       the ring of Laurent polynomials $\cR$. In the spirit of results from \cite{Munoz2}
       and \cite{Dostoglou-Salamon}, one should expect that the
       $\epsilon=1$ component of $I(Z_{n})$ is isomorphic to
       $\mathit{QH}(M)$ as an algebra.

   \subsection{General local coefficients}

   As an alternative to the local coefficient system $\Gamma$ for
   $I(Z_{n})$, there is a larger local coefficient system $\Gamma_{n}$
   that can be used. Rather than being a system of rank-1 modules over
   $\cR=\Q[\tau^{-1},\tau]$, the ground ring for $\Gamma_{n}$ is the
   ring of finite Laurent series in $n$ distinct variables
   $\tau_{1},\dots,\tau_{n}$ attached to the $n$ components of the
   singular set of $Z_{n}$:
    \[
             \cR_{n} = \Q[\tau_{1}, \tau_{1}^{-1}, \dots ,\tau_{n},
             \tau_{n}^{-1}]
    \]
    The instanton homology $I(Z_{n}; \Gamma_{n})$ is then a module
   over the ring
   \[
        \cR_{n}[\delta_{1}, \dots, \delta_{n}, \alpha,\epsilon ].
   \]
   It is no longer true that $\delta_{i}^{2}=\delta_{j}^{2}$; instead we
   have
   \[
   \delta_{i}^{2}-\tau_{i}^{2}-\tau_{i}^{-2}
        = \delta_{j}^{2}-\tau_{j}^{2}-\tau_{j}^{-2}
   \]
   for all $i,j$.

   It should be possible to compute
   $I(Z_{n}; \Gamma_{n})$ by adapting the ideas of this paper.
   As the
   simplest example, our two generators for the relations in
   $I(Z_{3,-1})$, where all $\delta_{i}$ and all $\tau_{i}$ are equal,
   were
   \[
   \begin{aligned}
   & \alpha + (3/2)\delta + \epsilon \tau^{3} \\
   & \alpha - (1/2)\delta + \epsilon \tau^{-1}
   \end{aligned}
   \]
   For $I(Z_{3}; \Gamma_{3})$ the corresponding  relations are
   \[
   \begin{aligned}
   & \alpha + (1/2)(\delta_{1} + \delta_{2} + \delta_{3}) +
   \epsilon \tau_{1}\tau_{2}\tau_{3} \\
   \end{aligned}
   \]
   and
     \[
   \begin{aligned}
   & \alpha + (1/2)(\delta_{1} - \delta_{2} - \delta_{3}) +
   \epsilon \tau_{1}\tau^{-1}_{2}\tau^{-1}_{3}, \\
   \end{aligned}
   \]
 together with cyclic rotations of the second one. The instanton
 homology $I(Z_{3}; \Gamma_{3})$ is a free $\cR_{3}$-module of rank
 $2$.

 There is an
 additional symmetry present when using $\Gamma_{n}$ which comes from
 the flip relation. So the ideal of generators is invariant under the
 symmetry which changes the sign of $\delta_{i}$ and $\delta_{j}$ for
 any two distinct indices while changing $\tau_{i}$ and $\tau_{j}$ to
 $\tau_{i}^{-1}$ and $\tau_{j}^{-1}$. In the example of
 $I(Z_{3};\Gamma_{3})$ there are four generators corresponding to the
 four subsets $\eta\subset \{1,2,3\}$ of even parity, and the corresponding
 relations are all obtained from the first one (corresponding to
 $\eta=\varnothing$) by applying flips. For larger $n$, the leading
 and sub-leading terms follow the same pattern. So the adaptation of
 Proposition~\ref{prop:quantum-Mumford-subleading} to the case of
 $\Gamma_{n}$ has the same leading term while the factor of
 $\tau^{n-2h}$ in front of the subleading term is replaced by
 \[
            \prod_{i\notin\eta} \tau_{i} \prod_{i\in\eta}
            \tau_{i}^{-1}. 
 \]

   \subsection{Instanton homology for torus knots}

   As mentioned in the introduction, a motivation for this paper comes
   from wishing to calculate variants of framed instanton homology for
   torus knots. In \cite{KM-InstConc}, concordance invariants of knots
   were defined using a version of framed instanton homology
   $\Isharp$. In that paper, for a knot $K\subset Y$, the framed
   instanton homology is defined using the connected sum $(Y,K)\#
   (S^{3}, \Theta)$, where $\Theta$ is a theta-graph in $S^{3}$. A
   local coefficient system is used in \cite{KM-InstConc}, where the
   ground ring is the Laurent polynomials in three variable
   $\tau_{i}$ corresponding to the three edges of $\Theta$. Because of
   the phenomenon of bubbling in codimension $2$ which arises from the
   vertices of $\Theta$, it was necessary in \cite{KM-InstConc} to
   use  ring of characteristic $2$.

   It is possible instead to work in characteristic zero by abandoning
   the pair $(S^{3},\Theta)$ and using the pair $Z_{3}$ instead (as
   described just above). The local coefficient system comes from
   $\Gamma_{3}$. Because $I(Z_{3};\Gamma_{3})$ has rank $2$, one
   should take just the $+1$ eigenspace of $\epsilon$ to obtain a
   rank-1 module. Thus one can define $\Isharp(Z;\Gamma_{3})$ for
   general bifolds $Z$ as being $I(Z\# Z_{3}; \Gamma_{3})_{+}$. The
   connected sum is of the 3-manifolds, not a connected sum of
   pairs. But a connected sum of pairs can be used instead to define a
   reduced version $\Inat(Z;\Gamma_{3})$.
   
     A variant of the connected sum theorem
        from \cite{Daemi-Scaduto} allows one to pass to 
        $\Inat(Z_{n,-1};\Gamma_{3})$ starting from the calculation of $I(Z_{n,-1})$
        in this paper.        
        Using the surgery exact
        triangle for instanton homology, one can therefore take the
        calculation of $I(Z_{n,-1})$ as a first step towards understanding
        the reduced instanton homology with local coefficients for torus knots
        in $S^{3}$. The authors hope to return to this in a future
        paper.

        \subsection{Universal relations}

        The relations in the instanton homology of $Z_{n}$ and
        $Z_{n,-1}$ give rise to universal relations for general
        admissible bifolds $(Y,K)$ containing spheres. The following
        is an illustration.

       \begin{proposition}\label{prop:universalJ}
        Let $(Y,K)$ be a bifold and suppose that the singular set $K$
        is a knot meeting an embedded sphere $S$ in $Y$ transversely
        with odd geometric intersection number $n$ and algebraic
        intersection number $n'$. Orient the sphere and $K$ so that $0<n'\le
        n$. Let $\alpha$ be the operator on $I(Y,K)$ corresponding
        the sphere $S$ and let $\delta$ be the operator arising from a
        point on $K$. Let $\epsilon$ be the involution on $I(Y,K)$
        arising from $S$. Then the elements of the ideal 
        \[
           (\tau^{4}-1)^{(n-n')/2}\barJOne{n'}
        \subset \cR[\delta,\alpha,\epsilon]/\langle
        \epsilon^{2}-1\rangle\]
        annihilate $I(Y,K)$. 
       \end{proposition}

       \begin{proof}
       Let
       $\delta_{1},\dots,\delta_{n}$ be the operators corresponding
       the intersection points of $K$ with $S$, all oriented with the
       normal orientation to $S$. From an application of the general
       principle of Proposition~\ref{prop:universalW}, the instanton
       homology $I(Y,K)$ is annihilated by the ideal $\cJ_{n}$ in the
       algebra $\cA_{n}$. On the other hand, because $K$ is a knot,
       all the operators $\delta_{i}$ are equal up to sign, so the
       action of the algebra $\cA_{n}$ factors through the quotient
       $\bar\cA=\cR[\delta,\alpha,\epsilon]/\langle
       \epsilon^{2}-1\rangle$
       in which we set $\delta_{i}=\pm \delta$ according to the sign
       of the corresponding intersection point of $K$ with $S$. From
       Corollary~\ref{cor:recursive-1} and 
       Corollary~\ref{cor:recursive} the image of $\cJ_{n}$ in the
       quotient contains the ideal described in the Proposition.
       \end{proof}

       As a simplest example, if $K$ is a knot in $Y=S^{1}\times S^{2}$
       which has geometric intersection $3$ and algebraic
       intersection $1$ with $S^{2}$, then $I(Y,K)$ is a torsion
       $\cR$-module annihilated by $\tau^{4}-1$. In general, the
       proposition provides a lower bound on the geometric
       intersection number of $K$ and $S^{2}$. 

       \begin{corollary}
        Let $Y$ contain an oriented $2$-sphere $S$, and let $K\subset
        Y$ be a knot having odd algebraic
        intersection number $n' > 0$ with $S$. Then a lower bound for
        the transverse geometric intersection number $K\cap S$
        for any knot isotopic to $K$ is
        $n' + 2f$, where
        \[
                f = \min \{ \, F\ge 0 \mid
                          \text{\textup{$(\tau^{4}-1)^{F} \Gen{n'}{i}$
                          annihilates $I(Y, K)$ for
                          $i=1,2$}}\,\} 
        \]
        and $\Gen{n'}{1}$ and $\Gen{n'}{2}$ are the two generators in
        Proposition~\ref{prop:two-minors}. \qed
       \end{corollary}

       In light of the results from \cite{Xie-Zhang} concerning
       higher-genus orbifolds, it is possible that the bound $n'+2f$
       defined in the Corollary is not particularly strong. It may be
       that $n'+2f$ is a lower bound for $n_{g} + 2g$, where $n_{g}$
       is the geometric intersection number with a surface $S_{g}$
       of genus $g$ homologous to $S$. It is easy to visualize
       examples where $n_{1}+2$ is much smaller than $n_{0}$, for
       example.

       In the case that $n=n'$ in Proposition~\ref{prop:universalJ}
       (i.e~when algebraic and geometric intersection numbers are
       equal), the $\bar\cA$-module $I(Y,K)$ is annihilated by the defining ideal
       of the curve $D_{n}$. This means that we can interpret $I(Y,K)$
       as a coherent sheaf on $D_{n}$.

       \subsection{The degrees of the relations}

          \begin{figure}
           \begin{center}
       \includegraphics[width=6cm]{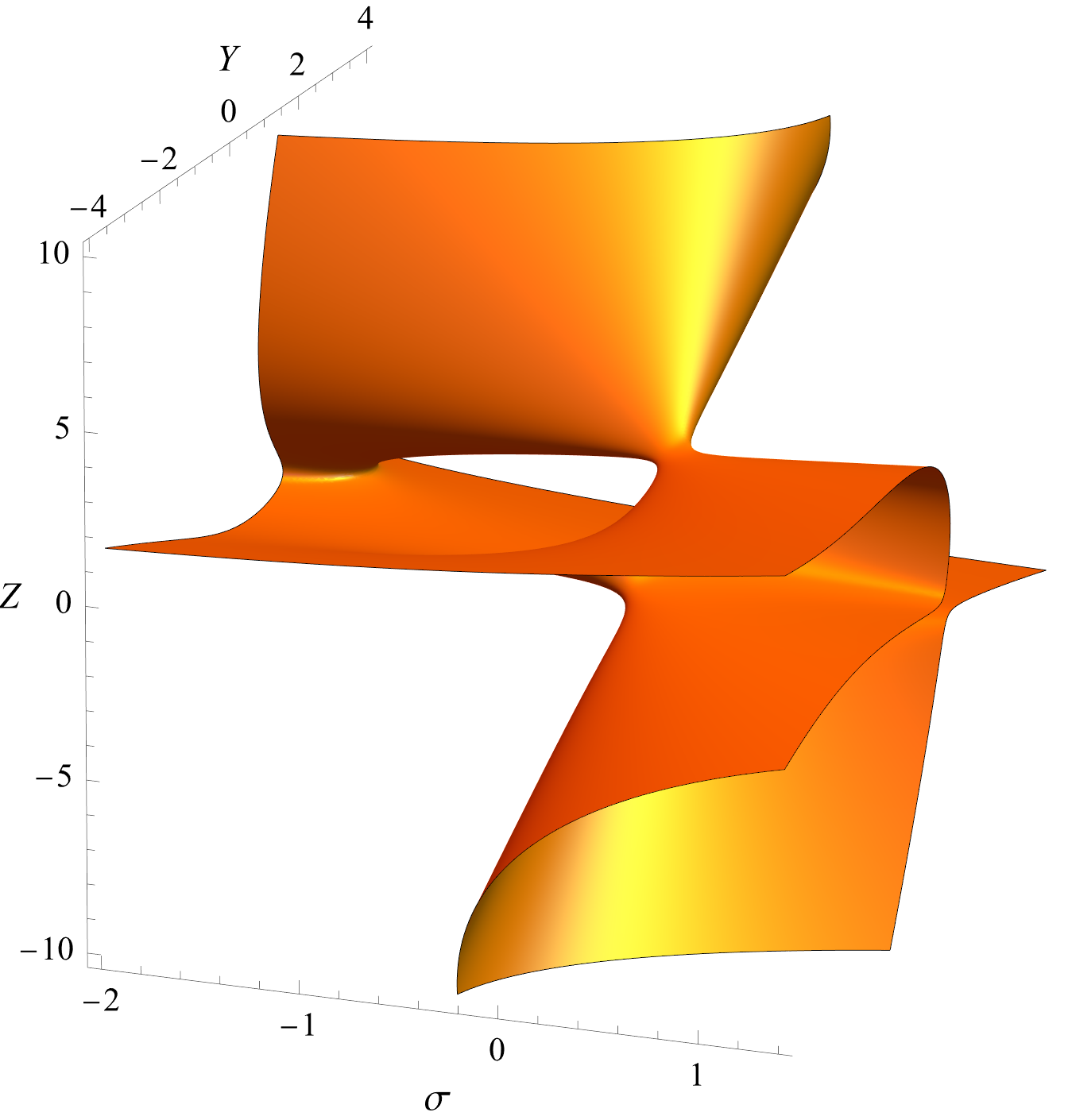}
        \includegraphics[width=6cm]{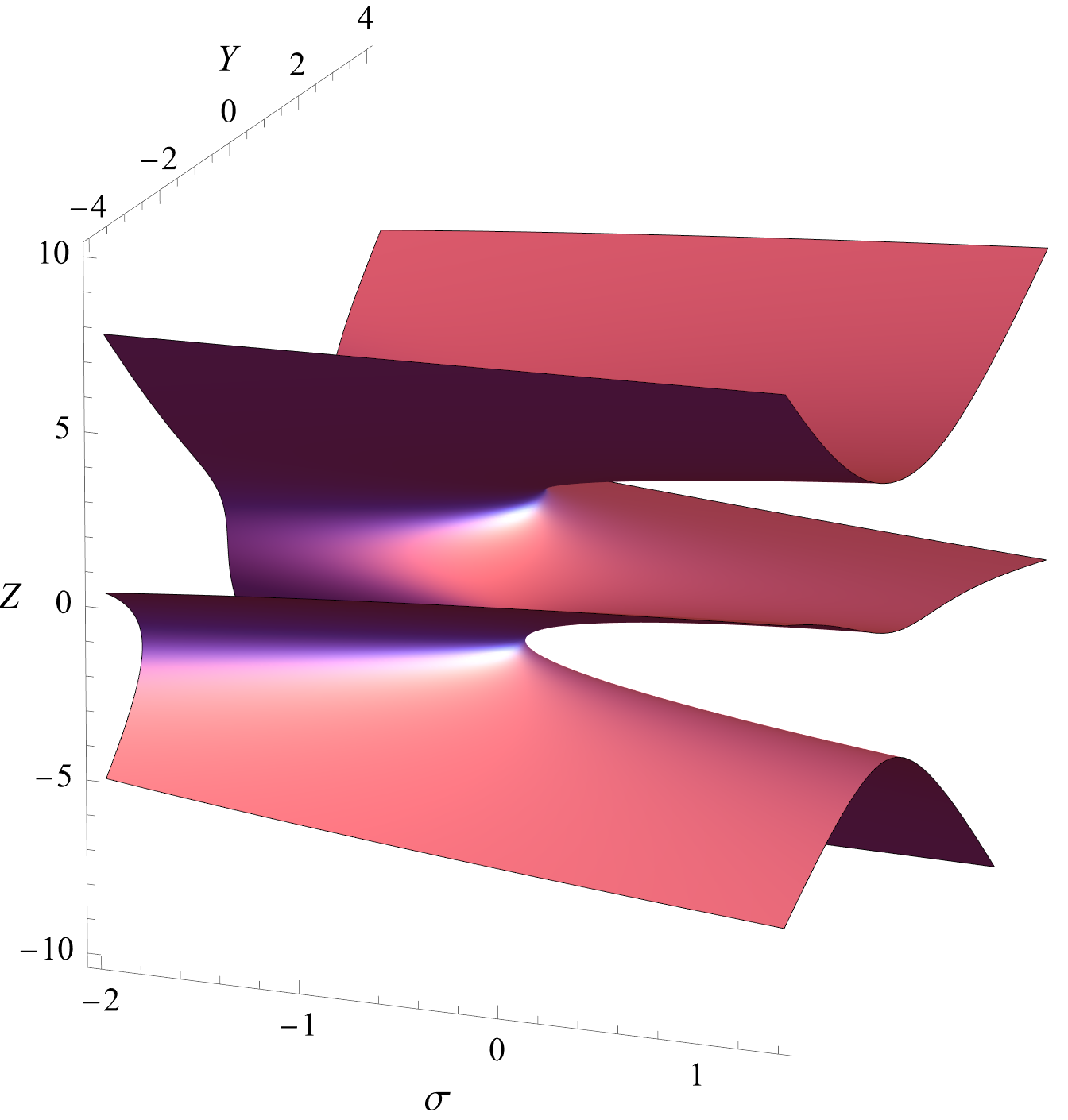}
    \end{center}
    \caption{\label{fig:surfaces}The real loci defined by the
    vanishing of the generators $G_{1}(n)$ (left) and $G_{2}(n)$
    (right) for $n=7$ in the coordinates $(\sigma,Y,Z)$. Only the
    part with $\epsilon=1$ is shown. The part with $\epsilon=-1$ is
    obtained by changing the sign of $Y$ and $Z$. These are smooth affine
    cubic surfaces.}
    \end{figure}
       
       The two generators $G_{1}(n)$, $G_{2}(n)$ for the ideal of
       relations for $I(Z_{n,-1})$ both have total degree $m=(n-1)/2$
       in $(\alpha,\delta)$ but larger degree in $\tau$. However, a
       substitution simplifies the polynomials a little: if we
       substitute
       \[
            \begin{gathered}
                Z = \tau \alpha \\
                Y = \tau \delta
            \end{gathered}
       \]
       then (after clearing unnecessary powers of $\tau$ from the
       denominator) we obtain a polynomial in $Z$, $Y$ and $\tau^{4}$.
       Writing $\sigma=\tau^{4}$, the total degree of the generators
       $G_{i}(n)$ in $(\sigma, Z, Y)$ is $m$. The real loci defined by
       the vanishing of these two polynomials in $(\sigma, Y, Z)$ are
       shown in Figure~\ref{fig:surfaces} for $n=7$.

       \bibliographystyle{abbrv}
\bibliography{itk}

 \end{document}